\documentclass[12pt,a4paper]{amsart}
\usepackage[T1]{fontenc}
\usepackage[left=25mm,top=35mm,right=25mm,bottom=26mm,headsep=10mm]{geometry}
\usepackage{microtype,textcomp,fbb}
\usepackage{
    amsmath,  amssymb,  amsthm,   amscd,
    gensymb,  graphicx, etoolbox, 
    booktabs, stackrel, mathtools    
}
\usepackage[usenames,dvipsnames]{xcolor}
\usepackage[colorlinks=true, linkcolor=blue, citecolor=blue, urlcolor=blue, breaklinks=true]{hyperref}
\usepackage[capitalise]{cleveref}
\usepackage{enumitem}
\newtheorem{theorem}{Theorem}[section]

\newtheorem{definition}[theorem]{Definition}
\newtheorem{example}[theorem]{Example}
\newtheorem{lemma}[theorem]{Lemma}

\newtheorem{proposition}[theorem]{Proposition}
\newtheorem{remark}[theorem]{Remark}

\usepackage{tikz}
\usetikzlibrary{arrows}
\usetikzlibrary{shapes}
\usetikzlibrary{positioning}
\usetikzlibrary{decorations.pathreplacing}
\usetikzlibrary{calc}
\usepackage{float}
\usepackage{tabularx}
\usepackage{kantlipsum}
\tikzset{edgee/.style = {> = latex'}}
\usepackage{array}
\newcolumntype{P}[1]{>{\centering\arraybackslash}p{#1}}
\newcolumntype{M}[1]{>{\centering\arraybackslash}m{#1}}
\tikzset{%
  add/.style args={#1 and #2}{to path={%
 ($(\tikztostart)!-#1!(\tikztotarget)$)--($(\tikztotarget)!-#2!(\tikztostart)$)%
  \tikztonodes}}
} 

\newcommand\R{\mathbb{R}}
\newcommand\Z{\mathbb{Z}}
\newcommand\N{\mathbb{N}}
\newcommand{\A}{\mathcal{A}}
\newcommand{\B}{\mathcal{B}}
\newcommand{\C}{\mathcal{C}}
\newcommand{\BC}{\mathcal{BC}}
\newcommand{\D}{\mathcal{D}}
\newcommand{\Po}{\mathcal{P}}
\newcommand{\E}{\mathcal{E}}

\newcommand{\CT}{\mathcal{CT}}
\newcommand{\ST}{\mathcal{ST}}
\newcommand{\ipa}{\mathrm{L}}
\newcommand{\lt}[2]{\ensuremath{\alpha_{#1}^{(#2)}}}

\title[Sketches, moves and partitions]{Sketches, moves and partitions: counting regions of deformations of reflection arrangements}
\author{Priyavrat Deshpande}
\address{Department of Mathematics, Chennai Mathematical Institute, India}
\email{pdeshpande@cmi.ac.in}

\author{Krishna Menon}
\address{Department of Mathematics, Chennai Mathematical Institute, India}
\email{krishnamenon@cmi.ac.in}
\date{}
\begin{document}

\begin{abstract}
    The collection of reflecting hyperplanes of a finite Coxeter group is called a reflection arrangement and it appears in many subareas of combinatorics and representation theory. 
    We focus on the problem of counting regions of reflection arrangements and their deformations. 
    Inspired by the recent work of Bernardi, we show that the notion of moves and sketches can be used to provide a uniform and explicit bijection between regions of (the Catalan deformation of) a reflection arrangement and certain non-nesting partitions. 
%    \textcolor{red}{This seems to imply that we've solved it for all deformations. How about adding a line like this here: We apply this method to study the Catalan deformations of reflection arrangements.} 
    We then use the exponential formula to describe a statistic on these partitions such that distribution is given by the coefficients of the characteristic polynomial. 
    Finally, we consider a sub-arrangement of type C arrangement called the threshold arrangement and its Catalan and Shi deformations. 
    %This is not a reflection arrangement but it exhibits similar combinatorial features. 
\end{abstract}

\maketitle

%\textcolor{red}{
%So as long as we can choose canonical sketches in such a way that the analogue of \eqref{regbreakup} holds for them as well, we get results like the branch statistic one.
%}

%\textcolor{red}{K: Should we do Type A as a subarrangement of Type C as an example before Type D, B, BC?} Yes.

%\textcolor{red}{K: Should we try and obtain properties of the coefficients as well?}

%\textcolor{red}{K: \eqref{regbreakup} also shows that we only have to find a description for the $\alpha,\beta$-words for bounded regions in general Catalan.}

%\textcolor{blue}{P: Prpopsed restructuring - Sec 1 Intro, Sec 2 overview of Bernardi's paper, Sec 3 Sketches and moves for other types, Sec 4 Catalan deformation and its subarrangements, Sec 5 GF and statistics, Sec 6 the case of threshold deformations, Sec 7 concluding remarks}

\section{Introduction}\label{intro}

A \emph{hyperplane arrangement} $\A$ is a finite collection of affine hyperplanes (i.e., codimension $1$ subspaces and their translates) in $\R^n$. 
% Without loss of generality we assume that arrangements we consider are \emph{essential}, i.e., the subspace spanned by the normal vectors is the ambient vector space. 
A \emph{flat} of $\A$ is a nonempty intersection of some of the hyperplanes in $\A$; the ambient vector space is a flat since it is an intersection of no hyperplanes. 
Flats are naturally ordered by reverse set inclusion; the resulting poset is called the \emph{intersection poset} and is denoted by $\ipa(\A)$. 
The \emph{rank} of $\A$ is the dimension of the span of the normal vectors to the hyperplanes. 
An arrangement in $\R^n$ is called \emph{essential} if its rank is $n$. 
A \emph{region} of $\A$ is a connected component of $\R^n\setminus \bigcup \A$. 
A region is said to be \emph{bounded} if its intersection with the subspace spanned by the normal vectors to the hyperplanes is bounded. 
Counting the number of regions of arrangements using diverse combinatorial methods is an active area of research.

The \emph{characteristic polynomial} of $\A$ is defined as 
$\displaystyle \chi_{\A}(t) := \sum \mu(\hat{0},x)\, t^{\dim(x)}$
where $x$ runs over all flats in $\ipa(\A)$,  $\mu$ is its the M\"obius function and $\hat{0}$ corresponds to the flat $\R^n$. 
Using the fact that every interval of the intersection poset of an arrangement is a geometric lattice, we have
\begin{equation}\label{charform}
    \chi_\A(t) = \sum_{i=0}^n (-1)^{n-i} c_i t^i
\end{equation}
where $c_i$ is a non-negative integer for all $0 \leq i \leq n$ \cite[Corollary 3.4]{sta_hyp}. 
The characteristic polynomial is a fundamental combinatorial and topological invariant of the arrangement and plays a significant role throughout the theory of hyperplane arrangements.

In this article, our focus is on the enumerative aspects of (rational) arrangements in $\R^n$. 
In that direction we have the following seminal result by Zaslavsky.

% \begin{theorem}[\cite{zas75}]\label{zaslavsky}
% Let $\A$ be an arrangement in $\R^n$. Then the number of regions of $\A$ is given by 
% \begin{align*}
%   r(\A) &= (-1)^n \chi_{\A}(-1) \\
%          &= \sum_{i=0}^n c_i. 
% \end{align*}
% \end{theorem}
\begin{theorem}[\cite{zas}] \label{zaslavsky}
Let $\A$ be an arrangement in $\R^n$. Then the number of regions of $\A$ is given by 
\[ r(\A) = (-1)^n \chi_{\A}(-1) = \sum_{i=0}^n c_i\]
and the number of bounded regions is given by 
\[b(\A) = (-1)^{\operatorname{rank}(A)} \chi_{\A}(1). \]
\end{theorem}

The finite field method,  developed by Athanasiadis \cite{ath_adv}, converts the computation of the characteristic polynomial to a point counting problem. 
A combination of these two results allowed for the computation of the number of regions of several arrangements of interest. 

Another way to count the number of regions is to give a bijective proof. 
This approach involves finding a combinatorially defined set whose elements are in bijection with the regions of the given arrangement and are easier to count. 
For example, the \textit{braid arrangement} in $\R^n$ is given by
\begin{equation*}
    \{x_i-x_j=0 \mid 1 \leq i < j \leq n\}.
\end{equation*}
It is straightforward to verify that its regions correspond to the permutations of $[n]$. 
Hence the number of regions of the braid arrangement in $\R^n$ is $n!$.

The Catalan arrangement of type $A$ in $\R^n$ is given by
\begin{equation*}
    \A_n = \{x_i - x_j = -1, 0, 1 \mid 1 \leq i < j \leq n\}.
\end{equation*}
This arrangement and its sub-arrangements have been studied in great detail (for example, see \cite{ber}). 
It is well-known that the number of regions of $\A_n$ where $x_1 < x_2 < \cdots < x_n$ (also known as the dominant regions) is given by the Catalan number $\frac{1}{n + 1}\binom{2n}{n}$. 

Using this, it is easy to see that \[r(\A_n)=\frac{n!}{n + 1}\binom{2n}{n}.\]

Let $\Phi$ be a (not necessarily reduced) crystallographic root system and let $\Phi^+$ be a choice of positive roots. 
The reflection (or Coxeter) arrangement $\A(\Phi)$ corresponding to $\Phi$ consists of hyperplanes with the defining equations 
\[(\alpha, x) = 0\quad \hbox{for~} \alpha\in\Phi^+.  \]
Note that these are the same hyperplanes that are fixed by the Weyl group of $\Phi$. 
A deformation of a reflection arrangement is an arrangement each of whose hyperplanes is parallel to some hyperplane in $\A(\Phi)$. 
Our main focus in the present paper is the Catalan deformation; for brevity we sometimes write Catalan arrangement of type $\Phi$.
%We have already defined Catalan arrangement of type $A$ above. 
The defining equations of Catalan arrangements of other types are as follows: 
\begin{itemize}
    \item The Catalan arrangement of type $B$ in $\R^n$ is given by
    \begin{equation*}
        \{x_i = -1, 0, 1 \mid i \in [n]\} \cup \{x_i + x_j = -1, 0, 1 \mid 1 \leq i < j \leq n\} \cup \A_n.
    \end{equation*}
    
    \item The Catalan arrangement of type $C$ in $\R^n$ is given by
    \begin{equation*}
        \{2x_i = -1, 0, 1 \mid i \in [n]\} \cup \{x_i + x_j = -1, 0, 1 \mid 1 \leq i < j \leq n\} \cup \A_n.
    \end{equation*}
    
    \item The Catalan arrangement of type $D$ in $\R^n$ is given by
    \begin{equation*}
        \{x_i + x_j = -1, 0, 1 \mid 1 \leq i < j \leq n\} \cup \A_n.
    \end{equation*}
    
    \item The Catalan arrangement of type $BC$ in $\R^n$ (defined in \cite{ath_lin}) is the union of the type $B$ and type $C$ Catalan arrangements in $\R^n$.
\end{itemize}

In addition to these, we also consider `type $C$ extended Catalan arrangements'; consisting of hyperplanes of the form $(\alpha, x) = k$ for $k = -m, \dots, m$ for a fixed integer $m\geq 1$.  
The characteristic polynomials, and hence the number of regions, of these arrangements are known (for example, see \cite{ath_lin}). 
We provide bijective proofs for the number of regions as well as bounded regions of these arrangements. 
Bijective proofs for the number of regions of the type $C$ Catalan arrangement have already been established in \cite{non_bij} and \cite{typc}. 
However, the proofs we present for the other arrangements seem to be new.

The idea used for the bijections is fairly simple but effective. 
This was used by Bernardi in \cite[Section 8]{ber} to obtain bijections for the regions of several deformations of the braid arrangement. 
This idea, that we call `sketches and moves', is to consider an arrangement $\B$ whose regions we wish to count as a sub-arrangement of an arrangement $\A$. 
This is done in such a way that the regions of $\A$ are well-understood and are usually total orders on certain symbols. 
These total orders are what we call \emph{sketches}. 
Since $\B \subseteq \A$, the regions of $\B$ partition the regions of $\A$ and hence define an equivalence on sketches. 
We define operations called \emph{moves} on sketches to describe the equivalence classes. 
In regions of $\A$, moves correspond to crossing hyperplanes in $\A \setminus \B$.

Apart from Bernardi's results, the results in \cite{ath_shi} and \cite{seo_cat} can also be viewed as applications of the sketches and moves idea to count regions of hyperplane arrangements.

When studying an arrangement, another interesting question is whether the coefficients of its characteristic polynomial can be combinatorially interpreted. 
By \Cref{zaslavsky}, we know that the sum of the absolute values of the coefficients is the number of regions. 
Hence, one could ask if there is a statistic on the regions whose distribution is given by the coefficients of the characteristic polynomial. 
The characteristic polynomial of the braid arrangement in $\R^n$ is 
$t(t-1)\cdots(t-n+1)$ \cite[Corollary 2.2]{sta_hyp}. 
Hence, the coefficients are the Stirling numbers of the first kind. 
Consequently, the distribution of the statistic `number of cycles' on the set of permutations of $[n]$ (which correspond to the regions of the arrangement) is given by the coefficients of the characteristic polynomial.

The paper is structured as follows: 
In \Cref{sketchmove}, we describe the sketches and moves idea mentioned above. 
We also use it to study the regions of some simple arrangements in \Cref{refarrsec}. 
In \Cref{typecsec}, we reprove the results in \cite{typc} about the type $C$ Catalan arrangement with a modification inspired by \cite{ath_non}. 
We then use the sketches and moves idea in \Cref{othertypessec} to obtain bijections for the regions of the Catalan arrangements of other types. 
In \Cref{statsec}, we describe statistics on the regions of the arrangements we have studied whose distribution is given by the corresponding characteristic polynomials. 
Finally, in \Cref{threshsec}, we use similar techniques to study an interesting arrangement called the threshold arrangement as well as some of its deformations.

\section{Sketches, moves and trees: a quick overview of Bernardi's bijection}\label{sketchmove}

In his paper \cite{ber}, Bernardi describes a method to count the regions of any deformation of the braid arrangement using certain objects called \textit{boxed trees}. 
He also obtains explicit bijections with certain trees for several deformations. 
The general strategy to establish the bijection is to consider an arrangement $\B$ whose regions we wish to count as a sub-arrangement of an arrangement $\A$ whose regions are well-understood. 
The regions of $\B$ then define an equivalence on the regions of $\A$. 
This is done by declaring two regions of $\A$ to be equivalent if they lie inside the same region of $\B$. 
Now counting the number of regions of $\B$ is the same as counting the number of equivalence classes of this equivalence on the regions of $\A$. 
This is usually done by choosing a canonical representative for each equivalence class, which also gives a bijection between the regions of $\B$ and certain regions of $\A$.

In particular, a (transitive) deformation of the braid arrangement is a sub-arrangement of the (extended or) $m$-Catalan arrangement (for some large $m$) in $\R^n$, whose hyperplanes are
\begin{equation*}
    \{x_i - x_j = k \mid 1 \leq i < j \leq n, k \in [-m, m]\}.
\end{equation*}
The regions of these arrangements are known to correspond labeled $(m + 1)$-ary trees with $n$ nodes (see \cite[Section 8.1]{ber}). 
Using the idea mentioned above, one can show that the regions a deformation correspond to certain trees. 
We should mention that while he obtains direct combinatorial arguments to describe this bijection for some transitive deformations (see \cite[Section 8.2]{ber}), the proof for the general bijection uses much stronger results (see \cite[Section 8.3]{ber}).

Coming back to the general strategy, which we aim to generalize in order to apply it to deformations of other types. 
It is clear that any two equivalent regions of $\A$ have to be on the same side of each hyperplane of $\B$. 
However, it turns out that this equivalence is the transitive closure of a simpler relation. 
This follows from the fact that one can reach a region in an arrangement from another by crossing exactly one hyperplane at a time with respect to which the regions lie on opposite sides. 
We now prove this result, for which we require the following definition.

\begin{figure}[H]
    \centering
    \begin{tikzpicture}[yscale=1.2]
        \node (a) at (0,2){};
        \node (b) at (-1,1){};
        \node (c) at (1,1){};
        \node (d) at (-2,0){};
        \node (e) at (2,0){};
        \node (f) at (-1,-1){};
        \node (g) at (1,-1){};
        \node (h) at (0,-2){};
        \node [circle,fill=blue,minimum size=0.1cm,inner sep=0pt] (i) at (-1,-1/3){};
        \node [circle,fill=blue,minimum size=0.1cm,inner sep=0pt] (j) at (-2,-2/3){};
        \node [circle,fill=blue,minimum size=0.1cm,inner sep=0pt] (k) at (-2.5,-1.75){};
        \node [circle,fill=blue,minimum size=0.1cm,inner sep=0pt] (l) at (-4,-1.45){};
        \draw [thick][add=0.5 and 1](a) to (e);
        \draw [thick][add=0.75 and 0.75](g) to (b);
        \draw [thick][add=0.75 and 0.75](c) to (f);
        \draw [thick][add=0.5 and 0.5](e) to (d);
        \draw [dashed][add=1 and 0.5](d) to (h);
        \draw [dashed][add=0.5 and 1](a) to (d);
        \draw [thick][add=0.5 and 1](h) to (e);
        \draw [thick][add=1.5 and 1.5](b) to (c);
        \draw [dashed][add=1.75 and 1.5](f) to (g);
        \draw [dashed,blue] (i)--(j)--(k)--(l);
    \end{tikzpicture}
    \caption{Bold lines form $\B$ and the dotted lines form $\A \setminus \B$.
    Equivalent $\A$ regions can be connected by changing one $\A \setminus \B$ inequality at a time.}
\end{figure}
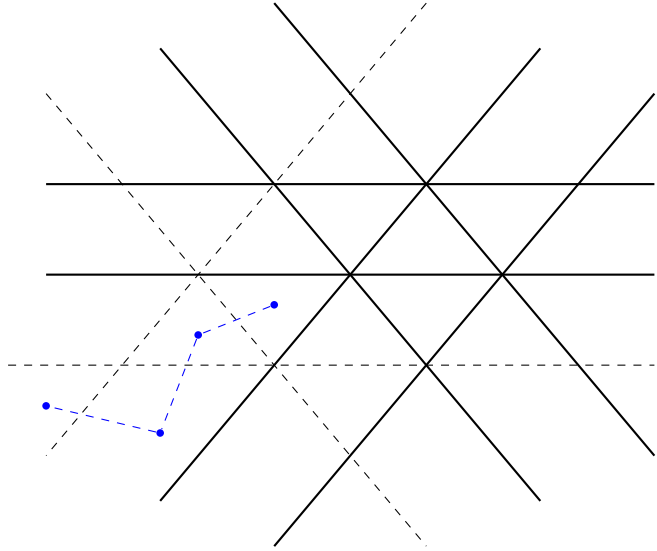

\begin{definition}
Let $R$ be a region of an arrangement $\A$. 
A \emph{determining set} of $R$ is a sub-arrangement $\D \subseteq \A$ such that the region of the arrangement $\D$ containing $R$, denoted $R_{\D}$, is equal to $R$.
\end{definition}

Note that a region of $\A$ always has the entire arrangement $\A$ as a determining set. 
Also, if a region $R'$ is on the same side as a region $R$ for each hyperplane in a determining set of $R$, then we must have $R = R'$.

Before going forward, we explicitly describe regions of an arrangement. 
First note that any hyperplane $H$ in $\R^n$ is a set of the form
\begin{equation*}
    \{\boldsymbol{x} \in \R^n \mid P_H(\boldsymbol{x}) = 0\}
\end{equation*}
where $P_H(\boldsymbol{x}) = a_1x_1 + a_2x_2 + \cdots + a_nx_n + c$ for some constants $a_1, \ldots, a_n, c \in \R$. 
Also, the regions of an arrangement $\A$ are precisely the non-empty intersections of sets of the form
\begin{equation*}
    \{\boldsymbol{x} \in \R^n \mid P_H(\boldsymbol{x}) > 0\}\text{ or }\{\boldsymbol{x} \in \R^n \mid P_H(\boldsymbol{x}) < 0\}
\end{equation*}
where we have one set for each $H \in \A$. 
Hence, crossing exactly one hyperplane $H$ in an arrangement corresponds to changing the inequality chosen for $H$ in this description of the region.

\begin{theorem}\label{thmcohaat}
If $\D$ is a minimal determining set of a region $R$ of an arrangement $\A$, then changing the inequality in the definition of $R$ of exactly one $H \in \D$, and keeping all other inequalities of hyperplanes in $\A$ the same, describes a non-empty region of $\A$.
\end{theorem}

Before proving this, we will see how it proves the fact mentioned above.
Start with two distinct regions $R$ and $R'$ of an arrangement $\A$.
We want to get from $R$ to $R'$ by crossing exactly one hyperplane at a time with respect to which the regions lie on opposite sides.
%We want to get from $R$ to $R'$ by changing one inequality by \todo{typo?} which $R$ and $R'$ differ at a time.

\begin{enumerate}
\item Let $\D$ be a minimal determining set of $R$.
\item Since $R \neq R'$ there is some $H \in \D$ for which $R'$ is on the opposite side as $R$.
\item Change the inequality corresponding to $H$ in $R$, call this new region $R''$.
\item The number of hyperplanes in $\A$ for which $R''$ and $R'$ lie on opposite sides is less than that for $R$ and $R'$.
\item Repeat this process to get to $R'$ by changing one inequality at a time.
\end{enumerate}

\begin{proof}[Proof of \Cref{thmcohaat}]
Let $H \in \D$.
Since $\D$ is a minimal determining set, $\E = \D \setminus \{H\}$ is not a determining set.
So $R$ is strictly contained in $R_{\E}$.
This means that the hyperplane $H$ intersects $R_{\E}$ and splits it into two open convex sets, one of which is $R$.
% Note that no other hyperplane of $\A$ intersects the portion of $H$ in $R_{\E}$ since this hyperplane would then split $R_\D$, which contradicts the fact that $R=R_\D$.
% This can be shown as follows:
% Suppose $H'$ intersects $H$ in $R_\E$.
% Consider the set of `diameters' of an $n$-ball centered at a point that the hyperplane $H'$ intersects $H$ in $R_{\E}$.
% Taking the radius of the ball to be small enough and using the fact that at least one of these diameters have to lie on the hyperplane $H'$, we get that $H'$ intersects $R_{\D}$.
% This will mean that $H'$ splits $R_\D$.

So we can choose a point $p \in H$ that lies inside $R_{\E}$ and an $n$-ball centered at $p$ that does not touch any other hyperplanes of $\A$ (since $\A$ is finite).
One half of the ball lies in $R$ and the other half lies in a region $R'$ of $\A$.
Since $R'$ can be reached from $R$ by just crossing the hyperplane $H$, we get the required result.
\end{proof}
% \begin{proof}
%     By the above remark, two Shi equivalent regions lie on the same side of $x_i-x_j=s$ hyperplanes for $i<j$ and $s=0,1$. Hence they lie in the same Shi region. Conversely, let two Catalan regions $R_1,R_2$ lie in the same Shi region. If $R_1=R_2$, we are done. If not, then $R_2$ lies on the opposite side as $R_1$ of at least one of the bounding hyperplanes of $R_1$ (bounding hyperplane is one such that just changing its inequality and keeping the others the same results in a nonempty region). Using this and induction, we can get that $R_2$ can be obtained from $R_1$ by changing exactly one inequality of the type $x_i-x_j>-1$ or $x_i-x_j<-1$ at a time (where $i<j$). This corresponds to a series of Shi moves. Hence $R_1$ and $R_2$ are Shi equivalent.
% \end{proof}

To sum up, we start with an arrangement $\B \subseteq \A$. 
We know the regions of $\A$ and usually represent them by combinatorial objects we call `sketches'. 
We then define `moves' on these sketches that correspond to changing exactly one inequality of a hyperplane in $\A \setminus \B$. 
We define sketches to be equivalent if one can be obtained from another through a series of moves. 
We then count the number of equivalence classes to obtain the number of regions of $\B$.
Before using this method to study the Catalan arrangements of various types, we first look at some simpler arrangements.

\section{Counting regions of reflection arrangements}\label{refarrsec}
In this section, as a warmup exercise, we illustrate the `sketches-moves' idea to study sub-arrangements of the type $C$ arrangement.
Hence, in the spirit of Bernardi \cite{ber}, we will define certain sketches corresponding to the region of the type $C$ arrangement and for any sub-arrangement, we choose a canonical sketch from each region.

\subsection{The type C arrangement}\label{typeC}
This arrangement in $\mathbb{R}^n$ is the set of reflecting hyperplanes of the root system $C_n$.
%Relevant definitions can be found in \cite{MR1066460}.
The defining equations of hyperplanes are
\begin{align*}
    2x_i&=0\\
    x_i+x_j&=0\\
    x_i-x_j&=0
\end{align*}
for $1 \leq i < j \leq n$.
Though we could write $x_i=0$ for the first type of hyperplanes, we think of them as $x_i+x_i=0$ to define sketches.
% Also, when we consider deformations, we will consider hyperplanes $2x_i=s$ for some integer $s$ and not all such hyperplanes are of the form $x_i=s'$ for some integer $s'$.

We can write the hyperplanes of the type $C$ arrangement as follows:
\begin{align*}
    x_i&=x_j, \hspace{0.75cm} 1 \leq i < j \leq n\\
    x_i&=-x_j, \quad i,j \in [n].
\end{align*}
Hence, any region of the arrangement is given by a \emph{valid} total order on
\begin{equation*}
    x_1, \ldots , x_n, -x_1, \ldots , -x_n.
\end{equation*}
A total order is said to be valid if there is some point in $\R^n$ that satisfies it. 
We will represent $x_i$ by $\overset{+}{i}$ and $-x_i$ by $\overset{-}{i}$ for all $i \in [n]$.

\begin{example}
The region $-x_2 < x_3 < x_1 < -x_1 < -x_3 < x_2$ is represented as $\overset{-}{2}\ \overset{+}{3}\ \overset{+}{1}\ \overset{-}{1}\ \overset{-}{3}\ \overset{+}{2}$.
\end{example}

It can be shown that words of the form
% \begin{equation*}
%     i_1\quad i_2 \quad  \dots\quad  i_n\quad  -i_n\quad  \dots \quad -i_2 \quad -i_1
% \end{equation*}
\begin{equation*}
    \overset{w_1}{i_1}\quad \overset{w_2}{i_2} \quad  \cdots\quad  \overset{w_n}{i_n}\quad  \overset{-w_n}{i_n}\quad  \cdots \quad \overset{-w_2}{i_2} \quad \overset{-w_1}{i_1}
\end{equation*}
where $\{i_1,\ldots,i_n\}=[n]$ are the ones that correspond to regions.
Such orders are the only ones that can correspond to regions since negatives reverse order.
Also, choosing $n$ distinct negative numbers, it is easy to construct a point satisfying the inequalities specified by such a word.
Hence the number of regions of the type $C$ arrangement is $2^nn!$.
We will call such words \emph{sketches} (which are basically signed permutations).
We will draw a line after the first $n$ symbols to denote the reflection and call the part of the sketch before the line its first half and similarly define the second half.

\begin{example}
$\overset{+}{3}\ \overset{-}{1}\ \overset{-}{2}\ \overset{+}{4}\ \textcolor{blue}{|}\ \overset{-}{4}\ \overset{+}{2}\ \overset{+}{1}\ \overset{-}{3}$ is a sketch.
\end{example}

We now study some sub-arrangements of the type $C$ arrangement. 
For each such arrangement, we will define the moves that we can apply to the sketches (which represent changing exactly one inequality corresponding to a hyperplane not in the arrangement) and then choose a canonical representative from each equivalence class. 
By \Cref{thmcohaat}, this gives a bijection between these canonical sketches and the regions of the sub-arrangement.

\subsection{The Boolean arrangement}\label{bool}

One of first examples one encounters when studying hyperplane arrangements is the Boolean arrangement.
The Boolean arrangement in $\mathbb{R}^n$ has hyperplanes $x_i = 0$ for all $i \in [n]$.
It is fairly straightforward to see that the number of regions is $2^n$.
We will do this using the idea of moves on sketches.

The hyperplanes missing from the type $C$ arrangement in the Boolean arrangement are
\begin{align*}
    x_i+x_j&=0\\
    x_i-x_j&=0
\end{align*}
for $1 \leq i < j \leq n$. 
Hence, the Boolean moves are as follows:
\begin{enumerate}
    \item Swapping adjacent $\overset{+}{i}$ and $\overset{-}{j}$ as well as $\overset{+}{j}$ and $\overset{-}{i}$ for distinct $i, j \in [n]$.
    \item Swapping adjacent $\overset{+}{i}$ and $\overset{+}{j}$ as well as $\overset{-}{j}$ and $\overset{-}{i}$ for distinct $i, j \in [n]$.
\end{enumerate}

The first kind of move corresponds to changing inequality corresponding to the hyperplane $x_i + x_j = 0$ and keeping all the other inequalities the same. 
Similarly, the second kind of move corresponds to changing only the inequality corresponding to $x_i - x_j = 0$.

\begin{example}
We can use a series of Boolean moves on a sketch as follows:
\begin{equation*}
    \overset{-}{4}\ \overset{+}{1}\ \overset{+}{2}\ \overset{-}{3}\ \textcolor{blue}{|}\ \overset{+}{3}\ \overset{-}{2}\ \overset{-}{1}\ \overset{+}{4}\ \longrightarrow \ \overset{-}{4}\ \overset{+}{2}\ \overset{+}{1}\ \overset{-}{3}\ \textcolor{blue}{|}\ \overset{+}{3}\ \overset{-}{1}\ \overset{-}{2}\ \overset{+}{4}\  \longrightarrow \ \overset{-}{4}\ \overset{+}{2}\ \overset{-}{3}\ \overset{+}{1}\ \textcolor{blue}{|}\ \overset{-}{1}\ \overset{+}{3}\ \overset{-}{2}\ \overset{+}{4}\ \longrightarrow \ \overset{-}{4}\ \overset{-}{3}\ \overset{+}{2}\ \overset{+}{1}\ \textcolor{blue}{|}\ \overset{-}{1}\ \overset{-}{2}\ \overset{+}{3}\ \overset{+}{4}
\end{equation*}
\end{example}

It can be shown that for any sketch, we can use Boolean moves to convert it to a sketch where the order of absolute values in the second half is $1, 2, \ldots, n$ (since adjacent transpositions generate the symmetric group). 
Also, since the signs of the numbers in the second half do not change there is exactly one such sketch in each equivalence class. 
Hence the number of Boolean regions is the number of ways of assigning signs to the numbers $1, 2, \ldots, n$ which is $2^n$.

\subsection{The type D arrangement}\label{typeD}

The type $D$ arrangement in $\mathbb{R}^n$ has the hyperplanes
\begin{align*}
    x_i+x_j&=0\\
    x_i-x_j&=0
\end{align*}
for $1 \leq i < j \leq n$.
The hyperplanes missing from missing from the type $C$ arrangement are
\begin{equation*}
    2x_i=0
\end{equation*}
for all $i \in [n]$.
Hence a type $D$ move, which we call a $D$ move, is swapping adjacent $\overset{+}{i}$ and $\overset{-}{i}$ for any $i \in [n]$.

\begin{example}
$\overset{+}{4}\ \overset{+}{1}\ \overset{-}{3}\ \overset{+}{2}\ \textcolor{blue}{|}\ \overset{-}{2}\ \overset{+}{3}\ \overset{-}{1}\ \overset{-}{4} \xrightarrow{D\ move} \overset{+}{4}\ \overset{+}{1}\ \overset{-}{3}\ \overset{-}{2}\ \textcolor{blue}{|}\ \overset{+}{2}\ \overset{+}{3}\ \overset{-}{1}\ \overset{-}{4}$
\end{example}

In a sketch the only such pair is the last term of the first half and the first term of the second half.
Hence $D$ moves actually define an involution on the sketches.
Hence the number of regions of the type $D$ arrangement is $2^{n-1}n!$.
We could also choose a canonical sketch in each type $D$ region to be the one where the first term of the second half is positive.

\subsection{The braid arrangement}\label{braid}

The braid arrangement in $\R^n$ has hyperplanes
\begin{equation*}
    x_i - x_j = 0
\end{equation*}
for $1 \leq i < j \leq n$. 
The hyperplanes missing from the type $C$ arrangement are
\begin{align*}
    2x_i &= 0\\
    x_i + x_j &= 0
\end{align*}
for all $1 \leq i < j \leq n$. 
Hence the braid moves are as follows:
\begin{enumerate}
    \item ($D$ move) Swapping adjacent $\overset{+}{i}$ and $\overset{-}{i}$ for any $i \in [n]$.
    
    \item Swapping adjacent $\overset{+}{i}$ and $\overset{-}{j}$ as well as $\overset{+}{j}$ and $\overset{-}{i}$ for distinct $i, j \in [n]$.
\end{enumerate}

\begin{example}
We can use a series of braid moves on a sketch as follows:

\begin{equation*}
    \overset{-}{4}\ \overset{-}{1}\ \overset{+}{3}\ \overset{+}{2}\ \textcolor{blue}{|}\ \overset{-}{2}\ \overset{-}{3}\ \overset{+}{1}\ \overset{+}{4} \longrightarrow \overset{-}{4}\ \overset{-}{1}\ \overset{+}{3}\ \overset{-}{2}\ \textcolor{blue}{|}\ \overset{+}{2}\ \overset{-}{3}\ \overset{+}{1}\ \overset{+}{4} \longrightarrow \overset{-}{4}\ \overset{-}{1}\ \overset{-}{2}\ \overset{+}{3}\ \textcolor{blue}{|}\ \overset{-}{3}\ \overset{+}{2}\ \overset{+}{1}\ \overset{+}{4} \longrightarrow \overset{-}{4}\ \overset{-}{1}\ \overset{-}{2}\ \overset{-}{3}\ \textcolor{blue}{|}\ \overset{+}{3}\ \overset{+}{2}\ \overset{+}{1}\ \overset{+}{4}
\end{equation*}
\end{example}

Any sketch is braid equivalent to one where the signs of all the numbers in the second half are positive. 
This can be proved using induction on the number of positive terms in the first half of the sketch. 
Find the rightmost term in the first half that is positive. 
Moves of the second type can be used to take it to the last position in the first half of the sketch. 
Then a $D$ move takes it to the second half of the sketch.

It can also be checked that braid moves do not change the order in which the positive terms appear in a sketch. 
This shows that there is a unique sketch in each braid equivalence class where all the terms in the second half are positive. 
Hence, the number of braid regions is the number of such sketches, which is $n!$.

\begin{remark}
The union of the braid and Boolean arrangements in $\R^n$ is just the essentialization of the braid arrangement in $\R^{n + 1}$. 
However, using the idea of moves on sketches, one can show that the regions of this arrangement are in bijection with sketches where the second half is of the form
\begin{equation*}
    \overset{+}{i_1}\ \overset{+}{i_2}\ \cdots\ \overset{+}{i_k}\ \overset{-}{i_{k + 1}}\ \cdots\ \overset{-}{i_n}
\end{equation*}
for some $k \in [0, n]$. 
This shows that the number of regions is $(n + 1)!$.
\end{remark}

We also study two other interesting sub-arrangements of the type $C$ arrangement in \Cref{threshfubsm}.

\section{Catalan deformation of type C}\label{typecsec}

In this section we reprove, with a modification inspired by \cite{ath_non}, the results of \cite{typc} about the regions of the type $C$ Catalan arrangements.

Fix $n \geq 1$ throughout this section. 
The type $C$ Catalan arrangement in $\mathbb{R}^n$ is the arrangement with hyperplanes
\begin{align*}
    2X_i&=-1,0,1\\
    X_i+X_j&=-1,0,1\\
    X_i-X_j&=-1,0,1
\end{align*}
for all $1 \leq i < j \leq n$. 
In this case, instead of looking at this arrangement directly, we will study the arrangement obtained by performing the translation $X_i=x_i+\frac{1}{2}$ for all $i \in [n]$. 
It is easy to see that this does not change the combinatorics of the arrangement. 
The translated arrangement, which we call $\C_n$, has hyperplanes
\begin{align}\label{Cnhyp}
    \begin{split}
    2x_i&=-2,-1,0\\
    x_i+x_j&=-2,-1,0\\
    x_i-x_j&=-1,0,1
    \end{split}
\end{align}
for all $1 \leq i < j \leq n$. 
The arrangement $\C_n$ consists of all hyperplanes of the form $x_i + s = \pm (x_j + t)$ for $i, j \in [n]$ and $s, t \in \{0, 1\}$. 
% The hyperplanes in this arrangement can be rewritten as
% \begin{align*}
%     &x_i+1=-x_i-1,\quad x_i+1=-x_i,\quad x_i=-x_i\\
%     &x_i+1=-x_j-1,\quad x_i+1=-x_j,\quad x_i=-x_j\\
%     &x_i+1=x_j,\quad x_i=x_j,\quad x_i=x_j+1
% \end{align*}
% for all $1 \leq i<j \leq n$. 
This shows that the regions of $\C_n$ are given by valid total orders on
\begin{equation*}
    \{x_i+s \mid i \in [n],\ s \in \{0,1\}\} \cup \{-x_i-s \mid i \in [n],\ s \in \{0,1\}\}.
\end{equation*}

Such orders will be represented by using the symbol \lt{i}{s} for $x_i+s$ and \lt{-i}{-s} for $-x_i-s$ for all $i \in [n]$ and $s \in \{0,1\}$.
Let $C(n)$ be the set
\begin{equation*}
    \{\lt{i}{s}\mid i \in [n],\ s \in \{0,1\}\} \cup \{\alpha_i^{(s)} \mid -i \in [n],\ s \in \{-1,0\}\}.
\end{equation*}
Hence, we use orders on the letters of $C(n)$ to represent regions of $\C_n$.

\begin{example}
The total order
\begin{equation*}
    x_1<-x_2-1<x_1+1<x_2<-x_2<-x_1-1<x_2+1<-x_1
\end{equation*}
is represented as $\alpha_{1}^{(0)} \ \alpha_{-2}^{(-1)} \ \alpha_{1}^{(1)} \ \alpha_{2}^{(0)} \ \alpha_{-2}^{(0)} \ \alpha_{-1}^{(-1)} \ \alpha_{2}^{(1)} \ \alpha_{-1}^{(0)}$.
\end{example}

Considering $-x_i$ as $x_{-i}$, the letter \lt{i}{s} represents $x_i+s$ for any $\lt{i}{s} \in C(n)$. 
For any $\lt{i}{s} \in C(n)$, we use $\overline{\lt{i}{s}}$ to represent the letter \lt{-i}{-s}, which we call the \emph{conjugate} of \lt{i}{s}.

\begin{definition}\label{symsk}
A \emph{symmetric sketch} is an order on the letters in $C(n)$ such that the following hold for any $\lt{i}{s}, \lt{j}{t} \in C(n)$:
\begin{enumerate}
    \item If $\alpha_i^{(s)}$ appears before $\alpha_j^{(t)}$, then $\overline{\alpha_j^{(t)}}$ appears before $\overline{\alpha_i^{(s)}}$.
    
    \item If $\alpha_i^{(s-1)}$ appears before $\alpha_j^{(t-1)}$, then $\alpha_i^{(s)}$ appears before $\alpha_j^{(t)}$.
    
    \item $\alpha_i^{(s-1)}$ appears before $\alpha_i^{(s)}$.
    % \item Each letter of $C(n)$ appears exactly once.
\end{enumerate}
\end{definition}

\begin{proposition}\label{pointconst}
An order on the letters of $C(n)$ corresponds to a region of $\C_n$ if and only if it is a symmetric sketch.
\end{proposition}

\begin{proof}
The idea of the proof is the same as that of \cite[Lemma 5.2]{ath_non}. 
It is clear that any order that corresponds to a region must satisfy the properties in \Cref{symsk} and hence be a symmetric sketch. 
For the converse, we show that there is a point in $\mathbb{R}^n$ satisfying the inequalities given by a symmetric sketch.

We prove this using induction on $n$, the case $n = 1$ being clear. 
Let $n \geq 2$ and $w$ be a symmetric sketch. 
Without loss of generality, we can assume that the first letter of $w$ is $\lt{n}{0}$. 
Deleting the letters with subscript $n$ and $-n$ from $w$ gives a symmetric sketch $w'$ in the letters $C(n - 1)$. 
Using the induction hypothesis, we can choose a point $\mathbf{x'}  \in \R^{n-1}$ satisfying the inequalities given by $w'$. 
Suppose the letter before $\lt{n}{1}$ in $w$ is $\lt{i}{s}$ and the letter after it is $\lt{j}{t}$. 
We choose $x_n \neq -1$ such that $x'_i + s < x_n + 1 < x'_j + t$ in such a way that $x_n + 1$ is also in the correct position with respect $0$ specified by $w$. 
This is possible since $\mathbf{x'}$ satisfies $w'$.

We show that $(x'_1, \ldots, x'_{n-1}, x_n)$ satisfies the inequalities given by $w$. 
We only have to check that $x_n$ and $(x_n + 1)$ are in the correct relative position with respect to the other letters since property (1) of \Cref{symsk} will then show that $-x_n$ and $-x_n - 1$ are also in the correct relative position. 
By the choice of $x_n$, we see that $x_n + 1$ in the correct position. 
We have to show that $x_n$ is less than $\pm x'_i$ and $\pm (x'_i + 1)$ for all $i' \in [n-1]$. 
If $x_n > x'_1$, then $x_n + 1 > x'_1 + 1$ and since $x_n + 1$ satisfies the inequalities specified by $w$, $\lt{1}{1}$ must be before $\lt{n}{1}$ in $w$. 
But by property (2) of \Cref{symsk}, this means that $\lt{1}{0}$ must be before $\lt{n}{0}$ in $w$, which is a contradiction. 
The same logic can be used to show that $x_n$ satisfies the other inequalities given by $w$.
\end{proof}

We now derive some properties of symmetric sketches. 
A symmetric sketch has $4n$ letters, so we call the word made by the first $2n$ letters its first half. 
Similarly we define its second half.

\begin{lemma}\label{mir}
The second half of a symmetric sketch is completely specified by its first half.
In fact, it is the `mirror' of the first half, i.e., it is the reverse of the first half with each letter replaced with its conjugate.
\end{lemma}

\begin{proof}
For any symmetric sketch, the letter \lt{i}{s} is in the first half if and only if the letter $\overline{\lt{i}{s}}$ is in the second half. 
This property can be proved as follows: 
Suppose there is a pair of conjugates in the first half of a symmetric sketch. 
Since conjugate pairs partition $C(n)$, this means that there is a pair of conjugates in the second half as well. 
But this would contradict property (1) of a symmetric sketch in \Cref{symsk}.

Hence, the set of letters in the second half are the conjugates of the letters in the first half.
The order in which they appear is forced by property (1) of \Cref{symsk}, that is, the conjugates appear in the opposite order as the corresponding letters in the first half. 
So if the first half of a symmetric sketch is $a_1 \cdots a_{2n}$ where $a_i \in C(n)$ for all $i \in [2n]$, the sketch is
\begin{equation*}
    a_1\quad a_2\quad \cdots\quad a_{2n}\quad \overline{a_{2n}}\quad \cdots\quad \overline{a_2}\quad \overline{a_1}.
\end{equation*}
\end{proof}

We draw a vertical line between the $2n^{th}$ and $(2n+1)^{th}$ letter in a symmetric sketch to indicate both the mirroring and the change in sign (note that if the $2n^{th}$ letter is \lt{i}{s}, we have $x_i+s < 0 < -x_i-s$ in the corresponding region).

\begin{example}
$\alpha_{-3}^{(-1)} \ \alpha_{-3}^{(0)} \ \alpha_{1}^{(0)} \ \alpha_{-2}^{(-1)} \ \alpha_{1}^{(1)} \ \alpha_{2}^{(0)} \ \textcolor{blue}{|}\ \alpha_{-2}^{(0)} \ \alpha_{-1}^{(-1)} \ \alpha_{2}^{(1)} \ \alpha_{-1}^{(0)} \ \alpha_{3}^{(0)} \ \alpha_{3}^{(1)}$.
\end{example}

A letter in $C(n)$ is called an \emph{$\alpha$-letter} if it is of the form \lt{i}{0} or \lt{-i}{-1} where $i \in [n]$. 
The other letters are called \emph{$\beta$-letters}. 
The $\beta$-letter `corresponding' to an $\alpha$-letter is the one with the same subscript. 
Hence, in a symmetric sketch, an $\alpha$-letter always appears before its corresponding $\beta$-letter by property (3) in \Cref{symsk}. 
The order in which the subscripts of the $\alpha$-letters appear is the same as the order in which the subscripts of the $\beta$-letters appear by property (2) of \Cref{symsk}. 
The proof of the following lemma is very similar to that of the previous lemma.

\begin{lemma}\label{ord}
    The order in which the subscripts of the $\alpha$-letters in a symmetric sketch appear is of the form
    \begin{equation*}
        i_1 \quad i_2 \quad \cdots \quad i_n \quad -i_n \quad \cdots \quad -i_2 \quad -i_1
    \end{equation*}
    where $\{|i_1|,\ldots,|i_n|\}=[n]$.
\end{lemma}

Using \Cref{mir,ord}, to specify the sketch, we only need to specify the following:
\begin{enumerate}
    \item The $\alpha,\beta$-word corresponding to the first half.
    \item The signed permutation given by the first $n$ $\alpha$-letters.
\end{enumerate}
The $\alpha,\beta$-word corresponding to the first half is a word of length $2n$ in the letters $\{\alpha, \beta\}$ such that the $i^{th}$ letter is an $\alpha$ if and only if the $i^{th}$ letter of the symmetric sketch is an $\alpha$-letter.

There is at most one sketch corresponding to a pair of an $\alpha,\beta$-word and a signed permutation. 
This is because the signed permutation tells us, by \Cref{ord}, the order in which the subscripts of the $\alpha$-letters (and hence $\beta$-letters) appears. 
Using this and the $\alpha,\beta$-word, we can construct the first half and, by \Cref{mir}, the entire sketch.

\begin{example}\label{sketchexample}
To the symmetric sketch
\begin{equation*}
    \alpha_{-3}^{(-1)} \ \alpha_{-3}^{(0)} \ \alpha_{1}^{(0)} \ \alpha_{-2}^{(-1)} \ \alpha_{1}^{(1)} \ \alpha_{2}^{(0)} \ \textcolor{blue}{|}\ \alpha_{-2}^{(0)} \ \alpha_{-1}^{(-1)} \ \alpha_{2}^{(1)} \ \alpha_{-1}^{(0)} \ \alpha_{3}^{(0)} \ \alpha_{3}^{(1)}
\end{equation*}
we associate the pair consisting of the following:
\begin{enumerate}
    \item $\alpha,\beta$-word: $\alpha\beta\alpha\alpha\beta\alpha$.
    \item Signed permutation: $-3\quad 1\quad -2$.
\end{enumerate}
If we are given the $\alpha,\beta$-word and signed permutation above, the unique sketch corresponding to it is the one given above.
\end{example}

The next proposition characterizes the pairs of $\alpha,\beta$-words and signed permutations that correspond to symmetric sketches.

\begin{proposition}\label{absp}
A pair consisting of
\begin{enumerate}
    \item an $\alpha,\beta$-word of length $2n$ such that any prefix of the word has at least as many $\alpha$-letters as $\beta$-letters and
    \item any signed permutation
\end{enumerate}
corresponds to a symmetric sketch and all symmetric sketches correspond to such pairs.
\end{proposition}

\begin{proof}
By property (3) of \Cref{symsk}, any $\alpha,\beta$-word corresponding to the first half of a sketch should have at least as many $\alpha$-letters as $\beta$-letters in any prefix.

We now prove that given such a pair, there is a symmetric sketch corresponding to it.
If the given $\alpha,\beta$-word is $l_1 l_2 \cdots l_{2n}$ and the given signed permutation is $i_1 i_2 \cdots i_n$, we construct the symmetric sketch as follows:

\begin{enumerate}
    \item Extend the $\alpha,\beta$-word to the one of length $4n$ given by
    \begin{equation*}
        l_1 \quad l_2 \quad \cdots \quad l_{2n} \quad \overline{l_{2n}} \quad \cdots \quad \overline{l_2} \quad \overline{l_1}
    \end{equation*}
    where $\overline{l_i}=\alpha$ if and only if $l_i=\beta$ for all $i \in [2n]$.
    
    \item Extend the signed permutation to the sequence of length $2n$ given by
    \begin{equation*}
    i_1 \quad i_2 \quad  \cdots \quad  i_n \quad  -i_n \quad  \cdots \quad -i_2 \quad -i_1.
    \end{equation*}
    
    \item Label the subscripts of the $\alpha$-letters of the extended $\alpha,\beta$-word in the order given by the extended signed permutation and similarly label the $\beta$-letters.
\end{enumerate}

If we show that the word constructed is a symmetric sketch, it is clear that it will correspond to the given $\alpha,\beta$-word and signed permutation. 
We have to check that the constructed word satisfies the properties in \Cref{symsk}.

The way the word was constructed, we see that it is of the form
\begin{equation*}
    a_1 \quad a_2 \quad \cdots \quad a_{2n} \quad \overline{a_{2n}} \quad \cdots \quad \overline{a_2} \quad \overline{a_1}
\end{equation*}
where $a_i \in C(n)$ for all $i \in [2n]$. 
Since the conjugate of the $i^{th}$ $\alpha$ is the $(2n-i+1)^{th}$ $\beta$ and vice-versa, the first half of the word cannot have a pair of conjugates. 
Hence the word has all letters of $C(n)$. 
This shows that property (1) of \Cref{symsk} holds. 
Property (2) is taken care of since, by construction, the subscripts of the $\alpha$-letters appear in the same order as those of the $\beta$-letters.

To show that property (3) holds, it suffices to show that any prefix of the word has at least as many $\alpha$-letters as $\beta$-letters. 
This is already true for the first half. 
To show that this is true for the entire word, we consider $\alpha$ as $+1$ and $\beta$ as $-1$. 
Hence, the condition is that any prefix has a non-negative sum. 
Since any prefix of size greater than $2n$ is of the form
\begin{equation*}
    l_1 \quad l_2 \quad \cdots \quad l_{2n} \quad \overline{l_{2n}} \quad \cdots \quad \overline{l_k}
\end{equation*}
for some $k \in [2n]$, the sum is $l_1 + \cdots + l_{k-1} \geq 0$. 
So property (3) holds as well and hence the constructed word is a symmetric sketch.
\end{proof}

We use this description to count symmetric sketches.

\begin{lemma}\label{countabtypc}
The number of $\alpha,\beta$-words of length $2n$ having at least as many $\alpha$-letters as $\beta$-letters in any prefix is $\binom{2n}{n}$.
\end{lemma}

\begin{proof}
We consider these $\alpha, \beta$-words as lattice paths. 
Using the step $U = (1, 1)$ for $\alpha$ and the step $D = (1, -1)$ for $\beta$, we have to count those lattice paths with each step $U$ or $D$ that start at the origin, have $2n$ steps, and never fall below the $x$-axis.

Using the reflection principle (for example, see \cite{Hilton1991CatalanNT}), we get that the number of such lattice paths that end at $(2n, 2k)$ for $k \in [0, n]$ is given by
\begin{equation*}
    \binom{2n}{n + k} - \binom{2n}{n + k + 1}.
\end{equation*}
The (telescoping) sum over $k \in [0, n]$ gives the required result.
\end{proof}

The above lemma and \Cref{absp} immediately give the following.

\begin{theorem}\label{thmtypeccatalannumber}
The number of symmetric sketches and hence regions of $\C_n$ is
\begin{equation*}
    2^nn!\binom{2n}{n}.
\end{equation*}
\end{theorem}

In \cite{ath_non}, Athanasiadis obtains bijections between several classes of non-nesting partitions and regions of certain arrangements. 
We will mention the one for the arrangement $\C_n$, which gives a bijection between the $\alpha,\beta$-words associated to symmetric sketches and certain non-nesting partitions.

\begin{definition}\label{symnonnest}
A \emph{symmetric non-nesting partition} is a partition of $[-2n,2n] \setminus \{0\}$ such that the following hold:
\begin{enumerate}
    \item Each block is of size $2$.
    
    \item If $B = \{a, b\}$ is a block, so is $-B = \{-a, -b\}$.
    
    \item If $\{a, b\}$ is a block and $c, d \in [-2n, 2n] \setminus \{0\}$ are such that $a < c < d < b$, then $\{c, d\}$ is not a block.
\end{enumerate}
\end{definition}

Symmetric non-nesting partitions are usually represented using arc-diagrams. 
This is done by using $4n$ dots to represent the numbers in $[-2n, 2n] \setminus \{0\}$ in order and joining dots in the same block using an arc. 
The properties of these partitions imply that there are no nesting arcs and that the diagram is symmetric, which we represent by drawing a line after $2n$ dots.

\begin{example}\label{arcex}
The arc diagram associated to the symmetric non-nesting partition of $[-6, 6] \setminus \{0\}$
\begin{equation*}
    \{-6, -3\}, \{-5, -1\}, \{-4, 2\}, \{-2, 4\}, \{1, 5\}, \{3, 6\}
\end{equation*}
is given in \Cref{arcdiag}.
\end{example}

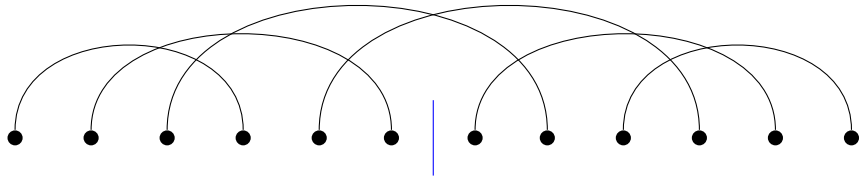
\begin{figure}[H]
    \centering
    \begin{tikzpicture}
        \node [circle, fill = black, inner sep = 2pt] (-6) at (-6+0.5,0) {};
        \node [circle, fill = black, inner sep = 2pt] (-5) at (-5+0.5,0) {};
        \node [circle, fill = black, inner sep = 2pt] (-4) at (-4+0.5,0) {};
        \node [circle, fill = black, inner sep = 2pt] (-3) at (-3+0.5,0) {};
        \node [circle, fill = black, inner sep = 2pt] (-2) at (-2+0.5,0) {};
        \node [circle, fill = black, inner sep = 2pt] (-1) at (-1+0.45,0) {};
        \draw [blue] (0,-0.5)--(0,0.5);
        \node [circle, fill = black, inner sep = 2pt] (1) at (1-0.45,0) {};
        \node [circle, fill = black, inner sep = 2pt] (2) at (2-0.5,0) {};
        \node [circle, fill = black, inner sep = 2pt] (3) at (3-0.5,0) {};
        \node [circle, fill = black, inner sep = 2pt] (4) at (4-0.5,0) {};
        \node [circle, fill = black, inner sep = 2pt] (5) at (5-0.5,0) {};
        \node [circle, fill = black, inner sep = 2pt] (6) at (6-0.5,0) {};
        \draw (-6.north)..controls +(up:15mm) and +(up:15mm)..(-3.north);
        \draw (6.north)..controls +(up:15mm) and +(up:15mm)..(3.north);
        \draw (-4.north)..controls +(up:22mm) and +(up:22mm)..(2.north);
        \draw (-2.north)..controls +(up:22mm) and +(up:22mm)..(4.north);
        \draw (-5.north)..controls +(up:17mm) and +(up:17mm)..(-1.north);
        \draw (5.north)..controls +(up:17mm) and +(up:17mm)..(1.north);
    \end{tikzpicture}
    \caption{The symmetric non-nesting partition of \Cref{arcex}.}
    \label{arcdiag}
\end{figure}

It can also be seen that there are exactly $n$ pairs of blocks of the form $\{B,-B\}$ with no block containing both a number and its negative. 
Also, the first $n$ blocks, with blocks being read in order of the smallest element in it, do not have a pair of the form $\{B,-B\}$. 
Hence, we can label the first $n$ blocks with a signed permutation and label the block $-B$ with the negative of the label of $B$ to obtain a labeling of all blocks. 
We call such objects \emph{labeled symmetric non-nesting partitions}. 
In the arc diagram, the labeling is done by replacing the dots representing the elements in a block with its label.

We can obtain a labeled symmetric non-nesting partition from a symmetric sketch by joining the letters \lt{i}{0} and \lt{i}{1} and similarly \lt{-i}{-1} and \lt{-i}{0} with arcs and replacing each letter in the sketch with its subscript. 
It can be shown that this construction is a bijection between symmetric sketches and labeled symmetric non-nesting partitions. 
In particular, the $\alpha,\beta$-words associated with symmetric sketches are in bijection with symmetric non-nesting partitions.

\begin{example}\label{slnnpex}
To the symmetric sketch
\begin{equation*}
    \lt{3}{0} \lt{2}{0} \lt{-1}{-1} \lt{3}{1} \lt{1}{0} \lt{2}{1} \textcolor{blue}{|} \lt{-2}{-1} \lt{-1}{0} \lt{-3}{-1} \lt{1}{1} \lt{-2}{0} \lt{-3}{0}
\end{equation*}
we associate the labeled symmetric non-nesting partition in \Cref{slnnp}.
\end{example}

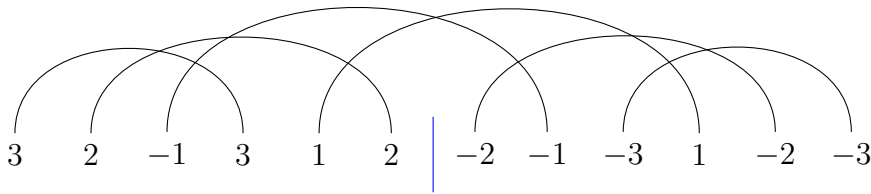
\begin{figure}[H]
    \centering
    \begin{tikzpicture}
        \node (-6) at (-6+0.5,0) {$3$};
        \node (-5) at (-5+0.5,0) {$2$};
        \node (-4) at (-4+0.5,0) {$-1$};
        \node (-3) at (-3+0.5,0) {$3$};
        \node (-2) at (-2+0.5,0) {$1$};
        \node (-1) at (-1+0.45,0) {$2$};
        \draw [blue] (0,-0.5)--(0,0.5);
        \node (1) at (1-0.45,0) {$-2$};
        \node (2) at (2-0.5,0) {$-1$};
        \node (3) at (3-0.5,0) {$-3$};
        \node (4) at (4-0.5,0) {$1$};
        \node (5) at (5-0.5,0) {$-2$};
        \node (6) at (6-0.5,0) {$-3$};
        \draw (-6.north)..controls +(up:15mm) and +(up:15mm)..(-3.north);
        \draw (6.north)..controls +(up:15mm) and +(up:15mm)..(3.north);
        \draw (-4.north)..controls +(up:22mm) and +(up:22mm)..(2.north);
        \draw (-2.north)..controls +(up:22mm) and +(up:22mm)..(4.north);
        \draw (-5.north)..controls +(up:17mm) and +(up:17mm)..(-1.north);
        \draw (5.north)..controls +(up:17mm) and +(up:17mm)..(1.north);
    \end{tikzpicture}
    \caption{Arc diagram associated to the symmetric sketch in \Cref{slnnpex}.}
    \label{slnnp}
\end{figure}

We now describe another way to represent the regions. 
We have already seen that a sketch corresponds to a pair consisting of an $\alpha,\beta$-word and a signed permutation. 
We represent the $\alpha,\beta$-word as a lattice path just as we did in the proof of \Cref{countabtypc}. 
% This is done by starting at the origin, reading the $\alpha,\beta$-word, and moving by an up-step $(1,1)$ if the letter read is $\alpha$ and a down-step $(1,-1)$ if the letter read is $\beta$. 
We specify the signed permutation by labeling the first $n$ up-steps of the lattice path.

\begin{example}\label{sketchtolatticex}
The lattice path associated to the symmetric sketch
\begin{equation*}
    \alpha_{-3}^{(-1)} \ \alpha_{-3}^{(0)} \ \alpha_{1}^{(0)} \ \alpha_{-2}^{(-1)} \ \alpha_{1}^{(1)} \ \alpha_{2}^{(0)} \ \textcolor{blue}{|}\ \alpha_{-2}^{(0)} \ \alpha_{-1}^{(-1)} \ \alpha_{2}^{(1)} \ \alpha_{-1}^{(0)} \ \alpha_{3}^{(0)} \ \alpha_{3}^{(1)}
\end{equation*}
is given in \Cref{latticeexample}.
\end{example}

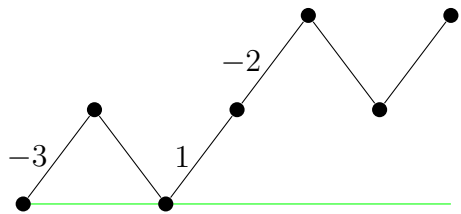
\begin{figure}[H]
    \centering
    \begin{tikzpicture}[xscale=0.75,scale=1.25]
    \draw [thin,green] (0,0)--(6,0);
    \node (0) [circle,inner sep=2pt,fill=black] at (0,0) {};
    \node (1) [circle,inner sep=2pt,fill=black] at (1,1) {};
    \node (2) [circle,inner sep=2pt,fill=black] at (2,0) {};
    \node (3) [circle,inner sep=2pt,fill=black] at (3,1) {};
    \node (4) [circle,inner sep=2pt,fill=black] at (4,2) {};
    \node (5) [circle,inner sep=2pt,fill=black] at (5,1) {};
    \node (6) [circle,inner sep=2pt,fill=black] at (6,2) {};
    \draw (0)--node[left]{$-3$}(1)--(2)--node[left]{$1$}(3)--node[left]{$-2$}(4)--(5)--(6);
    \end{tikzpicture}
    \caption{Lattice path associated to the symmetric sketch in \Cref{sketchtolatticex}.}
    \label{latticeexample}
\end{figure}

These representations for the regions of $\C_n$ also allow us to determine and count which regions are bounded.

\begin{theorem}\label{typeCbdd}
The number of bounded regions of the arrangement $\C_n$ is
\begin{equation*}
    2^nn!\binom{2n - 1}{n}.
\end{equation*}
\end{theorem}

\begin{proof}
First note that the arrangement $\C_n$ has rank $n$ and is hence essential. 
From the bijection defined above, it can be seen that the arc diagram associated to any region $R$ of $\C_n$ can be obtained by plotting a point $(x_1, \ldots, x_n) \in R$ on the real line. 
This is done by marking $x_i$ and $x_i + 1$ on the real line using $i$ for all $i \in [n]$ and then joining them with an arc and similarly marking $-x_i -1$ and $-x_i$ using $-i$ and joining them with an arc.

This can be used to show that a region of $\C_n$ is bounded if and only if the arc diagram is `interlinked'. 
For example, \Cref{slnnp} shows an arc diagram that is interlinked and \Cref{notinterlinked} shows one that is not. 
In terms of lattice paths, the bounded regions are those whose corresponding lattice path never touches the $x$-axis except at the origin.

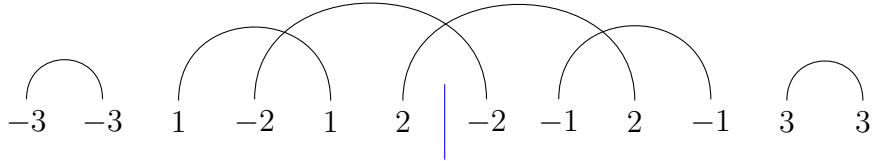
\begin{figure}[H]
    \centering
    \begin{tikzpicture}
        \node (-6) at (-6+0.5,0) {$-3$};
        \node (-5) at (-5+0.5,0) {$-3$};
        \node (-4) at (-4+0.5,0) {$1$};
        \node (-3) at (-3+0.5,0) {$-2$};
        \node (-2) at (-2+0.5,0) {$1$};
        \node (-1) at (-1+0.45,0) {$2$};
        \draw [blue] (0,-0.5)--(0,0.5);
        \node (1) at (1-0.45,0) {$-2$};
        \node (2) at (2-0.5,0) {$-1$};
        \node (3) at (3-0.5,0) {$2$};
        \node (4) at (4-0.5,0) {$-1$};
        \node (5) at (5-0.5,0) {$3$};
        \node (6) at (6-0.5,0) {$3$};
        \draw (-6.north)..controls +(up:7mm) and +(up:7mm)..(-5.north);
        \draw (6.north)..controls +(up:7mm) and +(up:7mm)..(5.north);
        \draw (-4.north)..controls +(up:13mm) and +(up:13mm)..(-2.north);
        \draw (4.north)..controls +(up:13mm) and +(up:13mm)..(2.north);
        \draw (-3.north)..controls +(up:17mm) and +(up:17mm)..(1.north);
        \draw (3.north)..controls +(up:17mm) and +(up:17mm)..(-1.north);
    \end{tikzpicture}
    \caption{Arc diagram associated to the symmetric sketch of \Cref{sketchexample}.}
    \label{notinterlinked}
\end{figure}

This shows that the number of bounded regions of $\C_n$ is $2^nn!$ times the number of unlabeled lattice paths of length $2n$ that never touch the $x$-axis apart from at the origin. 
Deleting the first step (which is necessarily an up-step) gives a bijection between such paths and those of length $2n - 1$ that never fall below the $x$-axis. 
Using the same idea as in the proof of \Cref{countabtypc}, it can be checked that the number of such paths is $\binom{2n - 1}{n}$. 
This proves the required result.
\end{proof}

% \textbf{For type C}: In terms of non-nesting partitions, it is those that are interlinked (`connected'). 
% In terms of the $\alpha,\beta$-words, it is those words whose corresponding path do not touch the $x$-axis except at the beginning.
% These can be counted as done for the statistic and give us the required number of bounded regions.
% \textbf{For type D}: The type D regions that are bounded are exactly those that contain only bounded type C regions (since ranks are same). 
% Such type D regions can also be counted to give the required expression. 
% The only problematic bounded type C regions are those that end with an up step at height $2$ which are in bijection with Dyck paths of semilength $n - 1$. 
% This gives number of bounded type D regions as
% \begin{equation*}
%     2^n n! \left(\frac{3}{8}\left(\binom{2n - 1}{n} + \frac{1}{n} \binom{2n - 2}{n - 1}\right) - \frac{1}{n}\binom{2n - 2}{n - 1}\right).
% \end{equation*}

\begin{remark}
In \cite{typc}, the authors study the type $C$ Catalan arrangement directly, i.e., without using the translation $\C_n$ mentioned above. 
Hence, using the same logic, they use orders on the letters
\begin{equation*}
    \{\lt{i}{s} \mid i \in [-n,n] \setminus \{0\},\ s \in \{0,1\}\}
\end{equation*}
to represent the regions of the type $C$ Catalan arrangement. 
They claim that these orders are those such that the following hold for any $i,j \in [-n,n] \setminus \{0\}$ and $s \in \{0, 1\}$:
\begin{enumerate}
    \item If $\alpha_i^{(0)}$ appear before $\alpha_j^{(0)}$, then $\alpha_i^{(1)}$ appears before $\alpha_j^{(1)}$.
    
    \item $\alpha_i^{(0)}$ appears before $\alpha_i^{(1)}$.
    
    \item If $\alpha_i^{(0)}$ appears before $\alpha_j^{(s)}$, then $\alpha_{-j}^{(0)}$ appears before ${\alpha_{-i}^{(s)}}$.
\end{enumerate}
Though this can be shown to be true, the method used in \cite{typc} to construct a point satisfying the inequalities given by such an order does not seem to work in general. 
We describe their method and then exhibit a case where it does not work.

Let $w=w_1 \cdots w_{4n}$ be an order satisfying the properties given above. 
Then construct $\boldsymbol{x}=(x_1, \ldots, x_n) \in \mathbb{R}^n$ as follows:
Let $z_0=0$ (or pick $z_0$ arbitrarily).
Then define $z_p$ for $p=1, 2, \ldots, 4n$ in order as follows:
If $w_p=\alpha_i^{(0)}$ then set $z_p=z_{p-1}+\frac{1}{2n+1}$ and $x_i=z_p$, and if $w_p=\alpha_i^{(1)}$ then set $z_p=x_i+1$.
Here we consider $x_{-i}=-x_i$ for any $i \in [n]$.
Then $\boldsymbol{x}$ satisfies the inequalities given by $w$.

The following example shows that this method does not always work; in fact $\boldsymbol{x}$ is not always well-defined.
Consider the order $w=\lt{-2}{0}\lt{1}{0}\lt{-2}{1}\lt{1}{1}\lt{-1}{0}\lt{2}{0}\lt{-1}{1}\lt{2}{1}$.
Following the above procedure, we would get that $x_1$ is both $\frac{2}{5}$ as well as $-1-\frac{3}{5}$.
\end{remark}

\subsection{Extended type C Catalan}\label{extc}

Fix $m, n \geq 1$. 
The type $C$ $m$-Catalan arrangement in $\R^n$ has hyperplanes
\begin{align*}
    2X_i &= 0, \pm 1, \pm2, \ldots, \pm m\\
    X_i+X_j &= 0, \pm 1, \pm2, \ldots, \pm m\\
    X_i-X_j &= 0, \pm 1, \pm2, \ldots, \pm m
\end{align*}
for all $1 \leq i < j \leq n$.
We will study the arrangement obtained by performing the translation $X_i=x_i+\frac{m}{2}$ for all $i \in [n]$.
The translated arrangement, which we call $\C_n^{(m)}$, has hyperplanes
\begin{align*}
    2x_i &= -2m, -2m + 1, \ldots, 0\\
    x_i+x_j &= -2m, -2m + 1, \ldots, 0\\
    x_i-x_j &= 0, \pm 1,\pm2, \ldots, \pm m
\end{align*}
for all $1 \leq i < j \leq n$. 
Note that $\C_n = \C_n^{(1)}$. 
The regions of $\C_n^{(m)}$ are given by valid total orders on
\begin{equation*}
    \{x_i+s \mid i \in [n],\ s \in [0,m]\} \cup \{-x_i-s \mid i \in [n],\ s \in [0,m]\}.
\end{equation*}

Just as we did for $\C_n$, such orders will be represented by using the symbol \lt{i}{s} for $x_i+s$ and \lt{-i}{-s} for $-x_i-s$ for all $i \in [n]$ and $s \in [0,m]$. 
Let $C^{(m)}(n)$ be the set
\begin{equation*}
    \{\lt{i}{s}\mid i \in [n],\ s \in [0,m]\} \cup \{\alpha_i^{(s)} \mid -i \in [n],\ s \in [-m,0]\}.
\end{equation*}
For any $\lt{i}{s} \in C^{(m)}(n)$, $\overline{\lt{i}{s}}$ represents \lt{-i}{-s} and is called the conjugate of \lt{s}{i}. 
Letters of the form \lt{i}{0} or \lt{-i}{-m} for any $i \in [n]$ are called $\alpha$-letters. 
The others are called $\beta$-letters.

\begin{definition}
An order on the letters in $C^{(m)}(n)$ is called a \emph{symmetric $m$-sketch} if the following hold for all $\lt{i}{s}, \lt{j}{t} \in C^{(m)}(n)$:
\begin{enumerate}
    \item If $\alpha_i^{(s)}$ appears before $\alpha_j^{(t)}$, then $\overline{\alpha_j^{(t)}}$ appears before $\overline{\alpha_i^{(s)}}$.
    
    \item If $\alpha_i^{(s-1)}$ appears before $\alpha_j^{(t-1)}$, then $\alpha_i^{(s)}$ appears before $\alpha_j^{(t)}$.
    
    \item $\alpha_i^{(s-1)}$ appears before $\alpha_i^{(s)}$.
\end{enumerate}
\end{definition}

The following result can be proved just as \Cref{pointconst}.

\begin{proposition}
An order on the letters in $C^{(m)}(n)$ corresponds to a region of $\C^{(m)}_n$ if and only if it is a symmetric $m$-sketch.
\end{proposition}

Similar to Lemma \ref{ord}, it can be shown that the order in which the subscripts of the $\alpha$-letters appear in a symmetric $m$-sketch is of the form
\begin{equation*}
    i_1\quad i_2 \quad  \cdots\quad  i_n\quad  -i_n\quad  \cdots \quad -i_2 \quad -i_1
\end{equation*}
where $\{|i_1|,\ldots,|i_n|\}=[n]$. 
Just as in the case of symmetric sketches, we associate an $\alpha,\beta$-word and signed permutation to a symmetric $m$-sketch which completely determines it.

\begin{example}\label{sym2ex}
To the symmetric 2-sketch
\begin{equation*}
    \alpha_{2}^{(0)}\alpha_{-1}^{(-2)}\alpha_{2}^{(1)}\alpha_{-1}^{(-1)}\alpha_{1}^{(0)}\alpha_{-2}^{(-2)}\ \textcolor{blue}{|}\ \alpha_{2}^{(2)}\alpha_{-1}^{(0)}\alpha_{1}^{(1)}\alpha_{-2}^{(-1)}\alpha_{1}^{(2)}\alpha_{-2}^{(0)}
\end{equation*}
we associate the pair consisting of the following:
\begin{enumerate}
    \item $\alpha,\beta$-word: $\alpha\alpha\beta\beta\alpha\alpha$.
    \item Signed permutation: $2\ \ -1$.
\end{enumerate}
\end{example}

The set of $\alpha,\beta$-words associated to symmetric $m$-sketches for $m>1$ does not seem to have a simple characterization like those for symmetric sketches (see Proposition \ref{absp}). 
However, looking at symmetric $m$-sketches as labeled non-nesting partitions as done in \cite{ath_non}, we see that such objects have already been counted bijectively (refer \cite{non_bij}).

\begin{definition}
A \emph{symmetric $m$-non-nesting partition} is a partition of $[-(m+1)n,(m+1)n] \setminus \{0\}$ such that the following hold:
\begin{enumerate}
    \item Each block is of size $(m+1)$.
    \item If $B$ is a block, so is $-B$.
    \item If $a,b$ are in some block $B$, $a<b$ and there is no number $a<c<b$ such that $c \in B$, then if $a<c<d<b$, $c$ and $d$ are not in the same block.
\end{enumerate}
\end{definition}

Just as we did for the $m = 1$ case, we can obtain a labeled symmetric $m$-non-nesting partition from a symmetric $m$-sketch by joining the letters $\lt{i}{0}, \lt{i}{1}, \ldots, \lt{i}{m}$ and similarly $\lt{-i}{-m}, \lt{-i}{-m + 1}, \ldots, \lt{-i}{0}$ with arcs and labeling each such chain with the subscript of the letters being joined.

\begin{example}
To the symmetric 2-sketch in \Cref{sym2ex}, we associate the labeled 2-non-nesting partition of Figure \ref{slnnp2}.
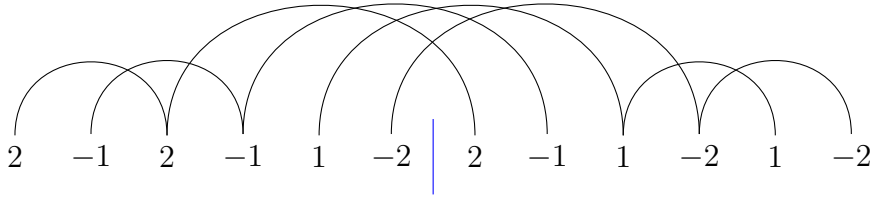
\begin{figure}[H]
    \centering
    \begin{tikzpicture}
        \node (-6) at (-6+0.5,0) {$2$};
        \node (-5) at (-5+0.5,0) {$-1$};
        \node (-4) at (-4+0.5,0) {$2$};
        \node (-3) at (-3+0.5,0) {$-1$};
        \node (-2) at (-2+0.5,0) {$1$};
        \node (-1) at (-1+0.45,0) {$-2$};
        \draw [blue] (0,-0.5)--(0,0.5);
        \node (1) at (1-0.45,0) {$2$};
        \node (2) at (2-0.5,0) {$-1$};
        \node (3) at (3-0.5,0) {$1$};
        \node (4) at (4-0.5,0) {$-2$};
        \node (5) at (5-0.5,0) {$1$};
        \node (6) at (6-0.5,0) {$-2$};
        \draw (-6.north)..controls +(up:13mm) and +(up:13mm)..(-4.north);
        \draw (-4.north)..controls +(up:23mm) and +(up:23mm)..(1.north);
        \draw (6.north)..controls +(up:13mm) and +(up:13mm)..(4.north);
        \draw (4.north)..controls +(up:23mm) and +(up:23mm)..(-1.north);
        \draw (-5.north)..controls +(up:13mm) and +(up:13mm)..(-3.north);
        \draw (-3.north)..controls +(up:23mm) and +(up:23mm)..(2.north);
        \draw (5.north)..controls +(up:13mm) and +(up:13mm)..(3.north);
        \draw (3.north)..controls +(up:23mm) and +(up:23mm)..(-2.north);
    \end{tikzpicture}
    \caption{A labeled 2-non-nesting partition}
    \label{slnnp2}
\end{figure}
\end{example}

The number of various classes of non-nesting partitions have been counted bijectively. 
In terms of \cite{non_bij} or \cite{ath_non}, the symmetric $m$-non-nesting partitions defined above are called type $C$ partitions of size $(m+1)n$ of type $(m+1,\ldots,m+1)$ where this is an $n$-tuple representing the size of the (nonzero) block pairs $\{B,-B\}$.
The number of such partitions is
\begin{equation*}
    \binom{(m+1)n}{n}.
\end{equation*}
Hence we get the following theorem.

\begin{theorem}
The number of symmetric $m$-sketches, which is the number of regions of $\C_n^{(m)}$ is
\begin{equation*}
    2^nn!\binom{(m+1)n}{n}.
\end{equation*}
\end{theorem}

\section{Catalan deformations of other types}\label{othertypessec}

We will now use `sketches and moves', as in \cite{ber}, to count the regions of Catalan arrangements of other types. 
Depending on the context, we represent the regions of arrangements using sketches, arc diagrams, or lattice paths and frequently make use of the bijections identifying them. 
We usually use sketches to define moves and use arc diagrams and lattice paths to count regions as well as bounded regions.

\subsection{Type D Catalan}\label{dcatsec}

Fix $n \geq 2$. 
The type $D$ Catalan arrangement in $\mathbb{R}^n$ has hyperplanes
\begin{align*}
    X_i+X_j&=-1,0,1\\
    X_i-X_j&=-1,0,1
\end{align*}
for $1 \leq i < j \leq n$. 
Translating this arrangement by setting $X_i=x_i+\frac{1}{2}$ for all $i \in [n]$, we get the arrangement $\D_n$ with hyperplanes
\begin{align*}
    x_i+x_j&=-2,-1,0\\
    x_i-x_j&=-1,0,1
\end{align*}
for $1 \leq i < j \leq n$. 
\Cref{typdintypc} shows $\D_2$ as a sub-arrangement of $\C_2$. 
It also shows how the regions of $\D_2$ partition the regions of $\C_2$.

\begin{figure}[H]
    \centering
    \begin{tikzpicture}[scale=3]
    \clip (-2,-2) rectangle (2,2);
    %\draw[step=0.25cm,blue,very thin] (-2,-2) grid (6,6);
        \draw[ultra thin] (0,2.5)--(0,-2.5);
        \draw[ultra thin] (0.5,2.5)--(0.5,-2.5);
        \draw[ultra thin] (-0.5,2.5)--(-0.5,-2.5);
        \draw[ultra thin] (3,0)--(-3,0);
        \draw[ultra thin] (3,0.5)--(-3,0.5);
        \draw[ultra thin] (3,-0.5)--(-3,-0.5);
        \draw[ultra thick] (-1.5,-1.5)--(1.5,1.5);
        \draw[ultra thick] (-1.5,1.5)--(1.5,-1.5);
        \draw[ultra thick] (-1.5,2.5)--(3,-2);
        \draw[ultra thick] (-3,2)--(1.5,-2.5);
        \draw[ultra thick] (2-0.5,3-0.5)--(-3.75+0.75,-2.75+0.75);
        \draw[ultra thick] (-2+0.5,-3+0.5)--(3.75-0.75,2.75-0.75);
        
        \draw (-1.30,-0.4)--(-0.7,-0.4);
        \node [inner sep=0pt] (1) at (-1.25,-0.4){};
        \node [inner sep=0pt] (2) at (-1.25+1/14,-0.4){};
        \node [inner sep=0pt] (3) at (-1.25+2/14,-0.4){};
        \node [inner sep=0pt] (4) at (-1.25+3/14,-0.4){};
        \node [inner sep=0pt] (-4) at (-1.25+4/14,-0.4){};
        \node [inner sep=0pt] (-3) at (-1.25+5/14,-0.4){};
        \node [inner sep=0pt] (-2) at (-1.25+6/14,-0.4){};
        \node [inner sep=0pt] (-1) at (-1.25+7/14,-0.4){};
        \filldraw (-1.25,-0.4) circle (0.3pt);
        \filldraw (-1.25+1/14,-0.4) circle (0.3pt);
        \filldraw (-1.25+2/14,-0.4) circle (0.3pt);
        \filldraw (-1.25+3/14,-0.4) circle (0.3pt);
        \filldraw (-1.25+4/14,-0.4) circle (0.3pt);
        \filldraw (-1.25+5/14,-0.4) circle (0.3pt);
        \filldraw (-1.25+6/14,-0.4) circle (0.3pt);
        \filldraw (-1.25+7/14,-0.4) circle (0.3pt);
        \draw (1.north)..controls +(up:0.75mm) and +(up:0.75mm)..(3.north);
        \draw (-1.north)..controls +(up:0.75mm) and +(up:0.75mm)..(-3.north);
        \draw (2.north)..controls +(up:1mm) and +(up:1mm)..(-4.north);
        \draw (-2.north)..controls +(up:1mm) and +(up:1mm)..(4.north);
        \draw (-1.25,-0.45) node {\tiny 1};
        \draw (-1.25+1/14,-0.45) node {\tiny 2};
        \draw (-1.25+2.85/14,-0.45) node {\tiny -2};
        \draw (-1.25+4.85/14,-0.45) node {\tiny -1};
        
        \draw (-1.30-0.25,-0.4+-0.4)--(-0.7-0.25,-0.4+-0.4);
        \node [inner sep=0pt] (1') at (-1.25-0.25,-0.4+-0.4){};
        \node [inner sep=0pt] (2') at (-1.25+1.15/14-0.25,-0.4+-0.4){};
        \node [inner sep=0pt] (3') at (-1.25+2/14-0.25,-0.4+-0.4){};
        \node [inner sep=0pt] (4') at (-1.25+3.15/14-0.25,-0.4+-0.4){};
        \node [inner sep=0pt] (-4') at (-1.25+3.85/14-0.25,-0.4+-0.4){};
        \node [inner sep=0pt] (-3') at (-1.25+5/14-0.25,-0.4+-0.4){};
        \node [inner sep=0pt] (-2') at (-1.25+5.85/14-0.25,-0.4+-0.4){};
        \node [inner sep=0pt] (-1') at (-1.25+7/14-0.25,-0.4+-0.4){};
        \filldraw (-1.25-0.25,-0.4+-0.4) circle (0.3pt);
        \filldraw (-1.25+1.15/14-0.25,-0.4+-0.4) circle (0.3pt);
        \filldraw (-1.25+2/14-0.25,-0.4+-0.4) circle (0.3pt);
        \filldraw (-1.25+3.15/14-0.25,-0.4+-0.4) circle (0.3pt);
        \filldraw (-1.25+3.85/14-0.25,-0.4+-0.4) circle (0.3pt);
        \filldraw (-1.25+5/14-0.25,-0.4+-0.4) circle (0.3pt);
        \filldraw (-1.25+5.85/14-0.25,-0.4+-0.4) circle (0.3pt);
        \filldraw (-1.25+7/14-0.25,-0.4+-0.4) circle (0.3pt);
        \draw (1'.north)..controls +(up:0.75mm) and +(up:0.75mm)..(3'.north);
        \draw (-1'.north)..controls +(up:0.75mm) and +(up:0.75mm)..(-3'.north);
        \draw (2'.north)..controls +(up:0.75mm) and +(up:0.75mm)..(4'.north);
        \draw (-2'.north)..controls +(up:0.75mm) and +(up:0.75mm)..(-4'.north);
        \draw (-1.25-0.25,-0.4-0.45) node {\tiny 1};
        \draw (-1.25+1/14-0.25,-0.4-0.45) node {\tiny 2};
        \draw (-1.25+3.65/14-0.25,-0.4-0.45) node {\tiny -2};
        \draw (-1.25+4.85/14-0.25,-0.4-0.45) node {\tiny -1};
    \end{tikzpicture}
    \caption{The arrangement $\C_2$ with the hyperplanes in $\D_2$ in bold. 
    Two regions of $\C_2$ are labeled with their symmetric labeled non-nesting partition.}
    \label{typdintypc}
\end{figure}
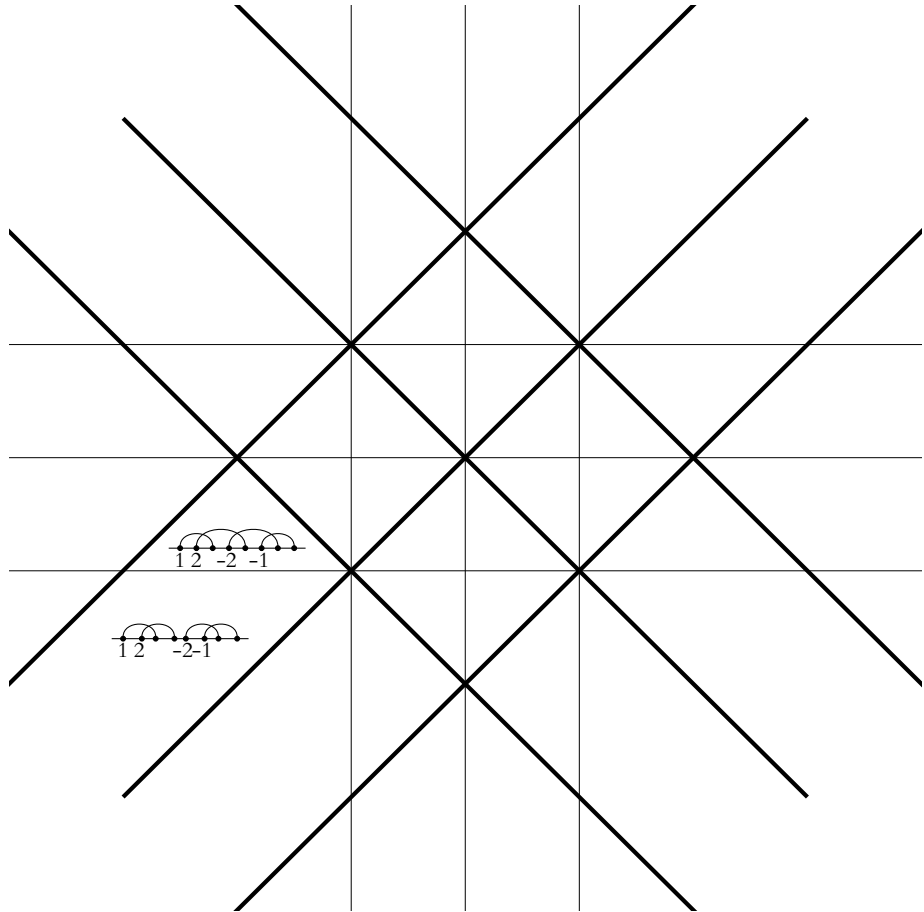

We use the idea of moves to count the regions of $\D_n$ by considering it as a sub-arrangement of $\C_n$. 
The hyperplanes from $\C_n$ that are missing in $\D_n$ are
\begin{equation*}
    2x_i=-2,-1,0
\end{equation*}
for all $i \in [n]$. 
Hence, the type $D$ Catalan moves on symmetric sketches (regions of $\C_n$), which we call $\D$ moves, are as follows:

\begin{enumerate}
    \item Swapping the $2n^{th}$ and $(2n+1)^{th}$ letter.
    \item Swapping the $n^{th}$ and $(n+1)^{th}$ $\alpha$-letters if they are adjacent, along with the $n^{th}$ and $(n+1)^{th}$ $\beta$-letters.
\end{enumerate}

The first move covers the inequalities corresponding to the hyperplanes $x_i+1=-x_i-1$ and $x_i=-x_i$ for all $i \in [n]$ since the only conjugates that are adjacent, by \Cref{mir}, are the $2n^{th}$ and $(2n+1)^{th}$ letter.

The second move covers the inequalities corresponding to the hyperplanes $x_i=-x_i-1$ (equivalently, $x_i+1=-x_i$) for all $i \in [n]$. 
This is due to the fact that the only way \lt{i}{0} and \lt{-i}{-1} as well as \lt{i}{1} and \lt{-i}{0} can be adjacent is, by \Cref{ord}, when the $n^{th}$ and $(n+1)^{th}$ $\alpha$-letters are adjacent. 
Also, by \Cref{mir}, the $n^{th}$ and $(n+1)^{th}$ $\alpha$-letters are adjacent if and only if the $n^{th}$ and $(n+1)^{th}$ $\beta$-letters are adjacent.

\begin{example}
A series of $\D$ moves applied to a symmetric sketch is given below:
\begin{align*}
&\alpha_{-1}^{(-1)}\alpha_{2}^{(0)}\alpha_{-2}^{(-1)}\alpha_{-1}^{(0)}\ \textcolor{blue}{|}\ \alpha_{1}^{(0)}\alpha_{2}^{(1)}\alpha_{-2}^{(0)}\alpha_{1}^{(1)}\\
\xrightarrow{\text{$\D$\ move}}\ &\alpha_{-1}^{(-1)}\alpha_{2}^{(0)}\alpha_{-2}^{(-1)}\alpha_{1}^{(0)}\ \textcolor{blue}{|}\ \alpha_{-1}^{(0)}\alpha_{2}^{(1)}\alpha_{-2}^{(0)}\alpha_{1}^{(1)}\\
\xrightarrow{\text{$\D$\ move}}\ &\alpha_{-1}^{(-1)}\alpha_{-2}^{(-1)}\alpha_{2}^{(0)}\alpha_{1}^{(0)}\ \textcolor{blue}{|}\ \alpha_{-1}^{(0)}\alpha_{-2}^{(0)}\alpha_{2}^{(1)}\alpha_{1}^{(1)}\\
\xrightarrow{\text{$\D$\ move}}\ &\alpha_{-1}^{(-1)}\alpha_{-2}^{(-1)}\alpha_{2}^{(0)}\alpha_{-1}^{(0)}\ \textcolor{blue}{|}\ \alpha_{1}^{(0)}\alpha_{-2}^{(0)}\alpha_{2}^{(1)}\alpha_{1}^{(1)}
\end{align*}
\end{example}

To count the regions of $\D_n$, we have to count the number of equivalence classes of symmetric sketches where two sketches are equivalent if one can be obtained from the other via a series of $\D$ moves. 
In \Cref{typdintypc}, the two labeled regions of $\C_2$ are adjacent and lie in the same region of $\D_2$. 
They are related by swapping of the fourth and fifth letters of their sketches, which is a $\D$ move.

The fact about these moves that will help with the count is that a series of $\D$ moves do not change the sketch too much. 
Hence we can list the sketches that are $\D$ equivalent to a given sketch.

First, consider the case when the $n^{th}$ $\alpha$-letter of the symmetric sketch is not in the $(2n-1)^{th}$ position. 
In this case, the $n^{th}$ $\alpha$-letter is far enough from the $2n^{th}$ letter that a $\D$ move of the first kind (swapping the $2n^{th}$ and $(2n+1)^{th}$ letter) will not affect the letter after the $n^{th}$ $\alpha$-letter. 
Hence it does not change whether the $n^{th}$ and $(n+1)^{th}$ $\alpha$-letters are adjacent.

Let $w$ be a sketch where the $n^{th}$ $\alpha$-letter is not in the $(2n-1)^{th}$ position. 
The number of sketches $\D$ equivalent to $w$ is $4$ when the $n^{th}$ and $(n+1)^{th}$ $\alpha$-letters are adjacent. 
They are illustrated below:
\begin{align*}
\cdots \alpha_{-i}^{(-1)}\alpha_{i}^{(0)} \cdots \alpha_j^{(s)}\ \textcolor{blue}{|}&\ \alpha_{-j}^{(-s)} \cdots \alpha_{-i}^{(0)}\alpha_{i}^{(1)} \cdots\\
\cdots \alpha_{-i}^{(-1)}\alpha_{i}^{(0)} \cdots \alpha_{-j}^{(-s)}\ \textcolor{blue}{|}&\ \alpha_j^{(s)} \cdots \alpha_{-i}^{(0)}\alpha_{i}^{(1)} \cdots\\
\cdots \alpha_{i}^{(0)}\alpha_{-i}^{(-1)} \cdots \alpha_j^{(s)}\ \textcolor{blue}{|}&\ \alpha_{-j}^{(-s)} \cdots \alpha_{i}^{(1)}\alpha_{-i}^{(0)} \cdots\\
\cdots \alpha_{i}^{(0)}\alpha_{-i}^{(-1)} \cdots \alpha_{-j}^{(-s)}\ \textcolor{blue}{|}&\ \alpha_j^{(s)} \cdots \alpha_{i}^{(1)}\alpha_{-i}^{(0)} \cdots
\end{align*}
The number of sketches $\D$ equivalent to $w$ is $2$ when the $n^{th}$ and $(n+1)^{th}$ $\alpha$-letter are not adjacent. 
They are illustrated below:
\begin{equation*}
\cdots \alpha_j^{(s)}\ \textcolor{blue}{|}\ \alpha_{-j}^{(-s)} \cdots \quad \cdots \alpha_{-j}^{(-s)}\ \textcolor{blue}{|}\ \alpha_j^{(s)} \cdots
\end{equation*}
% The number of sketches equivalent to a sketch when the $n^{th}$ $\alpha$-letter is not in the $(2n-1)^{th}$ position
% \begin{enumerate}
%     \item is $4$ when the $n^{th}$ and $(n+1)^{th}$ $\alpha$-letter are adjacent:
%         \begin{align*}
%         \cdots \alpha_{-i}^{(-1)}\alpha_{i}^{(0)} \cdots \alpha_j^{(s)}\ \textcolor{blue}{|}&\ \alpha_{-j}^{(-s)} \cdots \alpha_{-i}^{(0)}\alpha_{i}^{(1)} \cdots\\
%         \cdots \alpha_{-i}^{(-1)}\alpha_{i}^{(0)} \cdots \alpha_{-j}^{(-s)}\ \textcolor{blue}{|}&\ \alpha_j^{(s)} \cdots \alpha_{-i}^{(0)}\alpha_{i}^{(1)} \cdots\\
%         \cdots \alpha_{i}^{(0)}\alpha_{-i}^{(-1)} \cdots \alpha_j^{(s)}\ \textcolor{blue}{|}&\ \alpha_{-j}^{(-s)} \cdots \alpha_{i}^{(1)}\alpha_{-i}^{(0)} \cdots\\
%         \cdots \alpha_{i}^{(0)}\alpha_{-i}^{(-1)} \cdots \alpha_{-j}^{(-s)}\ \textcolor{blue}{|}&\ \alpha_j^{(s)} \cdots \alpha_{i}^{(1)}\alpha_{-i}^{(0)} \cdots
%         \end{align*}
%     \item and is $2$ when the $n^{th}$ and $(n+1)^{th}$ $\alpha$-letter are not adjacent:
%         \begin{equation*}
%         \cdots \alpha_j^{(s)}\ \textcolor{blue}{|}\ \alpha_{-j}^{(-s)} \cdots \quad \cdots \alpha_{-j}^{(-s)}\ \textcolor{blue}{|}\ \alpha_j^{(s)} \cdots
%         \end{equation*}
% \end{enumerate}
Notice also that the equivalent sketches also satisfy the same properties ($n^{th}$ $\alpha$-letter not being in the $(2n-1)^{th}$ position and whether the $n^{th}$ and $(n+1)^{th}$ $\alpha$-letters are adjacent).

In case the $n^{th}$ $\alpha$-letter is in the $(2n-1)^{th}$ position of the symmetric sketch, it can be checked that it has exactly 4 equivalent sketches all of which also have the $n^{th}$ $\alpha$-letter in the $(2n-1)^{th}$ position:
\begin{align*}
    \cdots \alpha_i^{(0)}\alpha_{i}^{(1)}\ \textcolor{blue}{|}&\ \alpha_{-i}^{(-1)}\alpha_{-i}^{(0)} \cdots\\  \cdots\ \alpha_i^{(0)}\alpha_{-i}^{(-1)}\ \textcolor{blue}{|}&\ \alpha_{i}^{(1)}\alpha_{-i}^{(0)} \cdots\\
    \cdots \alpha_{-i}^{(-1)}\alpha_{i}^{(0)}\ \textcolor{blue}{|}&\ \alpha_{-i}^{(0)}\alpha_{i}^{(1)} \cdots\\ \cdots\ \alpha_{-i}^{(-1)}\alpha_{-i}^{(0)}\ \textcolor{blue}{|}&\ \alpha_{i}^{(0)}\alpha_{i}^{(1)} \cdots
\end{align*}

\Cref{typdintypc} shows that each region of $\D_2$ contains exactly 2 or 4 regions of $\C_2$, as expected from the above observations.

\begin{theorem}\label{nooftypeDregions}
The number of $\D$ equivalence classes on symmetric sketches and hence the number of regions of $\D_n$ is
\begin{equation*}
    2^{n-1} \cdot \frac{(2n-2)!}{(n-1)!} \cdot (3n-2).
\end{equation*}
\end{theorem}

\begin{proof}
By the observations made above, the number of sketches equivalent to a given sketch only depends on its $\alpha,\beta$-word (see \Cref{absp}).
So, we need to count the number of $\alpha,\beta$-words of length $2n$ with any prefix having at least as many $\alpha$-letters as $\beta$-letters that are of the following types:
\begin{enumerate}
    \item The $n^{th}$ $\alpha$-letter is not in the $(2n-1)^{th}$ position and
    \begin{enumerate}
        \item the letter after the $n^{th}$ $\alpha$-letter is an $\alpha$.
        \item the letter after the $n^{th}$ $\alpha$-letter is a $\beta$.
    \end{enumerate}
    \item The $n^{th}$ $\alpha$-letter is in the $(2n-1)^{th}$ position.
\end{enumerate}

We first count the second type of $\alpha,\beta$-words.
If the $n^{th}$ $\alpha$-letter is in the $(2n-1)^{th}$ position, the first $(2n-2)$ letters have $(n-1)$ $\alpha$-letters and $(n-1)$ $\beta$-letters and hence form a ballot sequence. 
This means that there is no restriction on the $2n^{th}$ letter; it can be $\alpha$ or $\beta$. 
So, the total number of such $\alpha,\beta$-words is
\begin{equation*}
    2 \cdot \frac{1}{n}\binom{2n-2}{n-1}.
\end{equation*}

The number of both the types 1(a) and 1(b) of $\alpha,\beta$-words mentioned above are the same.
This is because changing the letter after the $n^{th}$ $\alpha$-letter is an involution on the set of $\alpha,\beta$-word of length $2n$ with any prefix having at least as many $\alpha$-letters as $\beta$-letters.
We have just counted such words that have the $n^{th}$ $\alpha$-letter in the $(2n-1)^{th}$ position. 
Hence, using \Cref{countabtypc}, we get that the number of words of type 1(a) and 1(b) are both equal to
\begin{equation*}
    \frac{1}{2} \cdot \left[\binom{2n}{n}-\frac{2}{n}\binom{2n-2}{n-1}\right].
\end{equation*}
Combining the observations made above, we get that the number of regions of $\D_n$ is
\begin{equation*}
    2^nn!\cdot \Bigg{(}\frac{1}{4} \cdot \Bigg{[}\frac{2}{n}\dbinom{2n-2}{n-1}+\frac{1}{2} \cdot \bigg{[}\dbinom{2n}{n}-\frac{2}{n}\dbinom{2n-2}{n-1}\bigg{]}\Bigg{]}
    + \frac{1}{2} \cdot \Bigg{[}\frac{1}{2} \cdot \bigg{[}\dbinom{2n}{n}-\frac{2}{n}\dbinom{2n-2}{n-1}\bigg{]}\Bigg{]}\Bigg{)}
\end{equation*}
which simplifies to the required formula.
\end{proof}

Just as we did for $\C_n$, we can describe and count which regions of $\D_n$ are bounded.

\begin{theorem}\label{typDbdd}
The number of bounded regions of $\D_n$ is
\begin{equation*}
    2^{n - 1} \cdot \frac{(2n - 3)!}{(n - 2)!} \cdot (3n - 4).
\end{equation*}
\end{theorem}

\begin{proof}
For $n \geq 2$, both $\C_n$ and $\D_n$ have rank $n$. 
Hence, a region of $\D_n$ is bounded exactly when all the regions of $\C_n$ it contains are bounded.

We have already seen in \Cref{typeCbdd} that a region of $\C_n$ is bounded exactly when its corresponding lattice path does not touch the $x$-axis except at the origin. 
Such regions are not closed under $\D$ moves. 
However, if we include regions whose corresponding lattice paths touch the $x$-axis only at the origin and $(2n,0)$, this set of regions, which we call $S$, is closed under the action of $\D$ moves because such lattice paths are closed under the action of changing the $2n^{th}$ step. 
Denote by $S_\D$ the set of equivalence classes that $\D$ moves partition $S$ into, i.e., $S_\D$ is the set of regions of $\D_n$ that contain regions of $S$.

Just as in the proof of \Cref{nooftypeDregions}, one can check that the set $S$ is closed under the action of changing the letter after the $n^{th}$ $\alpha$-letter. 
Also, note that the lattice paths in $S$ do not touch the $x$-axis at $(2n - 2,0)$, and hence the $n^{th}$ $\alpha$-letter cannot be in the $(2n - 1)^{th}$ position. 
Using the above observations and the same method to count regions of $\D_n$ as in the proof of \Cref{nooftypeDregions}, we get the number of regions in $S_\D$ is
\begin{equation*}
    2^n n! \cdot \frac{3}{8}\left(\binom{2n - 1}{n} + \frac{1}{n} \binom{2n - 2}{n - 1}\right).
\end{equation*}

It can also be checked that each unbounded region in $S$ is $\D$ equivalent to exactly one other region of $S$, and this region is bounded. 
This is because the lattice paths corresponding to these unbounded regions touch the $x$-axis at $(2n, 0)$. 
Hence, they cannot have the $n^{th}$ and $(n + 1)^{th}$ $\alpha$-letters being adjacent and changing the $2n^{th}$ letter to an $\alpha$ gives a bounded region. 
Since the unbounded regions in $S$ correspond to Dyck paths of length $(2n - 2)$ (by deleting the first and last step), we get that the number of unbounded regions in $S_\D$ is
\begin{equation*}
    2^n n! \cdot \frac{1}{n}\binom{2n - 2}{n - 1}.
\end{equation*}

Combining the above results, we get that the number of bounded regions of $\D_n$ is
\begin{equation*}
    2^n n! \left(\frac{3}{8}\left(\binom{2n - 1}{n} + \frac{1}{n} \binom{2n - 2}{n - 1}\right) - \frac{1}{n}\binom{2n - 2}{n - 1}\right).
\end{equation*}
This simplifies to give our required result.
\end{proof}

As mentioned earlier, we can choose a specific sketch from each $\D$ equivalence class to represent the regions of $\D_n$. 
It can be checked that symmetric sketches that satisfy the following are in bijection with regions of $\D_n$:
\begin{enumerate}
    \item The last letter is a $\beta$-letter.
    \item The $n^{th}$ $\alpha$-letter must have a negative label if the letter following it is an $\alpha$-letter or the $n^{th}$ $\beta$-letter.
\end{enumerate}
We will call such sketches type $D$ sketches. 
They will be used in \Cref{statsec} to interpret the coefficients of $\chi_{\D_n}$. 
Note that the type $D$ sketches that correspond to bounded regions of $\D_n$ are those, when converted to a lattice path, do not touch the $x$-axis apart from at the origin.

\subsection{Pointed type C Catalan}

The type $B$ and type $BC$ Catalan arrangements we are going to consider now are not sub-arrangements of the type $C$ Catalan arrangement. 
While it is possible to consider these arrangements as sub-arrangements of the type $C$ $2$-Catalan arrangement (see \Cref{extc}), this would add many extra hyperplanes. 
This would make defining moves and counting equivalence classes difficult. 
Also, we do not have a simple characterization of $\alpha, \beta$-words associated to symmetric $2$-sketches, as we do for symmetric sketches (see \Cref{absp}).

We instead consider them as a sub-arrangements of the arrangement $\Po_n$ in $\mathbb{R}^n$ that has hyperplanes
\begin{equation*}
    \begin{aligned}
    x_i&=-\frac{5}{2},-\frac{3}{2},-1,-\frac{1}{2},0,\frac{1}{2},\frac{3}{2}\\
    x_i+x_j&=-2,-1,0\\
    x_i-x_j&=-1,0,1
\end{aligned}
\end{equation*}
for all $1 \leq i< j \leq n$. 
% Rewriting the hyperplanes of the arrangement as
% \begin{align*}
%     &x_i+1=-\frac{3}{2},\ x_i=-\frac{3}{2},\ x_i+1=-x_i-1,\ x_i=-\frac{1}{2},\ x_i=-x_i,\ x_i=\frac{1}{2},\ x_i=\frac{3}{2}\\
%     &x_i+1=-x_j-1,\ x_i+1=-x_j,\ x_i=-x_j\\
%     &x_i+1=x_j,\ x_i=x_j,\ x_i=x_j+1
% \end{align*}
% for all $1 \leq i< j \leq n$, we can see that a region of this arrangement is given by a valid total order on
It can be checked that the regions of $\Po_n$ are given by valid total orders on
\begin{equation*}
    \{x_i+s \mid i \in [n],\ s\in \{0,1\}\} \cup \{-x_i-s \mid i \in [n],\ s\in \{0,1\}\} \cup \{-\frac{3}{2},-\frac{1}{2},\frac{1}{2},\frac{3}{2}\}.
\end{equation*}

\begin{remark}
The arrangement $\Po_n$ is the arrangement $\C_n(\lambda)$ defined in \cite[Equation (4)]{ath_non} with $\lambda_i = 2$ for all $i \in [n]$ and $m = 2$.
\end{remark}

We now define sketches that represent such orders. 
Just as beofre, we represent $x_i+s$ as \lt{i}{s} and $-x_i-s$ as \lt{-i}{-s} for any $i \in [n]$ and $s \in \{0,1\}$. 
The numbers $-\frac{3}{2},-\frac{1}{2},\frac{1}{2},\frac{3}{2}$ will be represented as \lt{-}{-1.5}, \lt{-}{-0.5}, \lt{+}{0.5}, \lt{+}{1.5} respectively.

\begin{example}
The total order
\begin{equation*}
    -\frac{3}{2}<x_2<-x_1-1<-\frac{1}{2}<x_1<x_2+1<-x_2-1<-x_1<\frac{1}{2}<x_1+1<-x_2<\frac{3}{2}
\end{equation*}
is represented as \lt{-}{-1.5} \lt{2}{0} \lt{-1}{-1} \lt{-}{-0.5} \lt{1}{0} \lt{2}{1} \lt{-2}{-1} \lt{-1}{0} \lt{+}{0.5} \lt{1}{1} \lt{-2}{0} \lt{+}{1.5}.
\end{example}

Set $B(n)$ to be the set
\begin{equation*}
    \{\lt{i}{s} \mid i\in[n],\ s\in\{0,1\}\} \cup \{\lt{i}{s} \mid -i\in[n],\ s\in\{-1,0\}\}  \cup \{\lt{-}{-1.5},\lt{-}{-0.5},\lt{+}{0.5},\lt{+}{1.5}\}.
\end{equation*}

We define \emph{pointed symmetric sketches} to be the words in $B(n)$ that correspond to regions of $\Po_n$ (this terminology will become clear soon). 
Denote by $\overline{\lt{x}{s}}$ the letter $\lt{-x}{-s}$ for any $\lt{x}{s} \in B(n)$. 
We have the following characterization of pointed symmetric sketches:

\begin{proposition}
A word in the letters $B(n)$ is a pointed symmetric sketch if and only if the following hold for any $\lt{x}{s}, \lt{y}{t} \in B(n)$:
\begin{enumerate}
    \item If $\alpha_x^{(s)}$ appears before $\alpha_y^{(t)}$ then $\overline{\alpha_y^{(t)}}$ appears before $\overline{\alpha_x^{(s)}}$.
    \item If $\alpha_x^{(s-1)}$ appears before $\alpha_y^{(t-1)}$ then $\alpha_x^{(s)}$ appears before $\alpha_y^{(t)}$.
    \item $\alpha_x^{(s-1)}$ appears before $\alpha_x^{(s)}$.
    \item Each letter of $B(n)$ appears exactly once.
\end{enumerate}
\end{proposition}

Just as was done in the proof of \Cref{pointconst}, we can inductively construct a point in $\R^n$ satisfying the inequalities specified by a pointed sketch. 
Also, just as for type $C$ sketches, it can be shown that these sketches are symmetric about the center. 
We also represent such sketches using arc diagrams in a similar manner. 
Note that in this case we also inlcude an arc between \lt{-}{-0.5} and \lt{+}{0.5}.

\begin{example}\label{valsketch}
To the pointed sketch given below, we associate the arc diagram in \Cref{symdiag}.
\begin{equation*}
    \lt{-}{-1.5}\ \lt{2}{0}\ \lt{-1}{-1}\ \lt{-}{-0.5}\ \lt{1}{0}\ \lt{2}{1}\ \textcolor{blue}{|}\ \lt{-2}{-1}\ \lt{-1}{0}\ \lt{+}{0.5}\ \lt{1}{1}\ \lt{-2}{0}\ \lt{+}{1.5}
\end{equation*}
\begin{figure}[H]
    \centering
    \begin{tikzpicture}
        \node (-6) at (-6+0.5,0) {$-$};
        \node (-5) at (-5+0.5,0) {$2$};
        \node (-4) at (-4+0.5,0) {$-1$};
        \node (-3) at (-3+0.5,0) {$-$};
        \node (-2) at (-2+0.5,0) {$1$};
        \node (-1) at (-1+0.45,0) {$2$};
        \draw [blue] (0,-0.5)--(0,0.5);
        \node (1) at (1-0.45,0) {$-2$};
        \node (2) at (2-0.5,0) {$-1$};
        \node (3) at (3-0.5,0) {$+$};
        \node (4) at (4-0.5,0) {$1$};
        \node (5) at (5-0.5,0) {$-2$};
        \node (6) at (6-0.5,0) {$+$};
        \draw (-6.north)..controls +(up:13mm) and +(up:13mm)..(-3.north);
        \draw (3.north)..controls +(up:25mm) and +(up:25mm)..(-3.north);
        \draw (6.north)..controls +(up:13mm) and +(up:13mm)..(3.north);
        \draw (-4.north)..controls +(up:20mm) and +(up:20mm)..(2.north);
        \draw (-2.north)..controls +(up:20mm) and +(up:20mm)..(4.north);
        \draw (-5.north)..controls +(up:15mm) and +(up:15mm)..(-1.north);
        \draw (5.north)..controls +(up:15mm) and +(up:15mm)..(1.north);
    \end{tikzpicture}
    \caption{Arc diagram associated to the pointed symmetric sketch in \Cref{valsketch}.}
    \label{symdiag}
\end{figure}
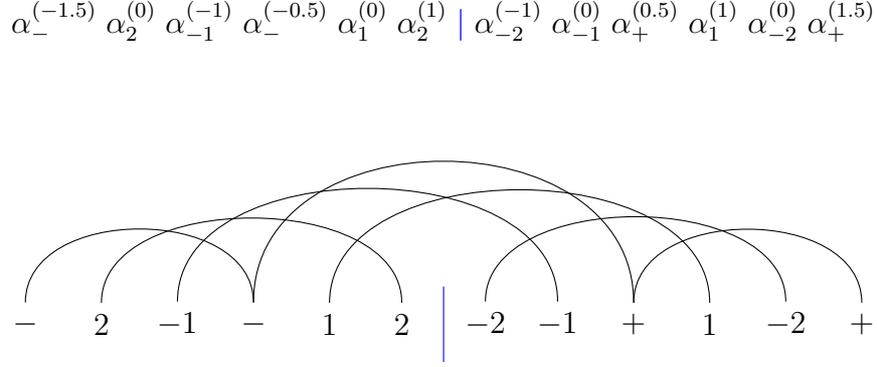
\end{example}

To a pointed symmetric sketch, we can associate a pointed $\alpha,\beta$-word of length $(2n+2)$ and a signed permutation as follows:
\begin{enumerate}
    \item For the letters in the first half of the pointed sketch of the form \lt{i}{0}, \lt{-i}{-1} or \lt{-}{-1.5}, we write $\alpha$ and for the others we write $\beta$ ($\alpha$ corresponds to `openers' in the arc diagram and $\beta$ to `closers'). 
    The $\beta$ corresponding to \lt{-}{0.5} is pointed to.
    
    \item The subscripts of the first $n$ $\alpha$-letters other than \lt{-}{-1.5} gives us the signed permutation.
\end{enumerate}

\begin{example}
To the pointed sketch in \Cref{valsketch}, we associate the following pair:
\begin{enumerate}
    \item Pointed $\alpha,\beta$-word: $\alpha\alpha\alpha\textcolor{red}{\beta}\alpha\beta$.
    \item Signed permutation: $2 \ \ -1$.
\end{enumerate}
\end{example}

As was done for symmetric sketches, we can see that the method given above to get a signed permutation does actually give a signed permutation. 
Also, such a pair has at most one pointed sketch associated to it. 
We now characterize the pointed $\alpha,\beta$-words and signed permutations associated to pointed sketches.

\begin{proposition}\label{pospoi}
A pair consisting of
\begin{enumerate}
    \item a pointed $\alpha,\beta$-word of length $(2n + 2)$ satisfying the property that in any prefix, there are at least as many $\alpha$-letters as $\beta$-letters and that the number of $\alpha$-letters before the pointed $\beta$ is $(n+1)$, and
    \item any signed permutation
\end{enumerate}
corresponds to a pointed symmetric sketch and all pointed sketches correspond to such pairs.
\end{proposition}

\begin{proof}
Most of the proof is the same as that for type $C$ sketches.
The main difference is pointing to the $\beta$-letter corresponding to \lt{-}{-0.5}.
The property we have to take care of is that there is no nesting in the arc joining \lt{-}{0.5} to \lt{+}{0.5}.
This is the same as specifying when an arc drawn from a $\beta$-letter in the first half to its mirror image in the second half does not cause any nesting.

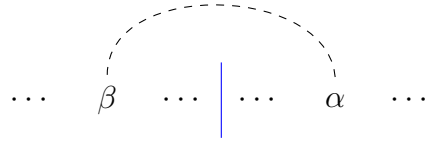
\begin{figure}[H]
    \centering
    \begin{tikzpicture}
        \node at (-2,0) {$\cdots$};
        \node (b) at (-1,0) {$\beta$};
        \node at (0,0) {$\cdots$};
        \draw [blue] (0.5,-0.5)--(0.5,0.5);
        \node at (1,0) {$\cdots$};
        \node (a) at (2,0) {$\alpha$};
        \node at (3,0) {$\cdots$};
        \draw [dashed](b.north)..controls +(up:13mm) and +(up:13mm)..(a.north);
    \end{tikzpicture}
    \caption{Arc from $\beta$ to its mirror image.}
\end{figure}

Denote by $N_{\alpha,b}$ the number of $\alpha$-letters before the $\beta$ under consideration, $N_{\alpha,a}$ the number of $\alpha$-letters in the first half after the $\beta$ and similarly define $N_{\beta,b}$ and $N_{\beta,a}$.
The condition that we do not want an arc inside the one joining the $\beta$ to its mirror is given by
\begin{equation*}
    N_{\alpha,b} \geq N_{\beta,b} + 1 + N_{\beta,a} + N_{\alpha,a}.
\end{equation*}
This is because of the symmetry of the arc diagram and the fact that we want any $\beta$-letter between the pointed $\beta$ and its mirror to have its corresponding $\alpha$ before the pointed $\beta$.
Similarly, the condition that we do not want the arc joining the $\beta$ to its mirror to be contained in any arc is given by
\begin{equation*}
    N_{\alpha,b} \leq N_{\beta,b} + 1 + N_{\beta,a} + N_{\alpha,a}.
\end{equation*}
This is because of the symmetry of the arc diagram and the fact that we want any $\alpha$-letter before the pointed $\beta$ to have its corresponding $\beta$ before the mirror of the pointed $\beta$.

Combining the above observations, we get
\begin{equation*}
    N_{\alpha,b} = N_{\beta,b} + 1 + N_{\beta,a} + N_{\alpha,a}.
\end{equation*}
But this says that the number of $\alpha$-letters before the pointed $\beta$ should be equal to the number of remaining letters in the first half.
Since the total number of letters in the first half is $(2n+2)$, we get that the arc joining a $\beta$ in the first half to its mirror does not cause nesting problems if and only if the number of $\alpha$-letters before it is $(n+1)$.
\end{proof}

Just as we used lattice paths for symmetric sketches, we use pointed lattice paths to represent pointed symmetric sketches. 
The one corresponding to the sketch in \Cref{valsketch} is given in \Cref{pointedlattice}.

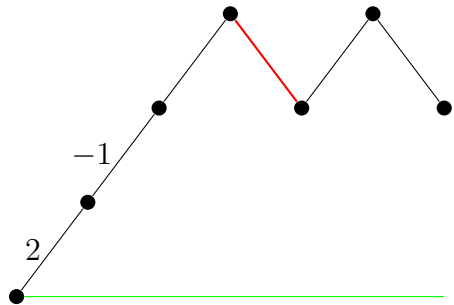
\begin{figure}[H]
    \centering
    \begin{tikzpicture}[xscale=0.75,scale=1.25]
    \draw [thin,green] (0,0)--(6,0);
    \node (0) [circle,inner sep=2pt,fill=black] at (0,0) {};
    \node (1) [circle,inner sep=2pt,fill=black] at (1,1) {};
    \node (2) [circle,inner sep=2pt,fill=black] at (2,2) {};
    \node (3) [circle,inner sep=2pt,fill=black] at (3,3) {};
    \node (4) [circle,inner sep=2pt,fill=black] at (4,2) {};
    \node (5) [circle,inner sep=2pt,fill=black] at (5,3) {};
    \node (6) [circle,inner sep=2pt,fill=black] at (6,2) {};
    \draw (0)--node[left]{$2$}(1)--node[left]{$-1$}(2)--(3);
    \draw[thick, color = red] (3)--(4);
    \draw (4)--(5)--(6);
    \end{tikzpicture}
    \caption{Pointed lattice path corresponding to the pointed sketch in \Cref{valsketch}.}
    \label{pointedlattice}
\end{figure}

\begin{theorem}\label{pointedsymcount}
The number of pointed symmetric sketches, which is the number of regions of $\Po_n$, is
\begin{equation*}
    2^n n! \binom{2n + 2}{n}.
\end{equation*}
\end{theorem}

\begin{proof}
Since there is no condition on the signed permutations, we just have to count the $\alpha,\beta$-words of the form mentioned in \Cref{pospoi}. 
We show that these words are in bijection with $\alpha,\beta$-words of length $(2n + 2)$ with any prefix having at least as many $\alpha$-letters as $\beta$-letters that have at least $(n + 2)$ $\alpha$-letters. 
This means that their corresponding lattice paths do not end on the $x$-axis. 
This will prove the required result since the number of such words, using \Cref{countabtypc} and the fact that Catalan numbers count Dyck paths, is
\begin{equation*}
    \binom{2n + 2}{n + 1} - \frac{1}{n + 2}\binom{2n + 2}{n + 1} = \binom{2n + 2}{n}.
\end{equation*}

Given a pointed $\alpha,\beta$-word, we replace the pointed $\beta$-letter with an $\alpha$-letter to obtain an $\alpha,\beta$-word of the type described above. 
Starting with an $\alpha,\beta$-word with at least $(n + 2)$ $\alpha$-letters, changing the $(n + 2)^{th}$ $\alpha$-letter to a $\beta$ and pointing to it gives a pointed $\alpha,\beta$-word. 
This gives us the required bijection.
\end{proof}

\begin{theorem}\label{bhabdd}
The number of bounded regions of $\Po_n$ is
\begin{equation*}
    2^n n! \binom{2n + 1}{n + 1}.
\end{equation*}
\end{theorem}

\begin{proof}
Just as for type $C$ regions, the region corresponding to a pointed sketch is bounded if and only if its arc diagram is interlinked. 
Also, the signed permutation does not play a role in determining if a region is bounded. 
Note that in this case, there is an arc joining a $\beta$-letter between the $(n + 1)^{th}$ and $(n + 2)^{th}$ $\alpha$-letter to its mirror image. 
If the arc diagram obtained by deleting this arc from the pointed $\beta$-letter is interlinked, then clearly so was the initial arc diagram. 
However, even if the arc diagram consists of two interlinked pieces when the arc from the pointed $\beta$-letter is removed (one on either side of the reflecting line), the corresponding region would still be bounded. 
Examining the bijection between arc diagrams and lattice paths, it can be checked that this means that pointed lattice paths corresponding to bounded regions are those that never touch the $x$-axis after the origin except maybe at $(2n + 2, 0)$.

Using the bijection mentioned in the proof of \Cref{pointedsymcount}, we can see that the pointed $\alpha,\beta$-words corresponding to bounded regions are in bijection $\alpha,\beta$-words whose lattice paths never touch the $x$-axis after the origin. 
We have already counted such paths in \Cref{typeCbdd} and their number is
\begin{equation*}
    \binom{2n + 1}{n + 1}.
\end{equation*}
This gives the required result.
\end{proof}

\subsection{Type B Catalan}

Fix $n \geq 1$. 
The type $B$ Catalan arrangement in $\mathbb{R}^n$ has the hyperplanes
\begin{align*}
    X_i&=-1,0,1\\
    X_i+X_j&=-1,0,1\\
    X_i-X_j&=-1,0,1
\end{align*}
for all $1 \leq i < j \leq n$. 
Translating this arrangement by setting $X_i=x_i+\frac{1}{2}$, we get the arrangement $\B_n$ with hyperplanes
\begin{align*}
    x_i&=-\frac{3}{2},-\frac{1}{2},\frac{1}{2}\\
    x_i+x_j&=-2,-1,0\\
    x_i-x_j&=-1,0,1
\end{align*}
for all $1 \leq i< j \leq n$. 
We consider $\B_n$ as a sub-arrangement of $\Po_n$. 
The hyperplanes missing from $\Po_n$ are
\begin{align*}
    x_i=-\frac{5}{2},-1,0,\frac{3}{2}
\end{align*}
for all $i \in [n]$.
Hence the moves on pointed sketches corresponding to changing one of the inequalities associated to these hyperplanes are as follows:
\begin{enumerate}
    \item Corresponding to $x_i=0$, $x_i=-1$: Swapping to $(2n+2)^{th}$ and $(2n+3)^{th}$ letter if they are not \lt{-}{-0.5} and \lt{+}{0.5}.
    
    \item Corresponding to $x_i=-\frac{5}{2}$, $x_i=\frac{3}{2}$: Swapping the pointed $\beta$, that is, \lt{-}{-0.5} and a $\beta$-letter immediately before or after it (and making the corresponding change in the second half).
\end{enumerate}

We can see that such moves change the pointed $\alpha,\beta$-word associated to a sketch by at most changing the last letter or changing which of the $\beta$-letters between the $(n+1)^{th}$ and $(n+2)^{th}$ $\alpha$-letter is pointed to.
So if we force that the last letter of the sketch has to be a $\beta$-letter and that the $\beta$-letter immediately after the $(n+1)^{th}$ $\alpha$-letter has to be pointed to, we get a canonical sketch in each equivalence class.
We will call such sketches type $B$ sketches.

\begin{theorem}
The number of type $B$ sketches, which is the number of regions of $\B_n$, is
\begin{equation*}
    2^nn!\binom{2n}{n}.
\end{equation*}
\end{theorem}

\begin{proof}
Since there is no condition on the signed permutation, we count the $\alpha,\beta$-words associated to type $B$ sketches.
From \Cref{pospoi}, we can see that the $\alpha,\beta$-words we need to count are those that satisfy the following properties:

\begin{enumerate}
    \item Length of the word is $(2n+2)$.
    \item In any prefix, there are at least as many $\alpha$-letters as $\beta$-letters.
    \item The letter immediately after the $(n+1)^{th}$ $\alpha$-letter is a $\beta$ (pointed $\beta$).
    \item The last letter is a $\beta$.
\end{enumerate}

We exhibit a bijection between these words and $\alpha,\beta$-words of length $2n$ that satisfy property 2. 
We already know, from \Cref{countabtypc}, that the number of such words is $\binom{2n}{n}$ and so this will prove the required result.

If the $(n+1)^{th}$ $\alpha$-letter is at the $(2n+1)^{th}$ position, deleting the last two letters gives us an $\alpha,\beta$-word of length $2n$ with $n$ $\alpha$-letters that satisfies property 2.
If the $(n+1)^{th}$ $\alpha$-letter is not at the $(2n+1)^{th}$ position, we delete the $\beta$-letter after it as well as the last letter of the word. 
This gives us an $\alpha,\beta$-word of length $2n$ with more than $n$ $\alpha$-letters that satisfies property 2.
The process described gives us the required bijection.
\end{proof}

\begin{theorem}\label{typBbdd}
The number of bounded regions of $\B_n$ is
\begin{equation*}
    2^nn!\binom{2n - 1}{n}.
\end{equation*}
\end{theorem}

\begin{proof}
Both $\B_n$ and $\Po_n$ have rank $n$. 
Hence a region of $\B_n$ if bounded if and only if all regions of $\Po_n$ that it contains are bounded.

In the proof of \Cref{bhabdd} we have characterized the pointed $\alpha,\beta$-words associated to bounded regions of $\Po_n$. 
These are the pointed lattice paths of length $(2n + 2)$ that satisfy the following properties (irrespective of the position of the pointed $\beta$):
\begin{enumerate}
    \item The step after the $(n + 1)^{th}$ up-step is a down step (for there to exist a pointed $\beta$).
    \item The path never touches the $x$-axis after the origin expect maybe at $(2n + 2, 0)$.
\end{enumerate}

We noted in \Cref{typDbdd} that lattice paths satisfying property 2 are closed under action of changing the letter after the $(n + 1)^{th}$ up-step as well as the action of changing the last step. 
This shows that the regions of $\Po_n$ that lie inside a region of $\B_n$ are either all bounded or all unbounded. 
Hence the number of bounded regions of $\B_n$ is just the number of type $B$ sketches whose corresponding lattice path satify property 1 and 2, which is
\begin{equation*}
    2^n n! \cdot \frac{1}{4} \cdot \left( \binom{2n + 1}{n + 1} + \frac{1}{n + 1} \binom{2n}{n} \right).
\end{equation*}
This simplifies to give the required result.
% \textcolor{red}{K: Consider the full total order arrangement. 
% Look at arc diagram. 
% Break up into interlinked parts ignoring the rigid arc (between $\pm 0.5$). 
% If only one part, already bounded. 
% If not, then it can have at most one (total 2 because of reflection) with the rigid arc causing the boundedness. 
% So we can see that all arcs corresponding to a given type B region are either bounded or all unbounded. 
% So we just have to count the number of $\alpha, \beta$-words of length $(2n + 2)$ where
% \begin{enumerate}
%     \item the letter after $(n + 1)^{th}$ $\alpha$ is a $\beta$,
%     \item last letter is $\beta$, and
%     \item the path either never touches the $x$-axis after the origin or only at $(2n + 2, 0)$.
% \end{enumerate}
% We have already seen in type $D$ bounded count that changing the last letter (respectively, the letter after $(n + 1)^{th}$ $\alpha$) in words that satisfy property 3 is an involution. 
% Hence the number of bounded regions is
% \begin{equation*}
%     2^n n! \cdot \frac{1}{4} \cdot \left( \binom{2n + 1}{n + 1} + \frac{1}{n + 1} \binom{2n}{n} \right)
% \end{equation*}
% which simplifies to give the required result.
% }
\end{proof}

\subsection{Type BC Catalan}

The type $BC$ Catalan arrangement in $\mathbb{R}^n$ has hyperplanes
\begin{align*}
    X_i&=-1,0,1\\
    2X_i&=-1,0,1\\
    X_i+X_j&=-1,0,1\\
    X_i-X_j&=-1,0,1
\end{align*}
for all $1 \leq i < j \leq n$. 
Translating this arrangement by setting $X_i=x_i+\frac{1}{2}$, we get the arrangement $\BC_n$ with hyperplanes
\begin{align*}
    x_i&=-\frac{3}{2},-1,-\frac{1}{2},0,\frac{1}{2}\\
    x_i+x_j&=-2,-1,0\\
    x_i-x_j&=-1,0,1
\end{align*}
for all $1 \leq i< j \leq n$.
Again, we consider this arrangement as a sub-arrangement of $\Po_n$.
To define moves on pointed sketches, note that the hyperplanes missing from $\Po_n$ are
\begin{align*}
    x_i=-\frac{5}{2},\frac{3}{2}
\end{align*}
for all $i \in [n]$.
Hence, the moves on pointed sketches corresponding to changing the inequalities associated to these hyperplanes are of the following form:
Swapping the pointed $\beta$, that is, \lt{-}{-0.5} and a $\beta$-letter immediately before or after it (and making the corresponding change in the second half).

We can see that such moves change the pointed $\alpha,\beta$-word associated to a sketch by at most changing which of the $\beta$-letters between the $(n+1)^{th}$ and $(n+2)^{th}$ $\alpha$-letter is pointed to.
So if we force that the $\beta$-letter immediately after the $(n+1)^{th}$ $\alpha$-letter has to be pointed to, we get a canonical sketch in each equivalence class.
We will call such sketches type $BC$ sketches.

\begin{theorem}
The number of type $BC$ sketches, which is the number of regions $\BC_n$ is
\begin{equation*}
    2^{n-1}n!\binom{2n+2}{n+1}.
\end{equation*}
\end{theorem}

\begin{proof}
Since there is no condition on the signed permutation for type $BC$ sketches, we count the number of $\alpha,\beta$-words that satisfy the following properties:
\begin{enumerate}
    \item Length of the word is $(2n+2)$.
    \item In any prefix, there are at least as many $\alpha$-letters as $\beta$-letters.
    \item The letter immediately after the $(n+1)^{th}$ $\alpha$-letter is a $\beta$ (pointed $\beta$).
\end{enumerate}

Using the involution on the set of words satisfying properties 1 and 2 of changing the letter immediately after the $(n+1)^{th}$ $\alpha$-letter and the fact that there are $\binom{2n+2}{n+1}$ words satisfying properties 1 and 2, we get that the number of words satisfying the required properties is
\begin{equation*}
    \frac{1}{2} \cdot \binom{2n+2}{n+1}.
\end{equation*}
This gives the required result.
\end{proof}

\begin{theorem}
The number of bounded regions $\BC_n$ is
\begin{equation*}
    2^nn!\binom{2n}{n}.
\end{equation*}
\end{theorem}

\begin{proof}
The proof of this result is very similar to that of \Cref{typBbdd}. 
Since type $BC$ sketches don't have the condition that the $2n^{th}$ letter should be a $\beta$-letter, the number of bounded regions of $\BC_n$ is
\begin{equation*}
    2^n n! \cdot \frac{1}{2} \cdot \left( \binom{2n + 1}{n + 1} + \frac{1}{n + 1} \binom{2n}{n} \right).
\end{equation*}
This simplifies to give the required result.
\end{proof}

\section{Statistics on regions via generating functions}\label{statsec}

As mentioned in \Cref{intro}, the characteristic polynomial of an arrangement $\A$ in $\R^n$ is of the form
\begin{equation*}
    \chi_\A(t) = \sum_{i=0}^n (-1)^{n-i} c_i t^i
\end{equation*}
where $c_i$ is a non-negative integer for all $0 \leq i \leq n$ and Zaslavsky's theorem tells us that
\begin{align*}
   r(\A) &= (-1)^n \chi_{\A}(-1) \\
         &= \sum_{i=0}^n c_i. 
\end{align*}
In this section, we interpret the coefficients of the characteristic polynomials of the arrangements we have studied. 
More precisely, for each arrangement we have studied, we first define a statistic on the objects that we have seen correspond to its regions. 
We then show that the distribution of this statistic is given by the coefficients of the characteristic polynomial.

We do this by giving combinatorial meaning to the exponential generating functions for the characteristic polynomials of the arrangements we have studied. 
To obtain these generating functions, we use \cite[Exercise 5.10]{sta_hyp}, which we state and prove for convenience.

\begin{definition}
A sequence of arrangements $(\A_1,\A_2,\ldots)$ is called a \emph{Generalized Exponential Sequence of Arrangements} (GESA) if
\begin{itemize}
    \item $\A_n$ is an arrangement in $\R^n$ such that every hyperplane is parallel to one of the form $x_i = cx_j$ for some $c \in \R$.
    \item For any $k$-subset $I$ of $[n]$, the arrangement
    \begin{equation*}
        \A_n^I = \{H \in \A_n \mid H \text{ is parallel to $x_i = cx_j$ for some $i,j \in I$ and some $c \in \R$}\}
    \end{equation*}
    satisfies $\ipa(\A_n^I) \cong \ipa(\A_k)$ (isomorphic as posets).
\end{itemize}
\end{definition}

Note that all the arrangements we have studied are GESAs.

\begin{proposition}\label{gesa}
Let $(\A_1,\A_2,\ldots)$ be a GESA, and define
\begin{equation*}
\begin{split}
    F(x) &= \sum_{n \geq 0} (-1)^n r(\A_n) \frac{x^n}{n!}\\
    G(x) &= \sum_{n \geq 0} (-1)^{\operatorname{rank}(\A_n)} b(\A_n) \frac{x^n}{n!}.
\end{split}
\end{equation*}
Then, we have
\begin{equation*}
    \sum_{n \geq 0} \chi_{\A_n}(t) \frac{x^n}{n!} = \frac{G(x)^{(t + 1) / 2}}{F(x)^{(t - 1) / 2}}.
\end{equation*}
\end{proposition}

\begin{proof}
The idea of the proof is the same as that of \cite[Theorem 5.17]{sta_hyp}. 
By Whitney's Theorem \cite[Theorem 2.4]{sta_hyp}, we have for all $n$,
\begin{equation*}
    \chi_{\A_n}(t) = \sum_{\B \subseteq \A,\ \bigcap \B \neq \phi} (-1)^{\# \B} t^{n - \operatorname{rank}(\B)}.
\end{equation*}
To each $\B \subseteq \A_n$, such that $\bigcap \B \neq \phi$, we associate a graph $G(\B)$ on the vertex set $[n]$ where there is an edge between the vertices $i$ and $j$ if there is a hyperplane in $\B$ parallel to a hyperplane of the form $x_i = cx_j$ for some $c \in \R$.

Using \cite[Corollary 5.1.6]{sta_ec2}, we get
\begin{equation*}
    \sum_{n \geq 0} \chi_{\A_n}(t) \frac{x^n}{n!} = \operatorname{exp} \sum_{n \geq 1} \tilde{\chi}_{\A_n}(t) \frac{x^n}{n!}
\end{equation*}
where for any $n$ we define
\begin{equation*}
    \tilde{\chi}_{\A_n}(t) = \sum_{\substack{\B \subseteq \A,\ \bigcap \B \neq \phi\\G(\B)\text{ connected}}} (-1)^{\# \B} t^{n - \operatorname{rank}(\B)}.
\end{equation*}

Note that if $G(\B)$ is connected, then any point in $\bigcap \B$ is determined by any one of its coordinates, say $x_1$. 
This is because any path from the vertex $1$ to a vertex $i$ in $G(\B)$ can be used to determine $x_i$. 
This shows us that $\operatorname{rank}(\B)$ is either $n$ or $n - 1$. 
Hence, $\tilde{\chi_{\A_n}}(t) = c_nt + d_n$ for some $c_n, d_n \in \Z$. 
Setting
\begin{equation*}
\begin{split}
    \operatorname{exp}\sum_{n \geq 1} c_n \frac{x^n}{n!} &= \sum_{n \geq 0} b_n \frac{x^n}{n!}\\
    \operatorname{exp}\sum_{n \geq 1} d_n \frac{x^n}{n!} &= \sum_{n \geq 0} a_n \frac{x^n}{n!}\\
\end{split}
\end{equation*}
we get
\begin{equation*}
    \sum_{n \geq 0} \chi_{\A_n}(t) \frac{x^n}{n!} = \left(\sum_{n \geq 0} b_n \frac{x^n}{n!}\right)^{t}\left(\sum_{n \geq 0} a_n \frac{x^n}{n!}\right).
\end{equation*}
% \textcolor{red}{K: the above expression looks more statistic-oriented than the final form. Maybe try and see if this can be used instead?}
Substituting $t = 1$ and $t = -1$ and using \Cref{zaslavsky}, we obtain expressions for the exponential generating functions of $\{b_n\}$ and $\{c_n\}$ and this gives us the required result.
\end{proof}

Before looking at the characteristic polynomials of these arrangements, we recall a few results from \cite{sta_ec2}. 
Suppose that $c: \N \rightarrow \N$ is a function and for each $n,j \in \N$, we define
\begin{equation*}
    c_j(n)=\sum_{\{B_1,\ldots,B_j\} \in \Pi_n} c(|B_1|)\cdots c(|B_j|)
\end{equation*}
where $\Pi_n$ is the set of partitions of $[n]$.
Define for each $n \in \N$,
\begin{equation*}
    h(n) = \sum_{j=0}^{n}c_j(n).
\end{equation*}
From \cite[Example 5.2.2]{sta_ec2}, we know that in such a situation,
\begin{equation*}
    \sum_{n,j\geq 0}c_j(n)t^j\frac{x^n}{n!} = \left(\sum_{n\geq 0}h(n)\frac{x^n}{n!}\right)^t.
\end{equation*}
Informally, we consider $h(n)$ to be the number of ``structures'' that can be placed on an $n$-set where each structure can be uniquely broken up into a disjoint union of ``connected sub-structures''. 
Here $c(n)$ denotes the number of connected structures on an $n$-set and $c_j(n)$ denotes the number of structures on an $n$-set with exactly $j$ connected sub-structures. 
We will call such structures \emph{exponential structures}.

In fact, in most of the computations below, we will be dealing with generating functions of the form
\begin{equation}\label{genfunint}
    \left(\sum_{n\geq 0}h(n)\frac{x^n}{n!}\right)^{\frac{t + 1}{2}}.
\end{equation}
We can interpret such a generating function as follows. 
Suppose that there are two types of connected structures, say positive and negative connected structures. 
Also, suppose that the number of positive connected structures on $[n]$ is the same as the number of negative ones, i.e., $c(n)/2$. 
Then the coefficient of $t^j\frac{x^n}{n!}$ in the generating function given above is the number of structures on $[n]$ that have $j$ positive connected sub-structures.

Also, note that since the coefficients of the characteristic polynomial alternate in sign, the distribution of any appropriate statistic we define would be
\begin{equation*}
    \sum_{n \geq 0} \chi_{\A_n}(-t) \frac{(-x)^n}{n!}.
\end{equation*}

\subsection{Reflection arrangements}

Before defining statistics for the Catalan arrangements, we first do so for the reflection arrangements we studied in \Cref{refarrsec}. 
As we will see, the same statistic we define for sketches (regions of the type $C$ arrangement) works for the canonical sketches we have chosen for the other arrangements as well.
% We also note that the following results can be proved by directly looking at the coefficients of the characteristic polynomials as in \cite{thresh}.

\subsubsection{The type C arrangement}\label{typeCstat}

We have seen that the regions of the type $C$ arrangement in $\R^n$ correspond to sketches (\Cref{typeC}) of length $2n$. 
We use the second half of the sketch to represent the regions, and call them signed permutations on $[n]$.

A statistic on signed permutations whose distribution is given by the coefficients of the characteristic polynomial is given in \cite[Section 2]{thresh}. 
We define a similar statistic. 
First break the signed permutation into \emph{compartments} using right-to-left minima as follows: 
Ignoring the signs, draw a line before the permutation and then repeatedly draw a line immediately following the least number after the last line drawn. 
This is repeated until a line is drawn at the end of the permutation. 
It can be checked that compartments give signed permutations an exponential structure. 
A \emph{positive compartment} of a signed permutations is one where the last term is positive.

\begin{example}
The signed permutation given by
\begin{equation*}
    \overset{+} 3 \  \overset{+} 1 \  \overset{-} 6 \  \overset{-} 7 \  \overset{-} 5 \  \overset{+} 2 \  \overset{-} 4
\end{equation*}
is split into compartments as
\begin{equation*}
    \textcolor{blue}{|} \  \overset{+} 3 \  \overset{+} 1 \  \textcolor{blue}{|} \  \overset{-} 6 \  \overset{-} 7 \  \overset{-} 5 \  \overset{+} 2 \  \textcolor{blue}{|} \  \overset{-} 4 \  \textcolor{blue}{|}
\end{equation*}
and hence has $3$ compartments, $2$ of which are positive.
\end{example}

By the combinatorial interpretation of \eqref{genfunint}, the distribution of the statistic `number of positive compartments' on signed permutations is given by
\begin{equation*}
    \left(\frac{1}{1 - 2x}\right)^{\frac{t + 1}{2}}.
\end{equation*}
Note that for the type $C$ arrangement, in terms of \Cref{gesa}, we have
\begin{align*}
    F(x) &= \left(\frac{1}{1 + 2x}\right),\\
    G(x) &= 1.
\end{align*}
Hence, we get that the distribution of the statistic `number of positive compartments' on signed permutations is given by the coefficients of the characteristic polynomial.

% Setting $c(n, j)$ to be the number of signed permutations of $[n]$ with $j$ odd compartments, we have
% \begin{equation*}
%     c(n, j) = (2n - 1) \cdot c(n - 1, j) + c(n - 1, j - 1)
% \end{equation*}
% for all $n, j \geq 0$ with initial condition $c(0, 0) = 1$. 
% This can be shown by deleting the first element of the signed permutation. 
% This same recurrence is satisfied by the absolute values of the coefficients of the characteristic polynomial.

For the arrangements that follow, we have described canonical sketches and hence signed permutations that correspond to regions in \Cref{refarrsec}. 
For each arrangement, we will show that the distribution of the statistic `number of positive compartments' on these canonical signed permutations is given by the characteristic polynomial.

\subsubsection{The Boolean arrangement}

The signed permutations that correspond to the Boolean arrangement in $\R^n$ (\Cref{bool}) are those that have all compartments of size $1$, i.e., the underlying permutation is $1\ 2\ \cdots\ n$. 
Just as before, it can be seen that the distribution of the statistic `number of positive compartments' on such signed permutations is given by
\begin{equation*}
    (e^{2x})^{\frac{t + 1}{2}}.
\end{equation*}
This agrees with the generating function for the characteristic polynomial we get from \Cref{gesa} using $F(x) = e^{-2x}$ and $G(x) = 1$.

% The regions of the boolean arrangement in $\R^n$ correspond to the subsets of $[n]$ (\Cref{bool}). 
% Its characteristic polynomial is $(t - 1)^n$ \cite[Proposition 1.2]{sta_hyp}. 
% Hence the absolute value of the coefficient of $t^j$ is $\binom{n}{j}$. 
% This shows that the distribution of the statistic `number of elements' on the subsets of $[n]$ is given by the coefficients of the characteristic polynomial.

\subsubsection{The type D arrangement}

From \Cref{typeD}, we can see that the regions of the type $D$ arrangement in $\R^n$ correspond to signed permutations on $[n]$ where the first sign is positive. 
Given $i \in [n]$ and a signed permutation $\sigma$ of $[n] \setminus \{i\}$, the signed permutation of $[n]$ obtained by appending $\overset{-}{i}$ to the start of $\sigma$ has the same number of positive compartments as $\sigma$. 
This shows that the distribution of the statistic on signed permutations whose first term is positive is
\begin{equation*}
    (1 - x) \left(\frac{1}{1 - 2x}\right)^{\frac{t + 1}{2}}.
\end{equation*}
This agrees with the generating function for the characteristic polynomial we get from \Cref{gesa} since we have
\begin{align*}
    F(x) &= \left(\frac{1 + x}{1 + 2x}\right),\\
    G(x) &= 1 + x.
\end{align*}
Note that the expression for $G(x)$ is due to the fact that the type $D$ arrangement in $\R^1$ is empty.
% The characteristic polynomial \cite[Section 5.1]{sta_hyp} is
% \begin{equation*}
%     (t-1)(t-3) \cdots (t-(2n-3)) \cdot (t - (n - 1)).
% \end{equation*}
% Comparing this with the characteristic polynomial of the type $C$ arrangement, we get that the absolute value of the coefficient of $t^j$ is
% \begin{equation*}
%     c(n - 1, j - 1) + (n - 1) \cdot c(n - 1, j).
% \end{equation*}
% The number of signed permutations with $j$ odd compartments where the first term is $\overset{-}{1}$ is $c(n - 1, j - 1)$. 
% The number of remaining signed permutations with $j$ odd compartments whose first term is negative is $(n - 1) \cdot c(n - 1, j)$. 
% This shows that the distribution of the statistic `number of odd compartments' on signed permutations where the first sign is negative is given by the coefficients of the characteristic polynomial.

\subsubsection{The braid arrangement}

From \Cref{braid}, we get that the regions of the brain arrangement in $\R^n$ corresponds to the signed permutations on $[n]$ where all terms are positive. 
Hence, the number of positive compartments is just the number of compartments in the underlying permutation. 
Since compartments give permutations an exponential structure, the distribution of this statistic is
\begin{equation*}
    \left(\frac{1}{1 - x}\right)^t.
\end{equation*}
This agrees with the generating function for the characteristic polynomial we get from \Cref{gesa} since $F(x) = \frac{1}{1 + x}$ and $G(x) = 1 + x$.
% The regions of the braid arrangement in $\R^n$ (\Cref{braid}) correspond to the $n!$ permutations of $[n]$. 
% As mentioned in \Cref{intro}, the distribution of the statistic `number of cycles' on the set of permutations is given by the coefficients of the characteristic polynomial.

We summarize the results of this section as follows. 
For any reflection arrangement $\A$, we use $\A$-signed permutation to mean those described above to represent the regions of $\A$.

\begin{theorem}
    For any reflection arrangement $\A$, the absolute value of the coefficient of $t^j$ in $\chi_\A(t)$ is the number of $\A$-signed permutations that have $j$ positive compartments.
\end{theorem}

\subsection{Catalan deformations}

We start with defining a statistic for the extended type $C$ Catalan arrangements. 
Using \Cref{gesa}, we then show that the generating function for the statistic and the characteristic polynomials match.

Fix $m \geq 1$. 
We define a statistic on labeled symmetric non-nesting partitions and show that its distribution is given by the characteristic polynomial. 
To do this, we first recall some definitions and results about the type $A$ extended Catalan arrangement.

\begin{definition}
An \emph{$m$-non-nesting partition of size $n$} is a partition of $[(m + 1)n]$ such that the following hold:
\begin{enumerate}
    \item Each block is of size $(m + 1)$.
    
    \item If $a, b$ are in the same block $B$ and $[a, b] \cap B = \{a, b\}$, then for any $c, d$ such that $a < c < d < b$, $c$ and $d$ are not in the same block.
\end{enumerate}
\end{definition}

Just as before, such partitions can be represented using arc diagrams.

\begin{example}\label{nonnestex}
The arc diagram corresponding to the $2$-non-nesting partition of size $3$
\begin{equation*}
    \{1, 2, 4\}, \{3, 5, 6\}, \{7, 8, 9\}
\end{equation*}
is given in \Cref{nonnestfig}.
\end{example}

\begin{figure}[H]
    \centering
    \begin{tikzpicture}
        \node [circle,inner sep=2pt,fill=black] (-6) at (-6+0.5,0) {};
        \node [circle,inner sep=2pt,fill=black] (-5) at (-5+0.5,0) {};
        \node [circle,inner sep=2pt,fill=black] (-4) at (-4+0.5,0) {};
        \node [circle,inner sep=2pt,fill=black] (-3) at (-3+0.5,0) {};
        \node [circle,inner sep=2pt,fill=black] (-2) at (-2+0.5,0) {};
        \node [circle,inner sep=2pt,fill=black] (-1) at (-1+0.5,0) {};
        \node [circle,inner sep=2pt,fill=black] (1) at (0+0.5,0) {};
        \node [circle,inner sep=2pt,fill=black] (2) at (1+0.5,0) {};
        \node [circle,inner sep=2pt,fill=black] (3) at (2+0.45,0) {};
        
        \draw (-6.north)..controls +(up:7mm) and +(up:7mm)..(-5.north);
        \draw (-5.north)..controls +(up:13mm) and +(up:13mm)..(-3.north);
        \draw (-4.north)..controls +(up:13mm) and +(up:13mm)..(-2.north);
        \draw (-2.north)..controls +(up:7mm) and +(up:7mm)..(-1.north);
        
        \draw (1.north)..controls +(up:7mm) and +(up:7mm)..(2.north);
        \draw (2.north)..controls +(up:7mm) and +(up:7mm)..(3.north);
    \end{tikzpicture}
    \caption{Arc diagram corresponding to the $2$-non-nesting partition in \Cref{nonnestex}.}
    \label{nonnestfig}
\end{figure}
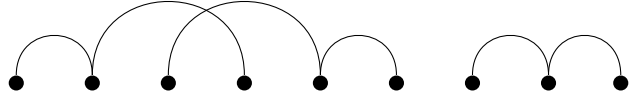

It is known (for example, see \cite[Theorem 2.2]{ath_non}) that the number of $m$-non-nesting partitions of size $n$ is
\begin{equation*}
    \frac{1}{mn + 1} \binom{(m + 1)n}{n}.
\end{equation*}
These numbers are called the Fuss-Catalan numbers or generalized Catalan numbers. 
Setting $m = 1$ gives us the usual Catalan numbers. 
Labeling the $n$ blocks distinctly using $[n]$ gives us labeled $m$-non-nesting partitions. 
These objects correspond to the regions of the type $A$ $m$-Catalan arrangement in $\R^n$ whose hyperplanes are
\begin{equation*}
    x_i - x_j = 0, \pm 1, \pm 2, \ldots, \pm m
\end{equation*}
for all $1 \leq i < j \leq n$ (for example, see \cite[Section 8.1]{ber}).

We now define a statistic on labeled non-nesting partitions similar to the one defined in \cite[Section 4]{branch}. 
The statistic defined in \cite{branch} is for labeled $m$-Dyck paths but these objects are in bijection with labeled $m$-non-nesting partitions.

A labeled non-nesting partition can be broken up into interlinked pieces, say $P_1, P_2, \ldots, P_k$. 
We group these pieces into \emph{compartments} as follows. 
If the label $1$ is in the $r^{th}$ interlinked piece, then the interlinked pieces $P_1, P_2, \ldots, P_r$ form the first compartment. 
Let $j$ be the smallest number in $[n] \setminus A$ where $A$ is the set of labels in first compartment. 
If $j$ is in the $s^{th}$ interlinked piece then interlinked pieces $P_{r + 1}, P_{r + 2}, \ldots, P_s$ form the second compartment. 
Continuing this way, we break up a labeled non-nesting partition into compartments.

\begin{example}
The labeled non-nesting partition in \Cref{compex} has $3$ interlinked pieces. 
The first compartment consists of just the first interlinked piece since it contains the label $1$. 
The smallest label in the rest of the diagram is $3$ which is in the last interlinked piece. 
Hence, this labeled non-nesting partition has $2$ compartments.
\end{example}

\begin{figure}[H]
    \centering
    \begin{tikzpicture}
        \node (-4) at (-4+0.5-2,0) {$1$};
        \node (-3) at (-3+0.5-2,0) {$4$};
        \node (-2) at (-2+0.5-2,0) {$1$};
        \node (-1) at (-1+0.5-2,0) {$2$};
        \node (1) at (1-0.5-2,0) {$4$};
        \node (2) at (2-0.5-2,0) {$5$};
        \node (3) at (3-0.5-2,0) {$2$};
        \node (4) at (4-0.5-2,0) {$5$};
        
        \node (-6) at (3-0.5,0) {$6$};
        \node (-5) at (4-0.5,0) {$6$};
        
        \node (5) at (5-0.5,0) {$3$};
        \node (6) at (6-0.5,0) {$3$};
        
        \draw (-6.north)..controls +(up:7mm) and +(up:7mm)..(-5.north);
        \draw (6.north)..controls +(up:7mm) and +(up:7mm)..(5.north);
        \draw (-4.north)..controls +(up:13mm) and +(up:13mm)..(-2.north);
        \draw (4.north)..controls +(up:13mm) and +(up:13mm)..(2.north);
        \draw (-3.north)..controls +(up:17mm) and +(up:17mm)..(1.north);
        \draw (3.north)..controls +(up:17mm) and +(up:17mm)..(-1.north);
    \end{tikzpicture}
    \caption{A labeled non-nesting partition with $3$ interlinked pieces and $2$ compartments.}
    \label{compex}
\end{figure}
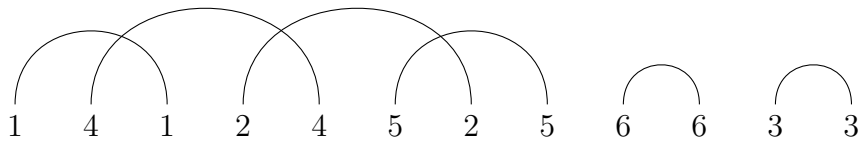

A non-nesting partition labeled with distinct integers (not necessarily of the form $[n]$) can be broken up into compartments in the same way. 
Here the first compartment consists of the interlinked pieces up to the one containing the smallest label.

It can be checked that compartments give labeled non-nesting partitions an exponential structure. 
This is because the order in which they appear can be determined by their labels. 
A labeled non-nesting partition is said to be \emph{connected} if it has only one compartment.

We now define a similar statistic for labeled symmetric non-nesting partitions. 
To a symmetric non-nesting partition we can associate a pair consisting of
\begin{enumerate}
    \item an interlinked symmetric non-nesting partition, which we call the \emph{bounded part} and
    \item a non-nesting partition, which we call the \emph{unbounded part}.
\end{enumerate}
This is easy to do using arc diagrams, as illustrated in the following example. 
The terminology becomes clear when one considers the boundedness of the coordinates in the region corresponding to a labeled symmetric non-nesting partition.

\begin{example}
To the symmetric $2$-non-nesting partition in \Cref{symnonfull} we associate
\begin{enumerate}
    \item the interlinked symmetric $2$-non-nesting partition marked $A$ and
    \item the $2$-non-nesting partition marked $B$.
\end{enumerate}
Here $A$ is the bounded part and $B$ is the unbounded part. 
We can obtain the original arc diagram back from $A$ and $B$ by placing a copy of $B$ on either side of $A$.
\end{example}

\begin{figure}[H]
    \centering
    \begin{tikzpicture}
        \node [circle,inner sep=2pt,fill=black] (-9) at (-9+0.5,0) {};
        \node [circle,inner sep=2pt,fill=black] (-8) at (-8+0.5,0) {};
        \node [circle,inner sep=2pt,fill=black] (-7) at (-7+0.5,0) {};
        
        \node [circle,inner sep=2pt,fill=black] (-6) at (-6+0.5,0) {};
        \node [circle,inner sep=2pt,fill=black] (-5) at (-5+0.5,0) {};
        \node [circle,inner sep=2pt,fill=black] (-4) at (-4+0.5,0) {};
        
        \draw [blue] (-3,-0.35)--(-3,0.35);
        
        \node [circle,inner sep=2pt,fill=black] (-3) at (-3+0.5,0) {};
        \node [circle,inner sep=2pt,fill=black] (-2) at (-2+0.5,0) {};
        \node [circle,inner sep=2pt,fill=black] (-1) at (-1+0.5,0) {};
        
        \draw [decorate,decoration={brace,amplitude=10pt,mirror},xshift=-4pt,yshift=0pt] (-5.5,-0.4)--(-0.25,-0.4) node [xshift=-2.6cm, yshift=-0.7cm]{$A$};
        
        \node [circle,inner sep=2pt,fill=black] (1) at (0+0.5,0) {};
        \node [circle,inner sep=2pt,fill=black] (2) at (1+0.5,0) {};
        \node [circle,inner sep=2pt,fill=black] (3) at (2+0.45,0) {};
        
        \draw [decorate,decoration={brace,amplitude=10pt,mirror},xshift=-4pt,yshift=0pt] (-8.5+9,-0.4)--(-6.25+9,-0.4) node [xshift=-1.15cm, yshift=-0.7cm]{$B$};
        
        \draw (-9.north)..controls +(up:7mm) and +(up:7mm)..(-8.north);
        \draw (-8.north)..controls +(up:7mm) and +(up:7mm)..(-7.north);
        
        \draw (-6.north)..controls +(up:7mm) and +(up:7mm)..(-5.north);
        \draw (-5.north)..controls +(up:13mm) and +(up:13mm)..(-3.north);
        \draw (-4.north)..controls +(up:13mm) and +(up:13mm)..(-2.north);
        \draw (-2.north)..controls +(up:7mm) and +(up:7mm)..(-1.north);
        
        \draw (1.north)..controls +(up:7mm) and +(up:7mm)..(2.north);
        \draw (2.north)..controls +(up:7mm) and +(up:7mm)..(3.north);
    \end{tikzpicture}
    \caption{Break up of a symmetric $2$-non-nesting partition.}
    \label{symnonfull}
\end{figure}
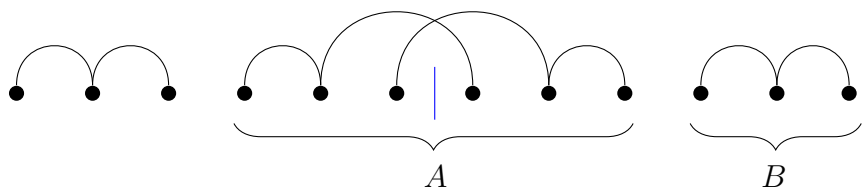

This is a bijection between symmetric non-nesting partitions and such pairs. 
Given a labeled symmetric non-nesting partition, we define the statistic using just the unbounded part. 
Ignoring the signs, we break the unbounded part into compartments just as we did for non-nesting partitions. 
A \emph{positive compartment} is one whose last element has a positive label.

\begin{example}
Suppose the arc diagram in \Cref{oddcompex} is the unbounded part of some symmetric non-nesting partition. 
Notice that ignoring the signs, this arc diagrams breaks up into compartments just as \Cref{compex}. 
But only the first compartment is positive since its last element has label $6$ which is positive.
\end{example}

\begin{figure}[H]
    \centering
    \begin{tikzpicture}
        \node (-4) at (-4+0.5-2,0) {$-1$};
        \node (-3) at (-3+0.5-2,0) {$4$};
        \node (-2) at (-2+0.5-2,0) {$-1$};
        \node (-1) at (-1+0.5-2,0) {$-2$};
        \node (1) at (1-0.5-2,0) {$4$};
        \node (2) at (2-0.5-2,0) {$6$};
        \node (3) at (3-0.5-2,0) {$-2$};
        \node (4) at (4-0.5-2,0) {$6$};
        
        \node (-6) at (3-0.5,0) {$8$};
        \node (-5) at (4-0.5,0) {$8$};
        
        \node (5) at (5-0.5,0) {$-3$};
        \node (6) at (6-0.5,0) {$-3$};
        
        \draw (-6.north)..controls +(up:7mm) and +(up:7mm)..(-5.north);
        \draw (6.north)..controls +(up:7mm) and +(up:7mm)..(5.north);
        \draw (-4.north)..controls +(up:13mm) and +(up:13mm)..(-2.north);
        \draw (4.north)..controls +(up:13mm) and +(up:13mm)..(2.north);
        \draw (-3.north)..controls +(up:17mm) and +(up:17mm)..(1.north);
        \draw (3.north)..controls +(up:17mm) and +(up:17mm)..(-1.north);
    \end{tikzpicture}
    \caption{The unbounded part of a symmetric non-nesting partition that has $1$ positive compartment.}
    \label{oddcompex}
\end{figure}
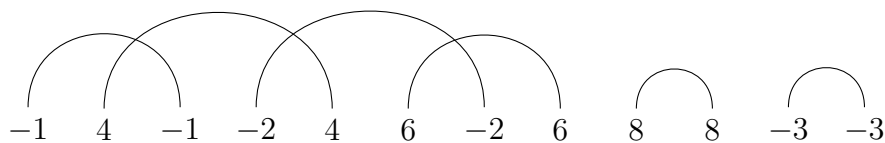

We claim that the statistic `number of positive compartments' meets our requirements. 
To prove that the distribution of this statistic is given by the characteristic polynomial, we apply \Cref{gesa} to the sequence of arrangements $\{\C_n^{(m)}\}$. 
Using the bijection between labeled symmetric $m$-non-nesting partitions and regions of $\C_n^{(m)}$, we note that those arc diagrams that are interlinked are the ones that correspond to bounded regions. 
Hence, using the notations form \Cref{gesa}, and \cite[Proposition 5.1.1]{sta_ec2}, we have
\begin{equation}\label{regbreakup}
    F(-x) = G(-x) \cdot \left( \sum_{n \geq 0} \frac{2^nn!}{mn + 1} \binom{(m + 1)n}{n} \frac{x^n}{n!} \right).
\end{equation}
Note that $\operatorname{rank}(\C_n^{(m)}) = n$. 
This gives us
\begin{equation*}
    \sum_{n \geq 0} \chi_{\A_n}(-t) \frac{(-x)^n}{n!} = G(-x) \cdot \left( \sum_{n \geq 0} \frac{2^nn!}{mn + 1} \binom{(m + 1)n}{n} \frac{x^n}{n!} \right)^{\frac{t + 1}{2}}.
\end{equation*}
Using the combinatorial interpretation of \eqref{genfunint}, we see that the right hand side of the above equation is the generating function for the distribution of the statistic.

We also obtain corresponding statistics on symmetric sketches using the bijection in \Cref{extc}. 
This gives us the following result.

\begin{theorem}
The absolute value of the coefficient of $t^j$ in $\chi_{\C_n^{(m)}}(t)$ is the number of symmetric $m$-sketches of size $n$ that have $j$ positive compartments.
\end{theorem}

For the arrangements $\D_n$, $\Po_n$, $\B_n$, and $\BC_n$ as well, the analogue of \eqref{regbreakup} holds. 
That is, for each of these arrangements, using the notation of \Cref{gesa}, we have
\begin{equation*}
    F(-x) = G(-x) \cdot \left( \sum_{n \geq 0} \frac{2^nn!}{n + 1} \binom{2n}{n} \frac{x^n}{n!} \right).
\end{equation*}
This can be proved using the definitions of type $D$, pointed, type $B$, and type $BC$ sketches and the description of which sketches correspond to bounded regions.

There is a slight difference in the proof for the sequence of arrangements $\{\D_n\}$. 
The arrangement $\D_1$ is empty and hence
\begin{equation*}
    G(-x) = 1 - x + \sum_{n \geq 2} b(\D_n) \frac{x^n}{n!}.
\end{equation*}
However, from the definition of type $D$ sketches, we see that we must not allow those symmetric non-nesting partitions where the bounded part is empty and the first interlinked piece of the unbounded part is of size $1$ with negative label. 
Hence, we still get the required expression for $F(-x)$.

Just as we did for the extended type $C$ Catalan arrangements, we define positive compartments for the arc diagrams corresponding to the regions of these arrangements, which gives corresponding statistics on the sketches.

\begin{example}
The arc diagram in \Cref{pointedcompex} corresponds to a pointed sketch with $2$ positive compartments.
\end{example}

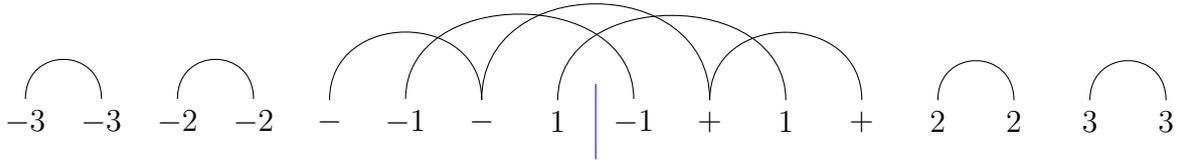
\begin{figure}[H]
    \centering
    \begin{tikzpicture}
        \node (-10) at (-8+0.5,0) {$-3$};
        \node (-9) at (-7+0.5,0) {$-3$};
        \node (-8) at (-6+0.5,0) {$-2$};
        \node (-7) at (-5+0.5,0) {$-2$};
        
        \node (-6) at (-4+0.5,0) {$-$};
        \node (-4) at (-3+0.5,0) {$-1$};
        \node (-3) at (-2+0.5,0) {$-$};
        \node (-2) at (-1+0.5,0) {$1$};
        \draw [blue] (0,-0.5)--(0,0.5);
        \node (2) at (1-0.5,0) {$-1$};
        \node (3) at (2-0.5,0) {$+$};
        \node (4) at (3-0.5,0) {$1$};
        \node (6) at (4-0.5,0) {$+$};
        
        \node (7) at (5-0.5,0) {$2$};
        \node (8) at (6-0.5,0) {$2$};
        \node (9) at (7-0.5,0) {$3$};
        \node (10) at (8-0.5,0) {$3$};
        
        \draw (-10.north)..controls +(up:7mm) and +(up:7mm)..(-9.north);
        \draw (-8.north)..controls +(up:7mm) and +(up:7mm)..(-7.north);
        \draw (8.north)..controls +(up:7mm) and +(up:7mm)..(7.north);
        \draw (10.north)..controls +(up:7mm) and +(up:7mm)..(9.north);
        
        \draw (-6.north)..controls +(up:12mm) and +(up:12mm)..(-3.north);
        \draw (3.north)..controls +(up:17mm) and +(up:17mm)..(-3.north);
        \draw (6.north)..controls +(up:12mm) and +(up:12mm)..(3.north);
        \draw (-4.north)..controls +(up:15mm) and +(up:15mm)..(2.north);
        \draw (-2.north)..controls +(up:15mm) and +(up:15mm)..(4.north);
    \end{tikzpicture}
    \caption{Arc diagram corresponding to a pointed sketch with $2$ positive compartments.}
    \label{pointedcompex}
\end{figure}

The following result can be proved just as before.

\begin{theorem}
The absolute value of the coefficient of $t^j$ in $\chi_{\A}(t)$ for $\A = \D_n$ (respectively $\Po_n$, $\B_n, \BC_n$) is the number of type $D$ (respectively pointed, type $B$, type $BC$) sketches of size $n$ that have $j$ positive compartments.
\end{theorem}

\section{Deformations of the threshold arrangement}\label{threshsec}

The threshold arrangement in $\R^n$ consists of the hyperplanes $x_i + x_j = 0$ for $1 \leq i < j \leq n$. 
These arrangements are of interest because their regions correspond to certain labeled graphs called \emph{threshold graphs} which have been extensively studied (see \cite{thresh_book}). 
In this section, we study this arrangement and some of its deformations using similar techniques as in previous sections.

\subsection{Sketches and moves}\label{threshfubsm}

We use the sketches and moves idea to study the regions of the threshold arrangement by considering it as a sub-arrangement of the type $C$ arrangement (\Cref{typeC}). 
Before doing that, we first study the arrangement obtained by adding the coordinate hyperplanes to the threshold arrangement.

\subsubsection{Fubini arrangement}\label{fubini}

We define the Fubini arrangement in $\R^n$ to be the one with hyperplanes
\begin{align*}
    2x_i&=0\\
    x_i+x_j&=0
\end{align*}
for all $1 \leq i < j \leq n$.
The hyperplanes missing from the type $C$ arrangement are
\begin{equation*}
    x_i - x_j = 0
\end{equation*}
for all $1 \leq i < j \leq n$.
Hence a Fubini move, which we call an $F$ move, is swapping adjacent $\overset{+}{i}$ and $\overset{+}{j}$ as well as $\overset{-}{j}$ and $\overset{-}{i}$ for distinct $i,j \in [n]$.

\begin{example}\label{bellex}
We can use a series of $F$ moves on a sketch as follows:

\begin{equation*}
    \overset{-}{3}\ \overset{-}{6}\ \overset{-}{2}\ \overset{+}{1}\ \overset{+}{4}\ \overset{-}{5}\ \textcolor{blue}{|} \ \overset{+}{5}\ \overset{-}{4}\ \overset{-}{1}\ \overset{+}{2}\ \overset{+}{6}\ \overset{+}{3}\ \longrightarrow \ \overset{-}{6}\ \overset{-}{3}\ \overset{-}{2}\ \overset{+}{1}\ \overset{+}{4}\ \overset{-}{5}\ \textcolor{blue}{|} \ \overset{+}{5}\ \overset{-}{4}\ \overset{-}{1}\ \overset{+}{2}\ \overset{+}{3}\ \overset{+}{6}\ \longrightarrow \ \overset{-}{6}\ \overset{-}{3}\ \overset{-}{2}\ \overset{+}{4}\ \overset{+}{1}\ \overset{-}{5}\ \textcolor{blue}{|} \ \overset{+}{5}\ \overset{-}{1}\ \overset{-}{4}\ \overset{+}{2}\ \overset{+}{3}\ \overset{+}{6}
\end{equation*}
\end{example}

We define a \textit{block} to be the set of absolute values in a maximal string of contiguous terms in the second half of a sketch that have the same sign. 
The blocks of the initial sketch in \Cref{bellex} are $\{5\},\{1,4\},\{2,3,6\}$ (these blocks appear in this order with the first one being positive). 
It can be checked that $F$ moves do not change the sequence of signs (above the numbers) and that they can only be used to reorder the elements in a block. 
Hence, each equivalence class has a unique sketch where the numbers in each block appear in ascending order. 
The last sketch in \Cref{bellex} is the unique such sketch in its equivalence class.

The number of such sketches is equal to the number of ways of choosing an ordered partition of $[n]$ (which correspond to the blocks of the sketch in order) and then choosing a sign for the first block. 
Hence the number of regions of the Fubini arrangement is $2 \cdot a(n)$ where $a(n)$ is the $n^{th}$ Fubini number, which is the number of ordered partitions of $[n]$ listed as \href{https://oeis.org/A000670}{A000670} in the OEIS \cite{oeis}.

\subsubsection{Threshold arrangement}\label{threshsketch}

The threshold arrangement in $\mathbb{R}^n$ has the hyperplanes
\begin{equation*}
    x_i+x_j=0
\end{equation*}
for all $1 \leq i < j \leq n$.
The hyperplanes missing from the type $C$ arrangement are
\begin{align*}
    2x_i&=0\\
    x_i-x_j&=0
\end{align*}
for all $1 \leq i < j \leq n$.
Hence the threshold moves, which we call $T$ moves, are as follows:
\begin{enumerate}
    \item ($D$ move) Swapping adjacent $\overset{+}{i}$ and $\overset{-}{i}$ for any $i \in [n]$.
    \item ($F$ move) Swapping adjacent $\overset{+}{i}$ and $\overset{+}{j}$ as well as $\overset{-}{j}$ and $\overset{-}{i}$ for distinct $i, j \in [n]$.
\end{enumerate}

For any sketch, there is a $T$ equivalent sketch for which the first block has more than 1 element. 
This is because, if the sketch has first block of size 1, applying a $D$ move (swapping the $n^{th}$ and $(n+1)^{th}$ term), will result in a sketch where the first block has size greater than 1 (first step in \Cref{tex}).

\begin{example}\label{tex}
We can use a series of $T$ moves on a sketch as follows:
\begin{equation*}
    \overset{+}{5}\ \overset{-}{4}\ \overset{-}{1}\ \overset{+}{2}\ \overset{+}{6}\ \overset{-}{3}\ \textcolor{blue}{|} \ \overset{+}{3}\ \overset{-}{6}\ \overset{-}{2}\ \overset{+}{1}\ \overset{+}{4}\ \overset{-}{5} \xrightarrow{D\ move} \overset{+}{5}\ \overset{-}{4}\ \overset{-}{1}\ \overset{+}{2}\ \overset{+}{6}\ \overset{+}{3}\ \textcolor{blue}{|} \ \overset{-}{3}\ \overset{-}{6}\ \overset{-}{2}\ \overset{+}{1}\ \overset{+}{4}\ \overset{-}{5} \xrightarrow{F\ moves} \overset{+}{5}\ \overset{-}{4}\ \overset{-}{1}\ \overset{+}{6}\ \overset{+}{3}\ \overset{+}{2}\ \textcolor{blue}{|} \ \overset{-}{2}\ \overset{-}{3}\ \overset{-}{6}\ \overset{+}{1}\ \overset{+}{4}\ \overset{-}{5}
\end{equation*}
\end{example}

To obtain a canonical sketch for each threshold region, we will need a small lemma.

\begin{lemma}\label{threshlem}
    Two $T$ equivalent sketches that have their first block of size greater than 1 have the same blocks which appear in the same order with the same signs.
\end{lemma}

\begin{proof}
    Looking at what the $T$ moves do to the sequence of signs (above the numbers), we can see that they at most swap the $n^{th}$ and $(n+1)^{th}$ sign ($D$ move). 
    Hence, if we require the first blocks to have size greater than 1, both the sketches have the same number of blocks and the number of elements in the corresponding blocks are the same. 
    An $F$ move can only reorder elements in the same block of a sketch. 
    A $D$ move changes the sign of the first element of the second half. 
    So if there are $k>1$ elements in the first block of a $T$ equivalent sketch, then the set of absolute values of the first $k$ elements of the second half remains the same in all $T$ equivalent sketches. 
    This gives us the required result.
\end{proof}

Using the above lemma, we can see that for any sketch there is a unique $T$ equivalent sketch where the size of the first block is greater than $1$ and the elements of each block are in ascending order. 
The last sketch in \Cref{tex} is the unique such sketch in its equivalence class. 
Similar to the count for Fubini regions, we get that the number of regions of the threshold arrangement is
\begin{equation*}
    2 \cdot (a(n) - n \cdot a(n-1))
\end{equation*}
where, as before, $a(n)$ is the $n^{th}$ Fubini number. 
The number of regions of the threshold arrangement is listed as \href{https://oeis.org/A005840}{A005840} in the OEIS \cite{oeis}.

\begin{remark}
The regions of the threshold arrangement in $\R^n$ are known to be in bijection with labeled threshold graphs on $n$ vertices (see \cite[Exercise 5.25]{sta_hyp}). 
Labeled threshold graphs on $n$ vertices are inductively constructed starting from the empty graph. 
Vertices labeled $1,\ldots,n$ are added in a specified order. 
At each step, the vertex added is either `dominant' or `recessive'. 
A dominant vertex is one that is adjacent to all vertices added before it and a recessive vertex is one that is isolated from all vertices added before it. 
It is not difficult to see that the canonical sketches described above are in bijection with threshold graphs.
\end{remark}

\subsection{Statistics}

The characteristic polynomial of the threshold arrangement and a statistic on its regions whose distribution is given by the characteristic polynomial has been studied in \cite{thresh}. 
This is done by directly looking at the coefficients of the characteristic polynomial. 
In fact, even the coefficients of the characteristic polynomial of the Fubini arrangement (\Cref{fubini}) have already been combinatorially interpreted in \cite[Section 4.1]{thresh}. 
This can be used to define an appropriate statistic on the regions of the Fubini arrangement. 
Here, just as in \Cref{statsec}, we use \Cref{gesa} to combinatorially interpret the generating functions of the characteristic polynomials for the Fubini and threshold arrangements. 
Just as before, we will show that the statistic `number of positive compartments' works for our purposes.

\subsubsection{Fubini arrangement}

We will use the second half of the canonical sketches described in \Cref{fubini} to represent the regions. 
We define blocks for signed permutations just as we did for sketches. 
Hence, the regions of the Fubini arrangement in $\R^n$ correspond to signed permutations on $[n]$ where each block is increasing.

In this special class of signed permutations as well, compartments give them an exponential structure. 
This is because there is no condition relating the signs of the last element of a compartment and the first element of the compartment following it. 
This is because the last element of a compartment is necessarily smaller in absolute value than the element following it. 
Also, suppose we are given a signed permutation such that each block is increasing. 
It can be checked that the signed permutation obtained by changing all the signs also satisfies this property.

Using the above observations and the combinatorial interpretation of \eqref{genfunint}, we get that
\begin{equation*}
    \left(\frac{e^x}{2 - e^x}\right)^{\frac{t + 1}{2}}
\end{equation*}
is the exponential generating function for signed permutations where each block is increasing where $t$ keeps track of the number of positive compartments. 
This agrees with the generating function for the characteristic polynomial we get from \Cref{gesa} since we have
\begin{align*}
    F(x) &= \left(\frac{1}{2e^x - 1}\right),\\
    G(x) &= 1.
\end{align*}

\subsubsection{Threshold arrangement}

From \Cref{threshsketch}, we can see that the regions of the threshold arrangement in $\R^n$ correspond to signed permutations on $[n]$ where each block is increasing and the first block has size greater than $1$. 
If such a permutation starts with $\overset{-}{1}$, we instead use the signed permutation obtained by changing $\overset{-}{1}$ to $\overset{+}{1}$ to represent the region. 
Similar to how we obtained the generating function for the statistic for type $D$ from the one for type $C$, we obtain our generating function from the one we have for the Fubini arrangement.

Suppose that we are given $i \in [n]$ and a signed permutation $\sigma$ on $[n] \setminus \{i\}$ whose blocks are increasing. 
If $i = 1$ we construct the signed permutation on $[n]$ obtained by appending $\overset{-}{1}$ to the front of $\sigma$. 
If $i > 1$, and the first element of $\sigma$ is $\overset{\pm}{j}$. 
We construct the signed permutation on $[n]$ obtained by appending $\overset{\mp}{i}$ to the start of $\sigma$. 
In both cases, it can be checked that the number of positive compartment of the new signed permutation constructed is the same as that for $\sigma$.

This shows that the distribution of the statistic `number of positive compartments' on the signed permutations that correspond to regions of the threshold arrangement is
\begin{equation*}
    (1 - x)\left(\frac{e^x}{2 - e^x}\right)^{\frac{t + 1}{2}}.
\end{equation*}
This agrees with the generating function for the characteristic polynomial we get from \Cref{gesa} since we have
\begin{align*}
    F(x) &= \left(\frac{1 + x}{2e^x - 1}\right),\\
    G(x) &= 1 + x.
\end{align*}

\subsection{Some deformations}

Deformations of the threshold arrangement have not been as well-studied as those of the braid arrangement. 
However, the finite field method has been used to compute the characteristic polynomial for some deformations. 
In \cite{seo_shi, seo_cat}, Seo computed the characteristic polynomials of the so called Shi and Catalan threshold arrangements. 
Expressions for the characteristic polynomials of more general deformations have been computed in \cite{balasubramanian2019generalized}.

In this section, we use the sketches and moves technique to obtain certain non-nesting partitions that are in bijection with the regions of the Catalan and Shi threshold arrangements. 
We do this by considering these arrangements as sub-arrangements of the type $C$ Catalan arrangement (\Cref{typecsec}). 
Unfortunately, we were not able to directly count the non-nesting partitions we obtained since their description is not as simple as the ones we have seen before.

Fix $n \geq 2$ throughout this section. 
Recall that we studied the type $C$ Catalan arrangement by considering a translation of it called $\C_n$ whose hyperplane are given by \eqref{Cnhyp} and whose regions correspond to symmetric sketches of size $n$ (see \Cref{symsk}). 
Symmetric sketches can also be viewed as labeled symmetric non-nesting partitions (see \Cref{slnnpex}).

\subsubsection{Catalan threshold}

The Catalan threshold arrangement in $\mathbb{R}^n$ consists of the hyperplanes
\begin{equation*}
    X_i+X_j=-1,0,1
\end{equation*}
for all $1 \leq i < j\leq n$. 
The translated arrangement by setting $X_i=x_i + \frac{1}{2}$, which we call $\CT_n$, has hyperplanes
\begin{equation*}
    x_i+x_j=-2,-1,0
\end{equation*}
for all $1 \leq i < j \leq n$. 
We consider this arrangement as a sub-arrangement of $\C_n$. 
Using Bernardi's idea of moves, we can define an equivalence on the symmetric sketches such that two sketches are equivalent if they lie in the same region of $\CT_n$.

An $\alpha_+$ letter is an $\alpha$-letter whose subscript is positive. 
We similarly define $\alpha_-, \beta_+$ and $\beta_-$ letters. 
The `mod-value' of a letter \lt{i}{s} is $|i|$.

The hyperplanes in $\C_n$ that are not in $\CT_n$ are
\begin{align*}
    2x_i&=-2,-1,0\\
    x_i-x_j&=-1,0,1
\end{align*}
where $1 \leq i < j \leq n$. 
Changing the inequality corresponding to exactly one of these hyperplanes is given by the following moves on a sketch, which we call $\CT$ moves.

\begin{enumerate}[label = (\alph*)]
    \item Swapping the $2n^{th}$ and $(2n+1)^{th}$ letter.
    
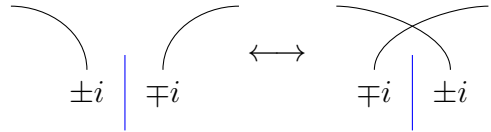
\begin{center}
        \begin{tikzpicture}
        \node (-1) at (-0.5,0) {$\pm i$};
        \node (1) at (0.5,0) {$\mp i$};
        \node (a) at (-1.5,1) {};
        \node (b) at (1.5,1) {};
        \draw [blue](0,-0.5)--(0,0.5);
        \draw (-1.north)..controls +(up:5mm) and +(right:5mm)..(a.north);
        \draw (1.north)..controls +(up:5mm) and +(left:5mm)..(b.north);
        \node at (2,0.5) {$\longleftrightarrow$};
    \end{tikzpicture}
    \begin{tikzpicture}
        \node (-1) at (-0.5,0) {$\mp i$};
        \node (1) at (0.5,0) {$\pm i$};
        \node (a) at (-1,1) {};
        \node (b) at (1,1) {};
        \draw [blue](0,-0.5)--(0,0.5);
        \draw (1.north)..controls +(up:5mm) and +(right:5mm)..(a.north);
        \draw (-1.north)..controls +(up:5mm) and +(left:5mm)..(b.north);
    \end{tikzpicture}
    \end{center}
    This corresponds to changing the inequality corresponding to a hyperplane of the form $2x_i=-2$ or $2x_i=0$.
    
    \item Swapping the $n^{th}$ and $(n+1)^{th}$ $\alpha$-letter if they are consecutive (along with the $n^{th}$ and $(n+1)^{th}$ $\beta$).
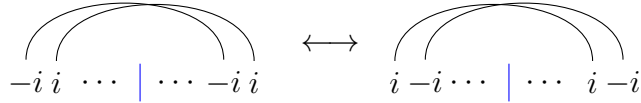
    \begin{center}
        \begin{tikzpicture}
        \node (-2) at (-1.5,0) {$-i$};
        \node (-1) at (-1.1,0) {$i$};
        \node at (-0.5,0) {$\cdots$};
        \node at (0.5,0) {$\cdots$};
        \node (1) at (1.1,0) {$-i$};
        \node (2) at (1.5,0) {$i$};
        \draw [blue](0,-0.25)--(0,0.25);
        \draw (-2.north)..controls +(up:10mm) and +(up:10mm)..(1.north);
        \draw (-1.north)..controls +(up:10mm) and +(up:10mm)..(2.north);
        \node at (2.5,0.5) {$\longleftrightarrow$};
        \end{tikzpicture}
        \begin{tikzpicture}
        \node (-2) at (-1.5,0) {$i$};
        \node (-1) at (-1.1,0) {$-i$};
        \node at (-0.5,0) {$\cdots$};
        \node at (0.5,0) {$\cdots$};
        \node (1) at (1.1,0) {$i$};
        \node (2) at (1.5,0) {$-i$};
        \draw [blue](0,-0.25)--(0,0.25);
        \draw (-2.north)..controls +(up:10mm) and +(up:10mm)..(1.north);
        \draw (-1.north)..controls +(up:10mm) and +(up:10mm)..(2.north);
        \end{tikzpicture}
    \end{center}
    This corresponds to changing the inequality corresponding to a hyperplane of the form $2x_i=-1$.
    
    \item Swapping consecutive $\alpha_+$ and $\beta_+$ letters (along with their negatives).
    
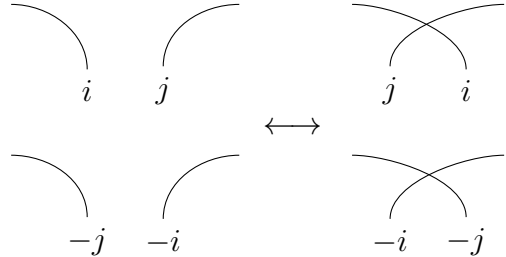
\begin{center}
        \begin{tikzpicture}
        \node (-1) at (-0.5,0) {$i$};
        \node (1) at (0.5,0) {$j$};
        \node (a) at (-1.5,1) {};
        \node (b) at (1.5,1) {};
        \draw (-1.north)..controls +(up:5mm) and +(right:5mm)..(a.north);
        \draw (1.north)..controls +(up:5mm) and +(left:5mm)..(b.north);
        \node (-1') at (-0.5,-2) {$-j$};
        \node (1') at (0.5,-2) {$-i$};
        \node (a') at (-1.5,-1) {};
        \node (b') at (1.5,-1) {};
        \draw (-1'.north)..controls +(up:5mm) and +(right:5mm)..(a'.north);
        \draw (1'.north)..controls +(up:5mm) and +(left:5mm)..(b'.north);
        \node at (2.2,-0.5) {$\longleftrightarrow$};
    \end{tikzpicture}
    \begin{tikzpicture}
        \node (-1) at (-0.5,0) {$j$};
        \node (1) at (0.5,0) {$i$};
        \node (a) at (-1,1) {};
        \node (b) at (1,1) {};
        \draw (1.north)..controls +(up:5mm) and +(right:5mm)..(a.north);
        \draw (-1.north)..controls +(up:5mm) and +(left:5mm)..(b.north);
        \node (-1') at (-0.5,-2) {$-i$};
        \node (1') at (0.5,-2) {$-j$};
        \node (a') at (-1,-1) {};
        \node (b') at (1,-1) {};
        \draw (1'.north)..controls +(up:5mm) and +(right:5mm)..(a'.north);
        \draw (-1'.north)..controls +(up:5mm) and +(left:5mm)..(b'.north);
    \end{tikzpicture}
    \end{center}
    This corresponds to changing the inequality corresponding to a hyperplane of the form $x_i-x_j=1$.
    
    \item Swapping $\{\lt{i}{0},\lt{j}{0}\}$ as well as $\{\lt{i}{1},\lt{j}{1}\}$ if both pairs are consecutive (as well as their negatives) where $i, j \in [n]$ are distinct.
    
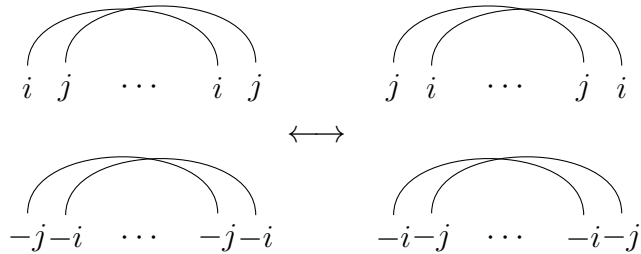
\begin{center}
        \begin{tikzpicture}
        \node (-2) at (-1.5,0) {$i$};
        \node (-1) at (-1,0) {$j$};
        \node at (0,0) {$\cdots$};
        \node (1) at (1,0) {$i$};
        \node (2) at (1.5,0) {$j$};
        \draw (-2.north)..controls +(up:10mm) and +(up:10mm)..(1.north);
        \draw (-1.north)..controls +(up:10mm) and +(up:10mm)..(2.north);
        \node (-2') at (-1.5,-2) {$-j$};
        \node (-1') at (-1,-2) {$-i$};
        \node at (0,-2) {$\cdots$};
        \node (1') at (1,-2) {$-j$};
        \node (2') at (1.5,-2) {$-i$};
        \draw (-2'.north)..controls +(up:10mm) and +(up:10mm)..(1'.north);
        \draw (-1'.north)..controls +(up:10mm) and +(up:10mm)..(2'.north);
        \node at (2.3,-0.65) {$\longleftrightarrow$};
        \end{tikzpicture}
        \begin{tikzpicture}
        \node (-2) at (-1.5,0) {$j$};
        \node (-1) at (-1,0) {$i$};
        \node at (0,0) {$\cdots$};
        \node (1) at (1,0) {$j$};
        \node (2) at (1.5,0) {$i$};
        \draw (-2.north)..controls +(up:10mm) and +(up:10mm)..(1.north);
        \draw (-1.north)..controls +(up:10mm) and +(up:10mm)..(2.north);
        \node (-2') at (-1.5,-2) {$-i$};
        \node (-1') at (-1,-2) {$-j$};
        \node at (0,-2) {$\cdots$};
        \node (1') at (1,-2) {$-i$};
        \node (2') at (1.5,-2) {$-j$};
        \draw (-2'.north)..controls +(up:10mm) and +(up:10mm)..(1'.north);
        \draw (-1'.north)..controls +(up:10mm) and +(up:10mm)..(2'.north);
        \end{tikzpicture}
    \end{center}
    This corresponds to changing the inequality corresponding to the hyperplane $x_i-x_j=1$.
\end{enumerate}
Two sketches are in the same region of $\CT_n$ if and only if they are related by a series of $\CT$ moves. 
We call such sketches $\CT$ equivalent.

Consider the sketches to be ordered in the lexicographic order induced by the following order on the letters.
\begin{equation*}
    \lt{n}{0} \succ \cdots \succ \lt{1}{0} \succ \lt{-1}{-1} \succ \cdots \succ \lt{-n}{-1} \succ \lt{n}{1} \succ \cdots \succ \lt{1}{1} \succ \lt{-1}{0} \succ \cdots \succ \lt{-n}{0}
\end{equation*}
In other words, the $\alpha$-letters are greater than the $\beta$-letters and for letters of the same type, the order is given by comparing the subscripts.

A sketch is called $\CT$ maximal if it is greater (in the lexicographic order) than all sketches to which it is $\CT$ equivalent. 
Hence the regions of $\CT_n$ are in bijection with the $\CT$ maximal sketches.

\begin{theorem}\label{catthresh}
    A symmetric sketch is $\CT$ maximal if and only if the following hold.
    \begin{enumerate}
        \item If a $\beta$-letter is followed by an $\alpha$-letter, they should be of opposite signs and different mod-values.
        \begin{center}
            \begin{tikzpicture}
            \node (-1) at (-0.5,0) {$X$};
            \node (1) at (0.5,0) {$Y$};
            \node (a) at (-1.5,1) {};
            \node (b) at (1.5,1) {};
            \draw (-1.north)..controls +(up:5mm) and +(right:5mm)..(a.north);
            \draw (1.north)..controls +(up:5mm) and +(left:5mm)..(b.north);
            \node at (4,0.5) {$\implies \text{X and Y of opposite sign}$};
            \node at (4.3,0) {and different mod value.};
            \end{tikzpicture}
        \end{center}
        \item If two $\alpha$-letters and their corresponding $\beta$-letters are both consecutive and of the same sign then the subscript of the first one should be greater.
        \begin{center}
            \begin{tikzpicture}
            \node (-2) at (-1.5,0) {$a_1$};
            \node (-1) at (-1,0) {$a_2$};
            \node at (0,0) {$\cdots$};
            \node (1) at (1,0) {$a_1$};
            \node (2) at (1.5,0) {$a_2$};
            \draw (-2.north)..controls +(up:10mm) and +(up:10mm)..(1.north);
            \draw (-1.north)..controls +(up:10mm) and +(up:10mm)..(2.north);
            \node at (5,0.5) {and $a_1,a_2$ same sign $\implies a_1 > a_2$.};
            \end{tikzpicture}
        \end{center}
        \item If the $n^{th}$ and $(n+1)^{th}$ $\alpha$-letters are consecutive, then so are the $(n-1)^{th}$ and $n^{th}$ with the $n^{th}$ $\alpha$-letter being positive. 
        In such a situation, if the $(n-1)^{th}$ $\alpha$-letter is negative and the $(n-1)^{th}$ and $n^{th}$ $\beta$-letters are consecutive, the $(n-1)^{th}$ $\alpha$-letter should have a subscript greater than that of the $(n+1)^{th}$ $\alpha$.
        \item If the $(2n-1)^{th}$ and $(2n+1)^{th}$ letters are both $\beta$-letters of the same sign and their corresponding $\alpha$-letters are consecutive, the subscript of the $(2n-1)^{th}$ letter should be greater than that of the $(2n+1)^{th}$.
        \begin{center}
        \begin{tikzpicture}
            \node (-2) at (-1.5,0) {$X$};
            \node (-1) at (-1,0) {$Y$};
            \node at (0,0) {$\cdots$};
            \node (1) at (1,0) {$X$};
            \node (3) at (1.5,0) {$-Y$};
            \node (a) at (3,0.75) {};
            \draw [blue](2,-0.25)--(2,0.25);
            \node (2) at (2.5,0) {$Y$};
            \draw (-2.north)..controls +(up:10mm) and +(up:10mm)..(1.north);
            \draw (3.north)..controls +(up:5mm) and +(left:5mm)..(a.north);
            \draw (-1.north)..controls +(up:10mm) and +(up:10mm)..(2.north);
            \node at (6.5,0.5) {and $\quad X,Y$ same sign $\implies X>Y$.};
            \end{tikzpicture}
        \end{center}
    \end{enumerate}
    Hence the regions of $\CT_n$ are in bijection with sketches of the form described above.
\end{theorem}

\begin{remark}
    The idea of ordering sketches and choosing the maximal sketch in each region of $\CT_n$ to represent it is the same one used by Bernardi \cite{ber} to study certain deformations of the braid arrangement. 
    In fact, \cite[Lemma 8.13]{ber} shows that in this case, any sketch that is locally maximal (greater than any sketch that can be obtained by applying a single move) is maximal. 
    Note that the sketches described in the above theorem are precisely the \emph{2-locally maximal} sketches. 
    That is, these are the sketches that can neither be converted into a greater sketch by applying a single $\CT$ move nor by applying two $\CT$ moves. 
    It is clear that any $\CT$ maximal sketch is 2-locally maximal. 
    The theorem states the converse is true as well.
\end{remark}

\begin{proof}[Proof of \Cref{catthresh}]
We first show that these conditions are required for a sketch to be $\CT$ maximal.
\begin{enumerate}
    \item The first condition is necessary since the $\CT$ moves of type (a) or (c) would result in a greater sketch if it were false.
    
    \item The second condition corresponds to $\CT$ moves of type (d).
    
    \item The part about the $n^{th}$ $\alpha$-letter being positive if the $n^{th}$ and $(n+1)^{th}$ $\alpha$-letters are consecutive is due to $\CT$ moves of type (c). 
    Suppose the letter before the $n^{th}$ $\alpha$-letter is a $\beta$-letter. 
    Then it can't be positive since we have already seen that condition (1) of the theorem statement must be satisfied. 
    But if it is negative, we can do the following to obtain a larger $\CT$ equivalent sketch:
    \begin{center}
        \begin{tikzpicture}[scale=0.8,xscale=-1]
            \node (-2) at (-1.5,0) {-$i$};
            \node (-1) at (-1,0) {$i$};
            \node at (-0.35,0) {\tiny $\cdots$};
            \node at (0.35,0) {\tiny $\cdots$};
            \node (1) at (1,0) {-$i$};
            \node (3) at (1.5,0) {$i$};
            \node (a) at (3,0.75) {};
            \draw [blue](0,-0.25)--(0,0.25);
            \node (2) at (2.15,0) {-$j$};
            \draw (-2.north)..controls +(up:10mm) and +(up:10mm)..(1.north);
            \draw (2.north)..controls +(up:5mm) and +(left:5mm)..(a.north);
            \draw (-1.north)..controls +(up:10mm) and +(up:10mm)..(3.north);
            \draw (-2.5,0.35) node {$\longrightarrow$};
        \end{tikzpicture}
        \begin{tikzpicture}[scale=0.8,xscale=-1]
            \node (-2) at (-1.5,0) {$i$};
            \node (-1) at (-1,0) {-$i$};
            \node at (-0.35,0) {\tiny $\cdots$};
            \node at (0.35,0) {\tiny $\cdots$};
            \node (1) at (1,0) {$i$};
            \node (3) at (1.5,0) {-$i$};
            \node (a) at (3,0.75) {};
            \draw [blue](0,-0.25)--(0,0.25);
            \node (2) at (2.15,0) {-$j$};
            \draw (-2.north)..controls +(up:10mm) and +(up:10mm)..(1.north);
            \draw (2.north)..controls +(up:5mm) and +(left:5mm)..(a.north);
            \draw (-1.north)..controls +(up:10mm) and +(up:10mm)..(3.north);
            \draw (-2.5,0.35) node {$\longrightarrow$};
        \end{tikzpicture}
        \begin{tikzpicture}[scale=0.8,xscale=-1]
            \node (-2) at (-1.5,0) {$i$};
            \node (-1) at (-1,0) {-$i$};
            \node at (-0.35,0) {\tiny $\cdots$};
            \node at (0.35,0) {\tiny $\cdots$};
            \node (1) at (1,0) {$i$};
            \node (3) at (1.5,0) {-$j$};
            \node (a) at (2.75,0.75) {};
            \draw [blue](0,-0.25)--(0,0.25);
            \node (2) at (2.15,0) {-$i$};
            \draw (-2.north)..controls +(up:10mm) and +(up:10mm)..(1.north);
            \draw (3.north)..controls +(up:5mm) and +(left:7mm)..(a.north);
            \draw (-1.north)..controls +(up:10mm) and +(up:10mm)..(2.north);
        \end{tikzpicture}
    \end{center}
    Hence the letter before the $n^{th}$ $\alpha$-letter has to be an $\alpha$-letter. 
    Now, suppose that the $(n-1)^{th}$ $\alpha$-letter is negative and the $(n-1)^{th}$ and $n^{th}$ $\beta$-letters are consecutive. 
    Let the subscript of the $(n-1)^{th}$ $\alpha$-letter be $-k$ and that of the $(n+1)^{th}$ $\alpha$-letter be $-i$ for some $k,i \in [n]$. 
    If $-k<-i$, we can do the following to obtain a larger $\CT$ equivalent sketch:
    \begin{center}
        \begin{tikzpicture}[scale=0.75]
            \node (-3) at (-2.5,0) {\small-$k$};
            \node (-2) at (-2,0) {\small$i$};
            \node (-1) at (-1.5,0) {\small-$i$};
            \node at (-0.5,0) {\small$\cdots$};
            \node at (0.5,0) {\small$\cdots$};
            \node (3) at (1.5,0) {\small-$k$};
            \node (1) at (2,0) {\small$i$};
            \node (2) at (2.5,0) {\small-$i$};
            \draw [blue](0,-0.25)--(0,0.25);
            \draw (-2.north)..controls +(up:10mm) and +(up:10mm)..(1.north);
            \draw (-3.north)..controls +(up:10mm) and +(up:10mm)..(3.north);
            \draw (-1.north)..controls +(up:10mm) and +(up:10mm)..(2.north);
            \draw (3.25,0.35) node {$\longrightarrow$};
            \end{tikzpicture}
            \begin{tikzpicture}[scale=0.75]
            \node (-3) at (-2.5,0) {\small-$k$};
            \node (-2) at (-2,0) {\small-$i$};
            \node (-1) at (-1.5,0) {\small$i$};
            \node at (-0.5,0) {\small$\cdots$};
            \node at (0.5,0) {\small$\cdots$};
            \node (3) at (1.5,0) {\small-$k$};
            \node (1) at (2,0) {\small-$i$};
            \node (2) at (2.5,0) {\small$i$};
            \draw [blue](0,-0.25)--(0,0.25);
            \draw (-2.north)..controls +(up:10mm) and +(up:10mm)..(1.north);
            \draw (-3.north)..controls +(up:10mm) and +(up:10mm)..(3.north);
            \draw (-1.north)..controls +(up:10mm) and +(up:10mm)..(2.north);
            \draw (3.25,0.35) node {$\longrightarrow$};
            \end{tikzpicture}
            \begin{tikzpicture}[scale=0.75]
            \node (-3) at (-2.5,0) {\small-$i$};
            \node (-2) at (-2,0) {\small-$k$};
            \node (-1) at (-1.5,0) {\small$i$};
            \node at (-0.5,0) {\small$\cdots$};
            \node at (0.5,0) {\small$\cdots$};
            \node (3) at (1.5,0) {\small-$i$};
            \node (1) at (2,0) {\small-$k$};
            \node (2) at (2.5,0) {\small$i$};
            \draw [blue](0,-0.25)--(0,0.25);
            \draw (-2.north)..controls +(up:10mm) and +(up:10mm)..(1.north);
            \draw (-3.north)..controls +(up:10mm) and +(up:10mm)..(3.north);
            \draw (-1.north)..controls +(up:10mm) and +(up:10mm)..(2.north);
            \end{tikzpicture}
    \end{center}
    Hence we must have $-k>-i$ in this case.
    \item Suppose the $(2n-1)^{th}$ and $(2n+1)^{th}$ letters are both $\beta$-letters of the same sign and their corresponding $\alpha$-letters are consecutive but the subscript $X$ of the $(2n-1)^{th}$ letter is less than the subscript $Y$ of the $(2n+1)^{th}$ letter. We can do the following to obtain a larger $\CT$ equivalent sketch:
    \begin{center}
        \begin{tikzpicture}[scale=0.8]
            \node (-2) at (-1.5,0) {$X$};
            \node (-1) at (-0.75,0) {$Y$};
            \node at (0,0) {$\cdots$};
            \node (1) at (0.75,0) {$X$};
            \node (3) at (1.5,0) {$-Y$};
            \node (a) at (3,0.75) {};
            \draw [blue](2,-0.25)--(2,0.25);
            \node (2) at (2.75,0) {$Y$};
            \draw (-2.north)..controls +(up:10mm) and +(up:10mm)..(1.north);
            \draw (3.north)..controls +(up:5mm) and +(left:5mm)..(a.north);
            \draw (-1.north)..controls +(up:10mm) and +(up:10mm)..(2.north);
            \draw (3.5,0.35) node {$\longrightarrow$};
        \end{tikzpicture}
        \begin{tikzpicture}[scale=0.8]
            \node (-2) at (-1.5,0) {$X$};
            \node (-1) at (-0.75,0) {$Y$};
            \node at (0,0) {$\cdots$};
            \node (1) at (0.75,0) {$X$};
            \node (3) at (1.5,0) {$Y$};
            \node (a) at (3.25,0.75) {};
            \draw [blue](2,-0.25)--(2,0.25);
            \node (2) at (2.75,0) {$-Y$};
            \draw (-2.north)..controls +(up:10mm) and +(up:10mm)..(1.north);
            \draw (2.north)..controls +(up:4mm) and +(left:4mm)..(a.north);
            \draw (-1.north)..controls +(up:10mm) and +(up:10mm)..(3.north);
            \draw (3.75,0.35) node {$\longrightarrow$};
        \end{tikzpicture}
        \begin{tikzpicture}[scale=0.8]
            \node (-2) at (-1.5,0) {$Y$};
            \node (-1) at (-0.75,0) {$X$};
            \node at (0,0) {$\cdots$};
            \node (1) at (0.75,0) {$Y$};
            \node (3) at (1.5,0) {$X$};
            \node (a) at (3.25,0.75) {};
            \draw [blue](2,-0.25)--(2,0.25);
            \node (2) at (2.75,0) {$-X$};
            \draw (-2.north)..controls +(up:10mm) and +(up:10mm)..(1.north);
            \draw (2.north)..controls +(up:4mm) and +(left:4mm)..(a.north);
            \draw (-1.north)..controls +(up:10mm) and +(up:10mm)..(3.north);
        \end{tikzpicture}
    \end{center}
\end{enumerate}

We now have to prove that these conditions are sufficient for a sketch to be $\CT$ maximal. 
Suppose $w$ is a symmetric sketch that satisfies the four properties mentioned in the statement of the theorem. 
Suppose there is a sketch $w'$ which is $\CT$ equivalent to $w$ but larger in the lexicographic order. 
This means that if $w=w_1\cdots w_{4n}$ and $w'=w'_1\cdots w'_{4n}$, there is some $p \in [4n]$ such that
\begin{equation*}
    w_i=w'_i\text{ for }i \in [p-1]\text{ and }w_p \prec w'_p.
\end{equation*}

The possible ways in which this can happen are listed below.
\begin{enumerate}
    \item $w_p$ is a $\beta_+$ letter and $w'_p$ is an $\alpha_+$ letter.
    \item $w_p$ is a $\beta_-$ letter and $w'_p$ is an $\alpha_-$ letter.
    \item $w_p$ is a $\beta_+$ letter and $w'_p$ is an $\alpha_-$ letter.
    \item $w_p$ is a $\beta_-$ letter and $w'_p$ is an $\alpha_+$ letter.
    \item $w_p$ and $w'_p$ are both $\alpha_+$ letters.
    \item $w_p$ and $w'_p$ are both $\alpha_-$ letters.
    \item $w_p$ is an $\alpha_-$ letter and $w'_p$ is an $\alpha_+$ letter.
\end{enumerate}

The case of both $w_p$ and $w'_p$ being $\beta$-letters is not possible since, by the properties of a sketch, this would mean $w_p = w'_p$. 
Since $\alpha_- \prec \alpha_+$ we cannot have $w_p$ being an $\alpha_+$ letter and $w'_p$ being an $\alpha_-$ letter. 
We will now show that each case leads to a contradiction, which will complete the proof of the theorem.

Before going forward, we formulate the meaning of $w$ and $w'$ being $\CT$ equivalent in terms of sketches. 
Since they have to be in the same region of $\CT_n$, the inequalities corresponding to the hyperplanes
\begin{equation*}
    x_i+x_j=-2,-1,0
\end{equation*}
for all $1\leq i<j \leq n$ are the same in both sketches. 
This means that the relationship between the pairs of the form
\begin{equation*}
    \{\lt{i}{1},\lt{-j}{-1}\},\ \{\lt{i}{1},\lt{-j}{0}\},\ \{\lt{i}{0},\lt{-j}{-1}\},\ \text{and }\{\lt{i}{0},\lt{-j}{0}\}
\end{equation*}
for any distinct $i, j \in [n]$ are the same in both $w$ and $w'$. 
This can be written as follows:
\begin{equation}\label{eqp}
\begin{aligned}
    &\text{The relationship between letters of opposite sign and}
    \\&\text{different mod value have to be the same in both $w$ and $w'$.}
\end{aligned}
\end{equation}

\begin{flushleft}
\underline{\textbf{Case 1:} $w_p$ is a $\beta_+$ letter and $w'_p$ is an $\alpha_+$ letter.}
\end{flushleft}

In this case $w$ and $w'$ are of the form
\begin{align*}
    w &= w_1 \cdots w_{p-1} \lt{k}{1} \cdots \\
    w' &= w'_1 \cdots w'_{p-1} \lt{l}{0} \cdots
\end{align*}
for some $k,l \in [n]$. 
Hence, \lt{l}{0} appears after \lt{k}{1} in $w$. 
By \eqref{eqp}, every letter between \lt{k}{1} and \lt{l}{0} in $w$ should be positive or one of \lt{-l}{-1} and \lt{-l}{0}. 
If all the letters are positive, since \lt{k}{1} is a $\beta_+$ letter and \lt{l}{0} is an $\alpha_+$ letter, there would be a consecutive pair of the form $\beta_+\alpha_+$ in $w$, which is a contradiction to property (1).

Now suppose \lt{-l}{0} is between \lt{k}{1} and \lt{l}{0} in $w$. 
It cannot be immediately before \lt{l}{0} since this would contradict property (1). 
But if it is not immediately before \lt{l}{0}, since \lt{-l}{0} and \lt{l}{0} are negatives of each other, there should be some negative letter between them. 
But this letter cannot be \lt{-l}{-1} (since this should be before \lt{-l}{0}). 
This is a contradiction to \eqref{eqp}. 
Hence \lt{-l}{0} cannot be between \lt{k}{1} and \lt{l}{0}.

So we must have \lt{-l}{-1} between \lt{k}{1} and \lt{l}{0} in $w$. 
Again, \lt{-l}{-1} cannot be immediately before \lt{l}{0} since this would contradict property (3). 
This means that there is at least one letter between \lt{-l}{-1} and \lt{l}{0} and all such letters are positive. 
If one of them is a $\beta_+$ letter, since \lt{l}{0} is an $\alpha_+$ letter, there would be a consecutive pair of the form $\beta_+\alpha_+$, which is a contradiction to property (1). 
Hence all the letters between \lt{-l}{-1} and \lt{l}{0} are $\alpha_+$ letters. 
But this is impossible by \Cref{ord}.

\begin{flushleft}
\underline{\textbf{Case 2:} $w_p$ is a $\beta_-$ letter and $w'_p$ is an $\alpha_-$ letter.}
\end{flushleft}

In this case $w$ and $w'$ are of the form
\begin{align*}
    w &= w_1 \cdots w_{p-1} \lt{-k}{0} \cdots \\
    w' &= w'_1 \cdots w'_{p-1} \lt{-l}{-1} \cdots
\end{align*}
for some $k,l \in [n]$. 
Hence, \lt{-l}{-1} appears after \lt{-k}{0} in $w$. 
By \eqref{eqp}, each letter between \lt{-k}{0} and \lt{-l}{-1} in $w$ has to be negative or one of \lt{l}{0} and \lt{l}{1}. 
Just as before, all letters between \lt{-k}{0} and \lt{-l}{-1} cannot be negative. 
The fact that \lt{l}{1} cannot be between \lt{-k}{0} and \lt{-l}{-1} also has a similar proof as in the last case.

So we must have \lt{l}{0} between \lt{-k}{0} and \lt{-l}{-1}. 
All the letters between \lt{l}{0} and \lt{-l}{-1} have to be negative. 
There are no $\beta_-$ letters between them, otherwise there would be consecutive letters of the form $\beta_-\alpha_-$, which contradicts property (1). 
So if there are letters between \lt{l}{0} and \lt{-l}{-1} they should all be $\alpha_-$ letters, but this cannot happen by \Cref{ord}. 
So \lt{l}{0} and \lt{-l}{-1} are consecutive. 
By property (3), the letter before \lt{l}{0} should be an $\alpha$-letter. 
And by \eqref{eqp}, it is an $\alpha_-$ letter. 
But since \lt{-k}{0} is a $\beta_-$ letter and all letters between \lt{-k}{0} and \lt{l}{0} are negative, there will be a consecutive pair of the form $\beta_-\alpha_-$, which is a contradiction to property (1).

\begin{flushleft}
\underline{\textbf{Case 3:} $w_p$ is a $\beta_+$ letter and $w'_p$ is an $\alpha_-$ letter.}
\end{flushleft}

In this case $w$ and $w'$ are of the form
\begin{align*}
    w &= w_1 \cdots w_{p-1} \lt{k}{1} \cdots \\
    w' &= w'_1 \cdots w'_{p-1} \lt{-l}{-1} \cdots
\end{align*}
for some $k,l \in [n]$. 
If $k \neq l$, this will contradict \eqref{eqp} since \lt{k}{1} will be before \lt{-l}{-1} in $w$ but not in $w'$. 
So \lt{-k}{-1} appears after \lt{k}{1} in $w$ and all letters between them are negative by \eqref{eqp} (note that \lt{k}{0} is before \lt{k}{1}). 
Again, \lt{-k}{-1} cannot be immediately after \lt{k}{1} since this would contradict property (1) and if there were some letters between \lt{k}{1} and \lt{-k}{-1}, at least one of them would be negative, which contradicts \eqref{eqp}.

\begin{flushleft}
\underline{\textbf{Case 4:} $w_p$ is a $\beta_-$ letter and $w'_p$ is an $\alpha_+$ letter.}
\end{flushleft}

Arriving at a contradiction in this case follows using the same method as in the last case.

\begin{flushleft}
\underline{\textbf{Case 5:} $w_p$ and $w'_p$ are both $\alpha_+$ letters.}
\end{flushleft}

In this case $w$ and $w'$ are of the form
\begin{align*}
    w &= w_1 \cdots w_{p-1} \lt{k}{0} \cdots \\
    w' &= w'_1 \cdots w'_{p-1} \lt{l}{0} \cdots
\end{align*}
for some $1 \leq k < l \leq n$. 
We split this case into two possibilities depending on whether or not \lt{l}{0} is before \lt{k}{1}.

\begin{flushleft}
\underline{\textbf{Case 5(a):} \lt{l}{0} is before \lt{k}{1} in $w$.}
\end{flushleft}

In this case $w$ and $w'$ are of the form
\begin{align*}
    w&=w_1 \cdots w_{p-1} \lt{k}{0} \cdots \lt{l}{0} \cdots \lt{k}{1} \cdots \lt{l}{1} \cdots \\
    w'&=w'_1 \cdots w'_{p-1} \lt{l}{0} \cdots.
\end{align*}
By \eqref{eqp}, each letter between \lt{k}{0} and \lt{l}{0} in $w$ is positive or one of \lt{-l}{-1} or \lt{-l}{0}. 
Just as in the \textbf{Case 1}, we can prove that \lt{-l}{-1} and \lt{-l}{0} cannot between \lt{k}{0} and \lt{l}{0}. 
Hence all the letters between \lt{k}{0} and \lt{l}{0} are positive. 
In fact, they all have to be $\alpha$-letters. 
Otherwise we would be a consecutive pair of the form $\beta_+\alpha_+$, which contradicts property (1).

Each letter between \lt{k}{1} and \lt{l}{1} is positive or one of \lt{-k}{-1}, \lt{-l}{-1}, \lt{-k}{0} or \lt{-l}{0}. 
Neither \lt{-k}{0} nor \lt{-l}{0} can be between \lt{k}{1} and \lt{l}{1}, since this would mean that \lt{-k}{-1} or \lt{-l}{-1} is between \lt{k}{0} and \lt{l}{0}, which cannot happen since we have already seen that there are only positive $\alpha$-letters between them.

If \lt{-k}{-1} were between \lt{k}{1} and \lt{l}{1}, it could not be immediately after \lt{k}{1} since this would contradict property (1). 
If there were some letters between \lt{k}{1} and \lt{-k}{-1}, at least one of them would be a negative letter other than \lt{-l}{-1}, which contradicts \eqref{eqp} (since \lt{l}{1} is after \lt{-k}{-1}).

So the only negative letter that can be between \lt{k}{1} and \lt{l}{1} is \lt{-l}{-1}. 
First, suppose that all letters between \lt{k}{1} and \lt{l}{1} are positive. 
Then all of them would have to be $\beta_+$ letters (otherwise there would be consecutive $\beta_+\alpha_+$ which contradicts property (1)). 
Then we would have that all letters between \lt{k}{0} and \lt{l}{0} are $\alpha_+$ letters and all letters between \lt{k}{1} and \lt{l}{1} are $\beta_+$ letters and repeated application of property (2) would give $k>l$, which is a contradiction.

Next, suppose \lt{-l}{-1} is between \lt{k}{1} and \lt{l}{1}. 
If \lt{-l}{-1} is not immediately before \lt{l}{1}, there will be some negative letter other than \lt{-l}{-1} between \lt{k}{1} and \lt{l}{1}, which we have already shown is not possible. 
So \lt{-l}{-1} is immediately before \lt{l}{1} and all the letters between \lt{k}{1} and \lt{-l}{-1} are positive and they have to all be $\beta_+$ letters (otherwise there would be a consecutive pair of the form $\beta_+\alpha_+$). 
If \lt{k'}{1} is the $\beta_+$ letter before \lt{-l}{-1} ($k'$ could be $k$), then \lt{k'}{0} is the letter before \lt{l}{0} and hence we get that the letters between \lt{k}{0} and \lt{k'}{0} are all $\alpha_+$ letters and their corresponding $\beta$-letters are consecutive and so by property (2), $k \geq k'$. 
But property (4) tells us that $k'>l$. 
So we get $k>l$, which is a contradiction.

\begin{flushleft}
\underline{\textbf{Case 5(b):} \lt{l}{0} is after \lt{k}{1} in $w$.}
\end{flushleft}

In this case $w$ and $w'$ are of the form
\begin{align*}
    w&=w_1 \cdots w_{p-1} \lt{k}{0} \cdots \lt{k}{1} \cdots \lt{l}{0} \cdots\\
    w'&=w'_1 \cdots w'_{p-1} \lt{l}{0} \cdots.
\end{align*}
By \eqref{eqp}, each letter between \lt{k}{0} and \lt{l}{0} in $w$ is positive or one of \lt{-l}{-1} or \lt{-l}{0}. 
Just as in \textbf{Case 1}, we can prove that \lt{-l}{-1} and \lt{-l}{0} cannot between \lt{k}{0} and \lt{l}{0}. 
Hence all the letters between \lt{k}{0} and \lt{l}{0} are positive. 
Since \lt{k}{1} is a $\beta_+$ letter and \lt{l}{0} is an $\alpha_+$ letter and all letters in between are positive, there is a consecutive pair of the form $\beta_+\alpha_+$, which is a contradiction to property (1).

\begin{flushleft}
\underline{\textbf{Case 6:} $w_p$ and $w'_p$ are both $\alpha_-$ letters.}
\end{flushleft}

In this case $w$ and $w'$ are of the form
\begin{align*}
    w &= w_1 \cdots w_{p-1} \lt{-k}{-1} \cdots \\
    w' &= w'_1 \cdots w'_{p-1} \lt{-l}{-1} \cdots
\end{align*}
for some $1 \leq l < k \leq n$. 
We split this case into two possibilities depending on whether or not \lt{-l}{-1} is before \lt{-k}{0}.

\begin{flushleft}
\underline{\textbf{Case 6(a):} \lt{-l}{-1} is before \lt{-k}{0} in $w$.}
\end{flushleft}

In this case $w$ and $w'$ are of the form
\begin{align*}
    w&=w_1 \cdots w_{p-1} \lt{-k}{-1} \cdots \lt{-l}{-1} \cdots \lt{-k}{0} \cdots \lt{-l}{0} \cdots \\
    w'&=w'_1 \cdots w'_{p-1} \lt{-l}{-1} \cdots.
\end{align*}
By \eqref{eqp}, each letter between \lt{-k}{-1} and \lt{-l}{-1} is negative or one of \lt{l}{0} or \lt{l}{1}. 
If \lt{l}{1} is between \lt{-k}{-1} and \lt{-l}{-1}, it should not be immediately before \lt{-l}{-1} since this would contradict property (1). 
But then there would be some positive letter other than \lt{l}{0} between \lt{-l}{-1} and \lt{l}{1} which would contradict \eqref{eqp}.

First, suppose \lt{l}{0} is between \lt{-k}{-1} and \lt{-l}{-1}. 
Just as before, using property (1) and \Cref{ord}, we can show that \lt{l}{0} has to be immediately before \lt{-l}{-1}. 
Also, all the letters between \lt{-k}{-1} and \lt{l}{0} have to be negative by \eqref{eqp}. 
By property (3), the letter before \lt{l}{0} has to be an $\alpha$-letter and hence here it is an $\alpha_-$ letter. 
Hence, the letters between \lt{-k}{-1} and \lt{l}{0} have to be $\alpha_-$ letters since otherwise there be a consecutive pair of the form $\beta_-\alpha_-$.

By \eqref{eqp}, each letter between \lt{-k}{0} and \lt{-l}{0} is negative or one of \lt{k}{0}, \lt{l}{0}, \lt{k}{1} or \lt{l}{1}. 
Now, \lt{k}{1} cannot be between \lt{-k}{0} and \lt{-l}{0} since this would mean \lt{k}{0} is between \lt{-k}{-1} and \lt{-l}{-1}, which we have already shown is not possible. 
We have already assumed \lt{l}{0} is between \lt{-k}{-1} and \lt{-l}{-1} and hence it cannot also be between \lt{-k}{0} and \lt{-l}{0}. 
If \lt{k}{0} were between \lt{-k}{0} and \lt{-l}{0}, it could not have been immediately after \lt{-k}{0} since this would contradict property (1). 
But then there would be some positive letter other than \lt{l}{1} between \lt{-k}{0} and \lt{k}{0} (since \lt{-l}{-1} is before \lt{-k}{0} and hence \lt{l}{1} is after \lt{k}{0}), which is a contradiction to \eqref{eqp}. 
This means that the only positive letter between \lt{-k}{0} and \lt{-l}{0} is \lt{l}{1} which is between them since \lt{l}{0} is between \lt{-k}{-1} and \lt{-l}{-1}. 
Since \lt{l}{0} and \lt{-l}{-1} are consecutive, so are \lt{l}{1} and \lt{-l}{0}. 
The letters between \lt{-k}{0} and \lt{l}{1} are all negative and should be $\beta_-$ letters or else it would cause a contradiction to property (1).

Hence, the situation in the case that \lt{l}{0} is between \lt{-k}{-1} and \lt{-l}{-1} is the following: 
There is a string of consecutive $\alpha_-$ letters starting with \lt{-k}{-1} ending before \lt{l}{0} which is immediately before \lt{-l}{-1} and the corresponding $\beta$-letters for all these $\alpha$-letters are consecutive. 
If \lt{-k'}{-1} is the $\alpha_-$ letter immediately before \lt{l}{0} ($k'$ could be $k$), then property (3) gives that $-k'>-l$ and property (2) gives that $-k \geq -k'$ and hence we get $-k>-l$, which is a contradiction.

Next, suppose that all the letters between \lt{-k}{-1} and \lt{-l}{-1} are negative. 
All of them should be $\alpha_-$ letters by property (1). 
It can be shown, just as before, that the only possible positive letter between \lt{-k}{0} and \lt{-l}{0} is \lt{l}{0}. 
If \lt{l}{0} is not between \lt{-k}{0} and \lt{-l}{0}, property (2) leads to a contradiction just as in \textbf{Case 5(a)}. 
If \lt{l}{0} is between \lt{-k}{0} and \lt{-l}{0}, it should be immediately before \lt{-l}{0} and again, following a method similar to \textbf{Case 5(a)}, this leads to a contradiction using property (4).

\begin{flushleft}
\underline{\textbf{Case 6(b):} \lt{-l}{-1} is after \lt{-k}{0} in $w$.}
\end{flushleft}

In this case $w$ and $w'$ are of the form
\begin{align*}
    w&=w_1 \cdots w_{p-1} \lt{-k}{-1} \cdots \lt{-k}{0} \cdots \lt{-l}{-1} \cdots \lt{-l}{0} \cdots \\
    w'&=w'_1 \cdots w'_{p-1} \lt{-l}{-1} \cdots.
\end{align*}
By \eqref{eqp}, each letter between \lt{-k}{-1} and \lt{-l}{-1} is negative or one of \lt{l}{0} or \lt{l}{1}. 
Just as before, \lt{l}{1} cannot be between \lt{-k}{-1} and \lt{-l}{-1}. 
If \lt{l}{0} is not between \lt{-k}{-1} and \lt{-l}{-1}, then all the letters between them are negative and there is a $\beta_-$ letter, namely \lt{-k}{0}, between them and this would result in a consecutive pair of the form $\beta_-\alpha_-$, which contradicts property (1).

So \lt{l}{0} is the only positive letter between \lt{-k}{-1} and \lt{-l}{-1}. 
If \lt{l}{0} is before \lt{-k}{0}, we would get a consecutive pair of the form $\beta_-\alpha_-$ between \lt{-k}{0} and \lt{-l}{-1} which contradicts property (1). 
So \lt{l}{0} is between \lt{-k}{0} and \lt{-l}{-1}. 
If \lt{l}{0} and \lt{-l}{-1} were not consecutive, we would get a contradiction to property (1) if there were some $\beta_-$ letter between them and if all were $\alpha_-$ letters, this would contradict \Cref{ord}. 
So \lt{l}{0} and \lt{-l}{-1} are consecutive, and by property (3), the letter before \lt{l}{0} should be an $\alpha$-letter and in this case an $\alpha_-$ letter, say \lt{-k'}{-1}. 
But then we would get a consecutive pair of the form $\beta_-\alpha_-$ between \lt{-k}{0} and \lt{-k'}{-1} which contradicts property (1).

\begin{flushleft}
\underline{\textbf{Case 7:} $w_p$ is a $\alpha_-$ letter and $w'_p$ is an $\alpha_+$ letter.}
\end{flushleft}

In this case $w$ and $w'$ are of the form
\begin{align*}
    w&=w_1 \cdots w_{p-1} \lt{-k}{-1} \cdots \\
    w'&=w'_1 \cdots w'_{p-1} \lt{l}{0} \cdots
\end{align*}
for some $k,l \in [n]$. 
If $k\neq l$, we would get a contradiction to \eqref{eqp} since \lt{-k}{-1} is before \lt{l}{0} is $w$ but not in $w'$. 
So \lt{k}{0} appears after \lt{-k}{-1} in $w$ and each letter between them is positive or \lt{-k}{0}. 
Just as before \lt{-k}{0} being between \lt{-k}{-1} and \lt{k}{0} would either contradict property (1) or \eqref{eqp}. 
So all letters between \lt{-k}{-1} and \lt{k}{0} are positive. 
If there is some $\beta_+$ letter between them, there will be a consecutive pair of the form $\beta_+\alpha_+$, which would contradict property (1). 
Hence, all letters between \lt{-k}{-1} and \lt{k}{0} are $\alpha_+$ letters. But this contradicts \Cref{ord}.
\end{proof}

\subsubsection{Shi threshold}

The Shi threshold arrangement in $\mathbb{R}^n$ consists of the hyperplanes
\begin{equation*}
    X_i+X_j=0,1
\end{equation*}
for all $1 \leq i < j\leq n$. 
The translated arrangement by setting $X_i=x_i + \frac{1}{2}$, which we call $\ST_n$, has hyperplanes
\begin{equation*}
    x_i+x_j=-1,0
\end{equation*}
for all $1 \leq i < j \leq n$. 
We use the same method as before to study the regions of this arrangement by considering $\ST_n$ as a sub-arrangement of $\C_n$.
% Using Bernardi's idea of moves, we can define an equivalence on the sketches such that two sketches are equivalent if they lie in the same Shi threshold region. Then set of sketches which are the maximum in their equivalence class are in bijection with the Shi threshold regions.

% Two sketches lie in the same regions of $\ST_n$ if and only if they are on the same side of the Shi threshold hyperplanes. 
The hypeplanes in $\C_n$ that are not in $\ST_n$ are
\begin{align*}
    2x_i&=-2,-1,0\\
    x_i+x_j&=-2\\
    x_i-x_j&=-1,0,1
\end{align*}
where $1 \leq i < j \leq n$. 
Changing the inequality corresponding to exactly one of these hyperplanes are given by the $\CT$ moves as well as the move corresponding to $x_i+x_j=-2$ where $i \neq j$ are in $[n]$: Swapping consecutive $\beta_+$ and $\alpha_-$ letters (along with their negatives).

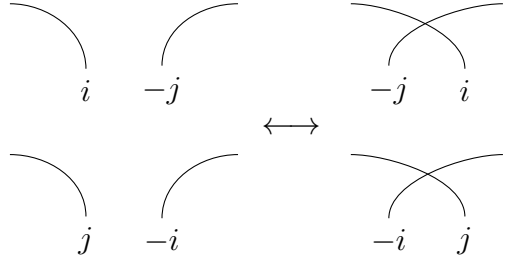
\begin{center}
            \begin{tikzpicture}
            \node (-1) at (-0.5,0) {$i$};
            \node (1) at (0.5,0) {$-j$};
            \node (a) at (-1.5,1) {};
            \node (b) at (1.5,1) {};
            \draw (-1.north)..controls +(up:5mm) and +(right:5mm)..(a.north);
            \draw (1.north)..controls +(up:5mm) and +(left:5mm)..(b.north);
            \node (-1') at (-0.5,-2) {$j$};
            \node (1') at (0.5,-2) {$-i$};
            \node (a') at (-1.5,-1) {};
            \node (b') at (1.5,-1) {};
            \draw (-1'.north)..controls +(up:5mm) and +(right:5mm)..(a'.north);
            \draw (1'.north)..controls +(up:5mm) and +(left:5mm)..(b'.north);
            \node at (2.2,-0.5) {$\longleftrightarrow$};
        \end{tikzpicture}
        \begin{tikzpicture}
            \node (-1) at (-0.5,0) {$-j$};
            \node (1) at (0.5,0) {$i$};
            \node (a) at (-1,1) {};
            \node (b) at (1,1) {};
            \draw (1.north)..controls +(up:5mm) and +(right:5mm)..(a.north);
            \draw (-1.north)..controls +(up:5mm) and +(left:5mm)..(b.north);
            \node (-1') at (-0.5,-2) {$-i$};
            \node (1') at (0.5,-2) {$j$};
            \node (a') at (-1,-1) {};
            \node (b') at (1,-1) {};
            \draw (1'.north)..controls +(up:5mm) and +(right:5mm)..(a'.north);
            \draw (-1'.north)..controls +(up:5mm) and +(left:5mm)..(b'.north);
        \end{tikzpicture}
\end{center}

Two sketches are in the same region of $\ST_n$ if and only if they are related by a series of such moves and we call such sketches $\ST$ equivalent. 
A sketch is called $\ST$ maximal if it is greater (in the lexicographic order) than all sketches to which it is $\ST$ equivalent. 
Hence the regions of $\ST_n$ are in bijection with the $\ST$ maximal sketches.
The following result can be proved just as \Cref{catthresh}.

\begin{theorem}
    A symmetric sketch is $\ST$ maximal if and only if the following hold.
    \begin{enumerate}
        \item If a $\beta$-letter is followed by an $\alpha$-letter, the $\beta$-letter should be negative and the $\alpha$-letter should be positive with different mod-values.
        
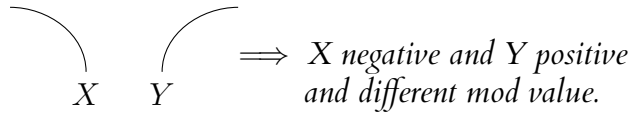
\begin{center}
            \begin{tikzpicture}
            \node (-1) at (-0.5,0) {$X$};
            \node (1) at (0.5,0) {$Y$};
            \node (a) at (-1.5,1) {};
            \node (b) at (1.5,1) {};
            \draw (-1.north)..controls +(up:5mm) and +(right:5mm)..(a.north);
            \draw (1.north)..controls +(up:5mm) and +(left:5mm)..(b.north);
            \node at (4,0.5) {$\implies \text{X negative and Y positive}$};
            \node at (4.3,0) {and different mod value.};
            \end{tikzpicture}
        \end{center}
        
        \item If two $\alpha$-letters and their corresponding $\beta$-letters are both consecutive and of the same sign then the subscript of the first one should be greater.
        
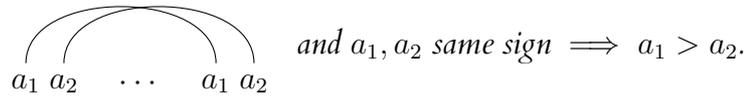
\begin{center}
            \begin{tikzpicture}
            \node (-2) at (-1.5,0) {$a_1$};
            \node (-1) at (-1,0) {$a_2$};
            \node at (0,0) {$\cdots$};
            \node (1) at (1,0) {$a_1$};
            \node (2) at (1.5,0) {$a_2$};
            \draw (-2.north)..controls +(up:10mm) and +(up:10mm)..(1.north);
            \draw (-1.north)..controls +(up:10mm) and +(up:10mm)..(2.north);
            \node at (5,0.5) {and $a_1,a_2$ same sign $\implies a_1 > a_2$.};
            \end{tikzpicture}
        \end{center}
        
        \item If the $n^{th}$ and $(n+1)^{th}$ $\alpha$-letters are consecutive, then so are the $(n-1)^{th}$ and $n^{th}$ with the $n^{th}$ $\alpha$-letter being positive. In such a situation, if the $(n-1)^{th}$ $\alpha$-letter is negative and the $(n-1)^{th}$ and $n^{th}$ $\beta$-letters are consecutive, the $(n-1)^{th}$ $\alpha$-letter should have a subscript greater than that of the $(n+1)^{th}$ $\alpha$-letter.
        
        \item If the $(2n-1)^{th}$ and $(2n+1)^{th}$ letters are both negative $\beta$-letters and their corresponding $\alpha$-letters are consecutive, the subscript of the $(2n-1)^{th}$ letter should be greater than that of the $(2n+1)^{th}$.
        
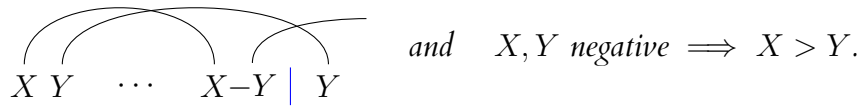
\begin{center}
        \begin{tikzpicture}
            \node (-2) at (-1.5,0) {$X$};
            \node (-1) at (-1,0) {$Y$};
            \node at (0,0) {$\cdots$};
            \node (1) at (1,0) {$X$};
            \node (3) at (1.5,0) {$-Y$};
            \node (a) at (3,0.75) {};
            \draw [blue](2,-0.25)--(2,0.25);
            \node (2) at (2.5,0) {$Y$};
            \draw (-2.north)..controls +(up:10mm) and +(up:10mm)..(1.north);
            \draw (3.north)..controls +(up:5mm) and +(left:5mm)..(a.north);
            \draw (-1.north)..controls +(up:10mm) and +(up:10mm)..(2.north);
            \node at (6.5,0.5) {and $\quad X,Y$ negative $\implies X>Y$.};
            \end{tikzpicture}
        \end{center}
    \end{enumerate}
    Hence the regions of $\ST_n$ are in bijection with sketches of the form described above.
\end{theorem}

\section{Concluding remarks}\label{concrem}
We end the paper with some open questions. 
Bernardi \cite{ber} has dealt with arbitrary deformations of the braid arrangement. 
The first (ambitious) problem is to generalize all the results in his paper to arbitrary deformations of all reflection arrangements. 
This is easier said then done! 
Bernardi proves that the number of regions is equal to the signed sum of certain ``boxed trees''. 
So the first step is to generalize the notion of boxed trees to certain decorated forests and then prove the counting formula, this is a work in progress. 
For certain well-behaved arrangements called ``transitive deformations'' Bernardi establishes an explicit bijection between the regions and the corresponding trees, via sketches.
We don't have trees for all deformations of reflection arrangements but, we do have sketches that are in bijection with regions of (extended) Catalan deformations. 
%Bernardi also expresses the exponential generating function for the characteristic (and the coboundary) polynomial in terms of trees. 
%We do have partial results in this direction, where we use the finite field method to get an expression for the exponential generating function. 

The main motivation behind Bernardi's work is an interesting pattern concerning certain statistic on labeled binary trees. 
Ira Gessel observed that the multivariate generating function for this statistic specializes to region counts of certain deformations of the braid arrangement. 
%However, this statistic is not readily translated in the language of sketches. 
So a new research direction could be to try and define a statistic on non-nesting partitions (of all types) such that the associated generating function specializes to region counts. 

Another aspect of Bernardi's work that has not been discussed in the present paper is the coboundary and Tutte polynomials. 
Using either, the finite field method or the method inspired by statistical mechanics one should get a closed form expression for these polynomials of the deformations we have considered. 
Moreover, the expression should be in terms of either sketches or non-nesting partitions. 

Having a combinatorial model for the coefficients of the characteristic polynomial could be quite useful. 
Especially to derive various inequalities that they satisfy. 
For example, denote by $C(m, n, j)$ be the number of symmetric $m$-non-nesting partitions of size $n$ with $j$ positive compartments.
Then following inequalities are not difficult to prove:
\begin{enumerate}
    \item $C(m, n, j) \leq C(m, n + 1, j)$
    \item $C(m, n, j) \leq C(m, n + 1, j + 1)$
    \item $C(m, n, j) \geq \sum_{k \geq j + 1} \binom{k}{j} C(m, n, k)$. 
    %In particular, $C(m, n, j) \geq C(m, n, j + 1)$.
\end{enumerate}
A research direction here is to develop a case-free strategy to obtain more such information. 
For example, we know that the coefficients are unimodal so, identify the peak in each case. 
%$C(m, n, j) \geq C(m + 1, n, j)$ is probably true but have to find some way to increase size of each block of partition and maintain symmetry.
%Maybe check how adding leaf as rightmost child of corresponding tree affects the partition.

Recall the Raney numbers that are defined by
\begin{equation*}
    A_n(m,r) := \frac{r}{n(m+1)+r}\binom{n(m+1)+r}{n}
\end{equation*}
for all positive integers $n,m,r$. 
The Catalan numbers are a special case of Raney numbers, obtained by setting $m=n=1$.
It was shown in \cite{dmwfuss23} that the number of regions of the hyperplane arrangement
\begin{equation*}\label{arrdef}
    \{x_i=0 \mid i \in [n]\} \cup \{x_i=2^kx_j \mid k \in [-m,m], 1\leq i<j \leq n\}
\end{equation*}
is equal to $n! A_n(m, 2)$. 
Note that these arrangements define a GESA. 
Find a family of arrangements which is GESA and the number of regions is $n! A_n(m, r)$. 
One can use tuples of labeled Dyck paths to enumerate these regions. 
So one can try and apply techniques from this paper to find a static for these objects.

\section{Acknowledgements}

The authors are partially supported by a grant from the Infosys Foundation. 
The computer algebra system SageMath \cite{sage} provided valuable assistance in studying examples.

\bibliographystyle{abbrv} % bibliography style
\bibliography{refs} % References file

\end{document}